\documentclass[10pt]{amsart} % main
\pdfoutput=1 % for arXiv
\usepackage{amsmath}
\usepackage{amsthm}
\usepackage{amsfonts}
\usepackage{amssymb}

\usepackage{ifpdf}
\ifpdf
    \usepackage[colorlinks]{hyperref}
    \hypersetup{urlcolor=blue,linkcolor=blue,citecolor=blue}
\fi
\usepackage{microtype} % pdftex microtype features

\usepackage{paralist}

\usepackage{graphicx}

\usepackage{color}

\ifdefined\draftversion
    \providecommand{\showkeys}{}
    \usepackage[shortalphabetic,initials,backrefs,msc-links]{amsrefs}
\else
    \usepackage[initials,msc-links]{amsrefs}
\fi

\ifdefined\showkeys
	\usepackage[color]{showkeys} % show labels in the pdf file, in grey color
\fi

\makeatletter
\def\blfootnote{\xdef\@thefnmark{}\@footnotetext}
\makeatother

\IfFileExists{.git/git.tex}{
\input{.git/git.tex}

\ifdefined\gitaddpagehash
  \usepackage{fancyhdr}
  \pagestyle{fancy}
  
  \lhead{} \rhead{}
  \rfoot{\tiny\gitshash}
\fi

}
{

}

\graphicspath{{}}
\ifpdf
  \newcommand{\fig}[2]{
    \IfFileExists{#1.pdf_tex}{
      \def\svgwidth{#2}\input{#1.pdf_tex}
    }{
      \frame{Missing figure ``#1.pdf''}
      \message{LaTeX Warning: Missing figure ``#1.pdf'' on input line \the\inputlineno}
    }
  }
\else
  \newcommand{\fig}[2]{\frame{PDF Figure here}}	% if the target is not .pdf, don't include the figure
\fi

\newcommand{\dimension}{n}

\newcommand{\dist}{\operatorname{dist}}
\newcommand{\diam}{\operatorname{diam}}

\newcommand{\interior}{\operatorname{int}}
\newcommand{\trace}{\operatorname{tr}}

\newcommand{\argmin}{\operatorname*{arg\,min}}

\newcommand{\lap}{\Delta}

\newcommand{\dx}{\;dx}
\newcommand{\dy}{\;dy}
\newcommand{\diff}[1]{\;d{#1}}

\newcommand{\ov}[1]{\frac{1}{#1}}

\newcommand{\abs}[1]{\left|#1\right|}
\newcommand{\pth}[1]{\left(#1\right)}
\newcommand{\bra}[1]{\left[#1\right]}
\newcommand{\set}[1]{{\left\{#1\right\}}}
\newcommand{\clset}[1]{\cl{\set{#1}}}
\newcommand{\pset}[1]{\partial\set{#1}}
\newcommand{\at}[2]{{{\left.{#1}\right|}_{#2}}}

\newcommand{\norm}[1]{\left\|#1\right\|}

\newcommand{\cl}[1]{\overline{#1}}	% closure

\newcommand{\al}{\ensuremath{\alpha}}
\newcommand{\be}{\ensuremath{\beta}}
\newcommand{\ga}{\ensuremath{\gamma}}
\newcommand{\de}{\ensuremath{\delta}}
\newcommand{\e}{\ensuremath{\varepsilon}}
\newcommand{\vp}{\ensuremath{\varphi}}
\newcommand{\la}{\ensuremath{\lambda}}
\newcommand{\si}{\ensuremath{\sigma}}
\newcommand{\ta}{\ensuremath{\theta}}

\newcommand{\om}{\ensuremath{\omega}}

\newcommand{\R}{\ensuremath{\mathbb{R}}}

\newcommand{\Rd}{\ensuremath{{\mathbb{R}^{\dimension}}}}
\newcommand{\Rn}{\Rd}
\newcommand{\Z}{\ensuremath{\mathbb{Z}}}
\newcommand{\N}{\ensuremath{\mathbb{N}}}

\newcommand{\halflimsup}{\operatorname*{\star-limsup}}
\newcommand{\halfliminf}{\operatorname*{\star-liminf}}

\newcommand{\pd}[2]{\frac{\partial {#1}}{\partial {#2}}}

\newcommand{\td}[2]{\frac{d {#1}}{d {#2}}}

\definecolor{grey}{rgb}{0.6,0.6,0.6}

\renewcommand{\labelenumi}{(\alph{enumi})}

\newcommand{\romanlist}{\renewcommand{\labelenumi}{\textup{(}\roman{enumi}\textup{)}}}

\numberwithin{equation}{section}

\newtheorem{theorem}{Theorem}[section]
\newtheorem{lemma}[theorem]{Lemma}
\newtheorem{proposition}[theorem]{Proposition}
\newtheorem{corollary}[theorem]{Corollary}

\newtheorem{definition}[theorem]{Definition}

\theoremstyle{definition}

\newtheoremstyle{remarkstyle}% name of the style to be used
  {3pt}% measure of space to leave above the theorem. E.g.: 3pt
  {3pt}% measure of space to leave below the theorem. E.g.: 3pt
  {\small}% name of font to use in the body of the theorem
  {}% measure of space to indent
  {\bfseries}% name of head font
  {.}% punctuation between head and body
  { }% space after theorem head; " " = normal interword space
  {}% Manually specify head

\theoremstyle{remarkstyle}
\newtheorem{remark}[theorem]{Remark}

\newcommand{\grid}{{::}}
\newcommand{\gride}{{::\e}}
\newcommand{\ou}{\overline{u}}
\newcommand{\uu}{\underline{u}} % upper solution
\newcommand{\lu}{\overline{u}}  % lower solution
\newcommand{\ur}{\underline{r}} % upper velocity
\newcommand{\lr}{\overline{r}}  % lower velocity
\newcommand{\lPhi}{\overline\Phi}
\newcommand{\uPhi}{\underline\Phi}

\newcommand{\eqr}{{\e;q,r}}

\newcommand{\Pqr}{P_{q,r}}
\newcommand{\supers}{\overline{\mathcal{S}}}
\newcommand{\subs}{\underline{\mathcal{S}}}
\newcommand{\sol}{\mathcal{S}}
\newcommand{\cone}{{\mathrm{Cone}}}

\newcommand{\parahead}[1]{\bigskip\noindent\textbf{{#1}.}}

\newcommand{\ha}[1]{\frac{#1}{2}}

\newcommand{\Span}{\operatorname{span}}

\title[Homogenization of the Hele-Shaw problem]{Homogenization of the Hele-Shaw problem in periodic spatiotemporal media}
\author[N. Po\v{z}\'{a}r]{Norbert Po\v{z}\'{a}r}
\address{
University of Tokyo, Graduate School of Mathematical Sciences;
Current address: Kanazawa University, Faculty of Mathematics and Physics, Institute of Science and Engineering, 920-1192 Ishikawa, Kanazawa city, Kakuma town}
\email{npozar@se.kanazawa-u.ac.jp}

\date{}


\begin{document}
\begin{abstract}
% abstract begin
We consider the homogenization of the Hele-Shaw problem
in periodic media that are inhomogeneous both in space and time.
After extending the theory of viscosity solutions into this context,
we show that the solutions of the inhomogeneous problem converge
in the homogenization
limit to the solution of a homogeneous Hele-Shaw-type problem
with a general, possibly nonlinear dependence of the free boundary velocity
on the gradient.
Moreover, the free boundaries converge locally uniformly
in Hausdorff distance.
% abstract end
\end{abstract}
\maketitle
\blfootnote{The final publication is available at Springer via \url{http://dx.doi.org/10.1007/s00205-014-0831-0}}

\tableofcontents

% CONTENTS BEGIN

\section{Introduction}
\label{sec:introduction}

Let $n \geq 2$ be the dimension
and let $\Omega \subsetneq \Rn$
be a domain (open, connected) with
a non-empty compact Lipschitz boundary
($\partial \Omega \in C^{0,1}$).
For given $\e > 0$,
we shall consider the following Hele-Shaw-type problem
on a parabolic cylinder $Q := \Omega \times (0, T]$ for some $T > 0$:
find $u^\e: Q \to [0, \infty)$ that formally satisfies
\begin{equation}
\label{HSt}
\begin{cases}
- \Delta u(x,t) = 0 & \text{in $\set{u > 0}$},\\
V_\nu(x,t) = g\pth{\frac{x}{\e},\frac{t}{\e}} \abs{Du^+(x,t)}
    & \text{on $\partial \set{u>0}$},
\end{cases}
\end{equation}
with some boundary data to be specified later
(Theorem~\ref{th:well-posedness}),
\begin{figure}
\fig{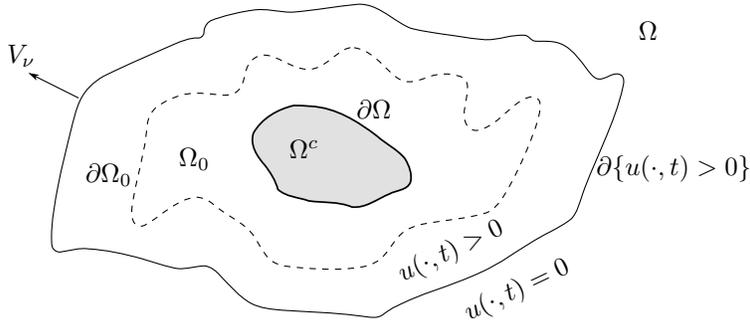}{4.5in}
\caption{The Hele-Shaw problem in a plane}
\label{fig:setting}
\end{figure}
where $D = D_x$ and $\Delta = \Delta_x$ are respectively the gradient and  the Laplace operator in the $x$ variable,
and $V_\nu$ is the normal velocity of
the \emph{free boundary} $\partial \set{u > 0}$.
Assuming that the solution is sufficiently smooth,
the free boundary $\partial \set{u > 0}$ is a level set of $u$
and its normal velocity can be expressed as
\begin{align*}
V_\nu = \frac{u_t^+}{\abs{Du^+}}.
\end{align*}
Here $u_t^+$ and $Du^+$ formally represent
the limits of the derivatives from the positive set of $u$.
The well-posedness of problem \eqref{HSt}
requires some basic regularity of $g$.
In the following we shall assume that $g$ satisfies
\begin{align}
\label{g-bound}
0 < m \leq g \leq M
\end{align}
for some positive constants $m, M$,
and that $g$ is $L$-Lipschitz both in $x$ and $t$, that is,
\begin{align}
\label{g-Lipschitz}
\abs{g(x,t) - g(y,s)} \leq L(\abs{x-y} + \abs{t-s})
\qquad
\text{for all $(x,t), (y,s) \in \Rn\times\R$.}
\end{align}

We are interested in the behavior of the solutions of \eqref{HSt}
in the \emph{homogenization limit} $\e \to 0+$.
Since we want to observe an averaging behavior,
we further assume that $g$ is a $\Z^{n+1}$-periodic function, i.e.,
\begin{align}
\label{g-periodic}
g(x + k, t +l) &= g(x, t)
    && \text{for all $(x,t) \in \Rn \times \R,\ (k,l) \in \Z^n \times \Z$.}
\end{align}
In the following we define
\begin{align*}
g^\e(x,t) := g\pth{\frac{x}{\e}, \frac{t}{\e}}.
\end{align*}
Note that $g^\e$ is an $\e \Z^{n+1}$-periodic $\frac{L}{\e}$-Lipschitz function
which satisfies \eqref{g-bound}.

Problem \eqref{HSt} with $g \equiv const$ is the standard
Hele-Shaw problem with no surface tension.
In two dimensions, it was introduced in \cite{HS} as a model
of a slow movement of a viscous fluid injected
in between two nearby parallel plates
that form the so-called Hele-Shaw cell.
This problem
naturally generalizes to all dimensions $n \geq 1$.
In particular, in three dimensions it serves as a model
of a pressure-driven flow of an incompressible fluid through a porous medium.
Following this motivation,
problem \eqref{HSt} with general $g$ describes a pressure-driven flow of
an incompressible fluid in an inhomogeneous, time-dependent medium.
Free boundary problems with similar velocity laws
have various applications in the plastics industry
and other fields \cites{PR,Richardson,Steinbach}.
The Hele-Shaw-type problem \eqref{HSt}
can also be viewed as a quasi-stationary limit
of the one-phase Stefan problem with a latent heat of phase transition
depending on position and time
\cite{AGM,LR,Primicerio,Rou}.

%\parahead{Previous results}
\subsection*{Homogenization overview}
There is a large amount of literature on homogenization
which is beyond the scope of this discussion
and thus we refer the reader to \cite{JKO,Tartar}
and the references therein.

In the context of viscosity solutions for fully nonlinear problems
of the first and second order,
the now standard approach to homogenization
using correctors that are the solutions of an appropriate cell problem
was pioneered by Lions, Papanicolau \& Varadhan \cite{LPV}
for first order equations,
and later by Evans \cite{Evans92} for second order equations.
Unfortunately, this approach does not apply to problems
like \eqref{HSt} because the zero level of solutions has a special
significance and the perturbation of a test function
by a global periodic solution
of some cell problem, i.e., the corrector, does not seem feasible.

The idea of using obstacle problems to recover the homogenized operator
without the need for a cell problem
was developed by Caffarelli, Souganidis \& Wang \cite{CSW}
for the stochastic homogenization of fully-nonlinear second order
elliptic equations.
It was then applied to the homogenization of the Hele-Shaw problem
with spatial periodic inhomogeneity
by Kim \cites{K07},
using purely the methods of the theory of viscosity solutions.
This approach was later extended to a model of contact angle dynamics \cite{K08},
and an algebraic rate of the convergence of free boundaries
was obtained \cite{K09}.

Kim \& Mellet \cites{KM09,KM10} later succeeded at applying
a combination of viscosity and variational approaches
and obtained a homogenization result in the setting of
spatial stationary ergodic random media.
A related question of long-time asymptotics of
the spatially inhomogeneous Hele-Shaw problem was addressed by the author
\cite{P11}.
This technique relies on the special structure of
problem \eqref{HSt} with time-independent $g$
which allows us to rewrite
the problem as a variational inequality of a certain obstacle problem.
We refer the reader to \cite{Rodrigues} and the references therein
for an exposition of obstacle problems and their homogenization.

The homogenization results for spatial media do not necessarily
translate directly to the spatiotemporal homogenization.
In the context of Hamilton-Jacobi equations, for instance,
a lack of uniform estimates in the spatiotemporal case was
encountered by Schwab \cite{Schwab} who
proved the homogenization
in spatiotemporal stationary ergodic random media, which was
before established by Souganidis \cite{Soug99} for spatial random media.
The situation seems to be more extreme in the case of the Hele-Shaw problem since the homogenization in spatiotemporal media
is qualitatively different even in the periodic case.

This difference can be already observed in the homogenization of a single
ordinary differential equation
of the type $x_\e'(t) = f(x_\e(t)/\e, t/\e, x_\e(t))$.
In fact, the Hele-Shaw problem \eqref{HSt} reduces
to this type of ODE in one dimension; see Section~\ref{sec:homogenizationin1D}
for further discussion.
It is known that, under some assumptions on $f$, $x_\e \rightrightarrows \bar x$ locally uniformly,
where $\bar x$ is the solution of a homogenized ODE
\cite{Piccinini1,IM10}.
However, the form of the homogenized problem depends on $f$.
If, for illustration, $f(y,s,x) = f(y)$ is a $\Z$-periodic
Lipschitz function,
then the homogenized problem has the form $\bar x'(t) = \bar f$
where $\bar f$ is the constant
\begin{equation}
\label{hom-const}
\bar f = \pth{\int_0^1 (f(y))^{-1} \dy}^{-1}.
\end{equation}
On the other hand,
if $f(y,s,x) = f(y,s)$ has a nontrivial dependence on both $y$ and $s$,
then no such explicit formula exists and, in fact, the right-hand side
of the homogenized problem, while still a constant,
can possibly attain any value in the range $[\min f, \max f]$.

Another explanation of the qualitative difference,
slightly more intuitive, is the interpretation
of the Hele-Shaw problem \eqref{HSt} as the quasi-stationary
limit of the Stefan problem.
Then the quantity $(g(x/\e,t/\e))^{-1}$ can be interpreted
as the \emph{latent heat of phase transition}, that is,
the ``energy'' necessary
to change a unit volume of the dry region into the wet region
and advance the free boundary \cite{P11}.
The ``energy flux'' is proportional to the gradient of the pressure.
If $g$ does not depend on time, the homogenized latent heat
is simply the average, i.e., $\int_{[0,1]^n} (g(x))^{-1} \dx$,
giving the homogenized velocity recovered in \cite{KM09}
(of the form \eqref{hom-const}).

This formula was rigorously justified in \cite{KM09}
using the variational formulation via an obstacle
problem,
which allows one to solve for the shape of the wet region
at a given time without the need to solve the problem
at previous times.
The variational formulation was introduced in \cite{EJ}
using a transformation due to \cite{B72}.
In a sense, for the evolution overall,
the free boundary feels the same influence of the medium no matter how it passes through it. Or, in other words, the amount of the energy required
to fill a given region is always the same.

This drastically changes when the latent heat depends on
both space and time.
In this case, the energy required to advance the boundary depends
on the specific history of the motion of the free boundary through the space-time.
The variational formulation does not apply anymore.
What is more, it is no longer obvious that the problem should homogenize.
As the one-dimensional situation indicates---Section~\ref{sec:homogenizationin1D}---the homogenized velocity
might have a complicated dependence on the gradient,
with velocity pinning and directional dependence.
Some of these features appear in the spatial homogenization
of non-monotone problems \cite{K08,K09}.

Our approach can be characterized as geometric, relying on
maximum principle arguments.
A similar approach to homogenization,
albeit of (local) geometric motions,
was recently pursued by Caffarelli \& Monneau \cite{CM12}.
The geometric approach to the Hele-Shaw problem
is complicated by the nonlocal nature of the problem.
Indeed, \eqref{HSt} can be interpreted as a geometric motion
of the free boundary with the velocity given
by a nonlocal operator based on the Dirichlet-to-Neumann map.
Thus the domain is of crucial importance.

%\parahead{Main results}
\subsection*{Main results}
We present new well-posedness and homogenization results
for the Hele-Shaw-type problem \eqref{HSt}.

The well-posedness result is
a generalization and an improvement of the previous results
in \cites{K03,K07}.
In full generality, we consider the Hele-Shaw type problem
\begin{align}
\label{hs-f}
\begin{cases}
-\Delta u(x,t) = 0 & \text{in } \set{u > 0} \cap Q,\\
u^+_t(x,t) = f(x,t,Du)\abs{Du^+(x,t)}^2 & \text{on  }  \pset{u > 0} \cap Q.
\end{cases}
\end{align}
Let us introduce a number of assumptions
on the function $f(x,t,p): Q \times \Rn \to \R$.
In the following, we use the semi-continuous envelopes $f_*$ and $f^*$
of $f$ on $Q \times \Rn$ (see \eqref{usclsc} for definition).
Let us point out that we do not require continuity of $f(x,t,p)$ in $p$.
We shall need this generality to
handle the homogenized problem.
\begin{enumerate}[\bfseries ({A}1)]
\item (non-degeneracy)
There exist constants $m$ and $M$
such that $0 < m \leq f(x,t,p) \leq M$
for all $(x,t,p) \in Q \times \Rn$.
\item (Lipschitz continuity)
There exists a constant $L > 0$
such that $f$ is $L$-Lipschitz in $x$ and $t$ for all $p$.
\item (monotonicity)
$f^*(x,t, a_1 p) \abs{a_1 p} \leq f_*(x,t, a_2 p) \abs{a_2 p}$
for any $(x, t, p) \in Q \times \Rn$
and $0 < a_1 < a_2$.
\end{enumerate}
The assumption (A3) implies that the free boundary
velocity is monotone with respect to the gradient,
while allowing for certain jumps.

We have the following well-posedness theorem
for viscosity solutions that are introduced in Section~\ref{sec:viscosity-solutions}.

\begin{theorem}[Well-posedness]
\label{th:well-posedness}
Let $Q := \Omega \times (0,T]$ where $\Omega$
is a domain that satisfies the assumptions above
and $T > 0$.
Assume that either $f(x,t,p) = f(x,t)$
or $f(x,t,p) = f(p)$
and that $f$ satisfies (A1)--(A3).
Then for any positive function $\psi \in C(\partial \Omega \times [0,T])$
strictly increasing in time,
and for any open set $\Omega_0 \subset \Rn$
with smooth boundary, $\partial \Omega_0 \in C^{1,1}$,
such that $\Omega^c \subset \Omega_0$ and $\Omega_0 \cap \Omega$
is bounded (see Figure~\ref{fig:setting}),
there exists a unique bounded viscosity solution
$u: \cl Q \to [0, \infty)$ of the Hele-Shaw-type problem \eqref{hs-f}
such that $u^* = u_* = u$ on $\partial_P Q := \cl Q \setminus Q$
with boundary data $u(x,t) = \psi(x,t)$ on $\partial \Omega \times [0,T]$
and initial data $u(\cdot, 0) > 0$ in $\Omega \cap \Omega_0$
and $u(\cdot, 0) = 0$ in $\Omega \setminus \Omega_0$.
The solution is unique in the sense that
$u_* = v_*$ and $u^* = v^*$ for any two viscosity solutions
$u$, $v$ with the given boundary data.
\end{theorem}

In the context of a flow in porous media,
$\Omega^c$ represents the source of a liquid
with the prescribed pressure $\psi$
on its boundary, and $\Omega_0$ is the initial wet region.
The situation is depicted in Figure~\ref{fig:setting}.

\medskip
However, the main result of this paper concerns the homogenization of \eqref{HSt}.
Note that because the solution of the homogenized problem
might be discontinuous,
we cannot expect uniform convergence of solutions in general.
Furthermore, we do not know
if the homogenized velocity
($r$ in Theorem~\ref{th:intro-homogenization} below) is continuous.

\begin{theorem}[Homogenization]
\label{th:intro-homogenization}
Suppose that $g \in C(\Rn \times \R)$
satisfies \eqref{g-bound}--\eqref{g-periodic}.
Then there exists a function $r: \Rn \to \R$
such that
\[f(x,t,p) := \frac{r(p)}{\abs p}\] satisfies (A1)--(A3),
and,
for any $Q$ and initial and boundary data $\Omega_0, \psi$
that satisfy the assumptions in Theorem~\ref{th:well-posedness}
the following results hold:
\begin{enumerate}
\item
the unique solutions $u^\e$ of \eqref{HSt}
with data $\Omega_0, \psi$
converge in the sense of half-relaxed limits as $\e\to0$
(Definition~\ref{def:half-relaxed}) to
the unique solution $u$ of
\eqref{hs-f} with $f(x,t,p) = r(p)/\abs{p}$
and the same data $\Omega_0, \psi$;
\item
if $u$ is also continuous on a compact set
$K \subset \cl Q$ then $u^\e \rightrightarrows u$
converge uniformly on $K$;
\item
the free boundaries $\partial \set{(u^\e)_* > 0}$
converge uniformly
to the free boundary $\partial \set{u_* > 0}$
with respect to the Hausdorff distance
(Definition~\ref{def:Hausdorff-convergence}).
\end{enumerate}
\end{theorem}

\subsection*{Sketch of the proof}
Let us give an overview of the main ideas in the paper.
Since the variational formulation via an obstacle problem,
discussed above, is not available,
we have to rely solely on the technically heavy tools of the viscosity theory. The time-dependence of $g$ poses significant new challenges which require a rather nontrivial extension of the previous results.

The first step is the identification of the homogenized problem.
Since the solutions of \eqref{HSt} are harmonic in space in their positive sets on any scale $\e$,
their limit as $\e \to 0$ should also be harmonic in space.
The free boundary velocity of the homogenized problem is, however,
unknown.
Following the ideas from \cite{CSW,K07},
we identify the correct homogenized free boundary velocity
by solving an obstacle problem and study its behavior as $\e\to0$.
To motivate our approach, suppose that the solutions $u^\e$ in the limit $\e \to 0$ converge
to a solution $u$ of the homogenized Hele-Shaw problem with
the free boundary law given as
\begin{align*}
    V_\nu &= r(Du) & &\text{on } \pset{u > 0},
\end{align*}
where $a \mapsto r(a q)$ is an increasing function on $\R_+$ for any $q \in \Rn$.
The crucial observation is that this problem has traveling wave solutions;
these are the planar solutions of the form
$\Pqr(x,t) = (\abs{q}r t + x \cdot q)_+$,
with the particular choice $r = r(q)$.
Here $(\cdot)_+$ stands for the positive part, i.e.,
$s_+ := \max(s, 0)$.
If $r > r(q)$ then $\Pqr$ is a supersolution
of the homogenized problem,
and if $r < r(q)$ it is a subsolution.
This observation allows us to identify the correct velocity $r(q)$
by solving the $\e$-problem \eqref{HSt} for each $\e > 0$ and comparing
the solution with the planar solution $\Pqr$ for given $r$ and $q$.
Loosely speaking, if $r$ is too large for a given $q$,
the solutions $u^\e$ should evolve slower than $\Pqr$
for small $\e$
and, similarly, if $r$ is too small then $u^\e$ should evolve faster
than $\Pqr$ for small $\e$.

To make this idea rigorous,
we solve the following obstacle problems
for each fixed $\e$:
given a fixed domain $Q$, $r$ and $q$,
find the largest subsolution $\lu_\eqr$ of the $\e$-problem
that stays under $\Pqr$ in $Q$
and the smallest supersolution $\uu_\eqr$ of the $\e$-problem
that stays above $\Pqr$ in $Q$.
Following this reasoning,
we find two candidates for the correct homogenized velocity $r(q)$
by ``properly measuring'' how much contact there is as $\e \to 0$
between the obstacle $\Pqr$
and the largest subsolution $\lu_\eqr$,
yielding $\lr(q)$,
and the smallest supersolution $\uu_\eqr$, yielding $\ur(q)$.
Since the free boundary velocity law is nonlocal,
given by a Dirichlet-to-Neumann map,
a good choice of the domain $Q$ for this procedure is important.

For homogenization to occur, it is necessary that both candidates
$\lr(q)$ and $\ur(q)$
yield the same limit problem \eqref{hs-f}.
We need to give a proper meaning to the intuitive idea of
evolving slower or faster than the obstacle $P_{q,r}$.
The selection of a quantity
that not only ``properly measures'' the amount of contact
between the obstacle $\Pqr$ and the solutions of the obstacle problem
in the homogenization limit $\e\to0$
but is also convenient to work with is far from obvious.
In particular, we want to take advantage
of the natural monotonicity (Birkhoff property)
of the solutions of the obstacle problem,
which is a consequence of the periodicity of the medium.

In \cite{CSW} as well as in \cite{K07,K08,K09},
the authors consider the coincidence set
of the solutions of the obstacle problem
and the obstacle $\Pqr$ as their quantity of choice.
This choice is motivated by the fact that
if there is a contact on a sufficiently large set,
then the solution must
be close to the obstacle everywhere.
This kind of estimate can be established using the
Alexandroff-Bakelman-Pucci estimate
in the case of fully nonlinear elliptic problems in \cite{CSW}.
Since ABP-type estimates are not available for
the Hele-Shaw problem,
the closeness to the obstacle must be derived by other means
\cite{K07}.

The disadvantage of this approach in the context of free boundary problems,
in particular for the homogenization of \eqref{HSt},
stems from the restriction that it imposes on
the directions in which one can translate
the solutions of the obstacle problem
to take advantage of their natural monotonicity property,
and to recover the monotonicity
of the contact set.
This creates technical difficulties and
requires a separate treatment of rational and irrational directions.
It seems that, to overcome these difficulties,
it is necessary to scale the solutions in time
for the arguments to work.
However, a scaling in time is not available
in the time-dependent medium.

The main new idea in this paper is
the introduction of \emph{flatness}
and its critical value $\e^\be$ (Section~\ref{sec:flatness}).
More specifically, to recover the correct homogenized
boundary velocity for a given gradient $q$
we test
if,
for some fixed $\be \in (0,1)$,
the free boundary of the solution of the obstacle problem
stays $\e^\be$-close to the obstacle for a unit time
on an arbitrary small scale $\e$.
This choice is motivated by the new cone flatness property
(Proposition~\ref{pr:cone-flatness}):
the boundary of the solution of the obstacle problem
stays between two cones that are $\sim \e \abs{\ln\e}^{1/2}$ apart.
Therefore if, for $\e \ll 1$, the boundary is farther than $\e^\be \gg \e \abs{\ln\e}^{1/2}$,
it will be detached from the obstacle on a large set.
This is formulated in the \emph{detachment lemma} (Lemma~\ref{le:detachment-ball}).
Moreover, the improvement of the \emph{local comparison principle}
(Theorem~\ref{th:localComparison})
for $\e^\be$-flat solutions to allow for $\be \in (4/5,1)$
is a necessary ingredient to close the argument.
This approach allows us to prove that
the candidates $\lr(q)$ and $\ur(q)$
almost coincide, i.e., $\lr = \ur^*$,
and that $f(x,t,p) = \lr(p)/\abs p$ satisfies
the assumptions (A1)--(A3).
Moreover, if the medium is time-independent, that is, $g(x,t) = g(x)$,
we are able to show that $\lr = \ur$
are continuous and one-homogeneous
(Section~\ref{sec:time-independent}),
essentially recovering the result of \cite{K07}.

Once the homogenized velocity has been identified,
we can prove that the half-relaxed limits
satisfy the homogenized problem (Section~\ref{sec:convergence}).
The perturbed test function method cannot be applied to free boundary
problems and a different, more geometric argument
based on a comparison of the solutions
of the $\e$-problem with rescaled translations of the solutions of the obstacle problems must
be engaged.
The detachment lemma (Lemma~\ref{le:detachment-ball}) plays a key role in this argument.
Finally, since the comparison with barriers guarantees that
the limits have the correct boundary data,
the comparison principle for the limit problem establishes
that the upper and lower half-relaxed limits coincide with the
unique solution of the homogenized problem.

\subsection*{Open problems}
Let us conclude the introduction by mentioning some of the open problems.
We have been only able to show that the homogenized velocity
$r$ is semi-continuous.
However, it seems quite reasonable
to expect continuity, or even H\"older continuity.
As the one-dimensional case suggests in Section~\ref{sec:homogenizationin1D},
this is the highest regularity one may hope for
in a general situation.
Another open problem is the question of a convergence rate of
free boundaries in the homogenization limit.
An algebraic rate was obtained in the periodic spatial homogenization case
\cite{K09}.
These issues will be addressed in future work.

A related open problem is the homogenization in
random media when the free boundary problem does not have a variational
structure, unlike \cite{KM09,KM10}.
This problem has not yet been solved since there is no obvious
subadditive quantity to which the subadditive ergodic theorem
can be applied to overcome the lack of compactness of a general
probability space.

Finally, in the current paper, we use the monotone propagation property
of the free boundary---that is, that the wet region cannot recede---to
obtain some of the important estimates.
The presented method, however, seems robust enough to handle
non-monotone problems such as the model of contact angle dynamics
as in \cite{K08,K09}. This will also be a subject of future work.

%\parahead{Outline}
\subsection*{Outline}
The exposition of the proof of the homogenization result
 was split in a number
of steps.
First we give the necessary definitions of solutions
in Section~\ref{sec:viscosity-solutions},
together with some preliminary results including the comparison principle
and a well-posedness result.
These are used in Section~\ref{sec:limit-problem}, where we identify a candidate for the limit velocity.
Finally, Section~\ref{sec:convergence} is devoted
to showing that the solutions converge in the homogenization limit.
In addition, the appendices contain some auxiliary results
used throughout the text.

\section{Viscosity solutions}
\label{sec:viscosity-solutions}

In this section, we briefly revisit the theory of viscosity solutions
of the Hele-Shaw problem \eqref{HSt}.
We generalize the definitions introduced in \cites{K03,K07},
and outline the proof of the comparison principle in our settings, Section~\ref{sec:comparison-principle},
with an aim at establishing Theorem~\ref{th:well-posedness} in Section~\ref{sec:well-posedness}.
Let us mention that viscosity solutions
seem to be the natural class of solutions of \eqref{HSt}
since the problem has a maximum principle structure.
Moreover, due to the possible topological changes
and merging of free boundaries, the solutions
might be discontinuous \cite{K03}.
The standard notion of solutions due to \cite{EJ}
does not apply when $g$ in \eqref{HSt} depends on time.

Since viscosity solutions are the only weak notion of solutions used
throughout the paper,
we will often refer to them simply as \emph{solutions}.

Before we give the definitions of viscosity solutions,
we need to introduce some notation.
For given radius $\rho > 0$ and center $(x,t) \in \Rn\times\R$,
we define the open balls
\begin{align*}
B_\rho(x,t) &:= \set{(y,s) \in \Rn\times\R:
\abs{y-x}^2 + \abs{s-t}^2 < \rho^2},\\
B_\rho(x) &:= \set{y \in \Rn : \abs{y-x} < \rho}.
\end{align*}
Let $E \subset \R^d$ for some $d \geq 1$.
Then $USC(E)$ and $LSC(E)$ are respectively
the sets of all upper semi-continuous
and lower semi-continuous functions on $E$.
For a locally bounded function $u$ on $E$
we define the semi-continuous envelopes
$u^{*, E} \in USC(\R^d)$ and $u_{*,E} \in LSC(\R^d)$
as
\begin{align}\label{usclsc}
\begin{aligned}
    u^{*,E} &:=
    \inf \set{v \in USC(\R^d): v \geq u \text{ on } E},\\
    u_{*, E} &:= \sup \set{v \in LSC(\R^d): v \leq u \text{ on } E}.
\end{aligned}
\end{align}
Note that $u^{*,E} : \R^d \to [-\infty, \infty)$
and $u_{*, E} : \R^d \to (-\infty, \infty]$ are finite on $\cl E$.
We simply write $u^*$ and $u_*$
if the set $E$ is understood from the context.
The envelopes can be also expressed as
\[
u^{*,E}(x) = \lim_{\de\to0} \sup \set{u(y) : y \in E,\ \abs{y - x} < \de}
\quad \text{for } x \in \cl E,
\qquad u_{*,E} = -(-u)^{*,E}.
\]

It will be also useful to use a shorthand notation for
the set of positive values of a given function $u: E \to [0,\infty)$,
defined on a set $E \subset \Rn\times \R$,
\begin{equation*}
\Omega(u; E) := \set{(x,t) \in E: u(x,t) > 0}, \qquad
\Omega^c(u; E) := \set{(x,t) \in E: u(x,t) = 0},
\end{equation*}
and the closure $\cl\Omega(u;E) := \cl{\Omega(u;E)}$.
For $t \in \R$, the time-slices $\cl\Omega_t(u; E)$,
$\Omega_t(u; E)$ and $\Omega^c_t(u; E)$
are defined in the obvious way, i.e.,
\begin{align*}
\cl\Omega_t(u;E) = \set{x : (x,t) \in \cl\Omega(u;E)}, \qquad \text{etc.}
\end{align*}
We shall call the boundary of the positive set in $E$ the \emph{free
boundary} of $u$ and denote it $\Gamma(u;E)$, i.e.,
\begin{align*}
\Gamma(u; E) = (\partial \Omega(u; E)) \cap E.
\end{align*}
If the set $E$ is understood from the context, we shall simply write
$\Omega(u)$, etc.

For given constant $\tau \in \R$ we will often abbreviate
\begin{align*}
\set{t \leq \tau} := \set{(x,t) \in \Rn\times\R: t \leq \tau},
    \qquad \text{etc.}
\end{align*}

It will be convenient to define viscosity solutions
on general parabolic neighborhoods;
we refer the reader to \cite{WangI} for a more general definition.

\begin{definition}[Parabolic neighborhood and boundary]
\label{def:parabolic-nbd}
\ \\A nonempty set $E \subset \Rn\times \R$ is called a \emph{parabolic neighborhood} if $E = U \cap \set{t \leq \tau}$ for some open set $U \subset \Rn\times \R$ and some $\tau \in \R$.
We say that $E$ is a parabolic neighborhood of $(x,t) \in \Rn\times\R$
if $(x,t) \in E$.
Let us define $\partial_P E := \cl{E} \setminus E$, the \emph{parabolic boundary} of $E$.
\end{definition}

\begin{remark}
Since $\interior E = U \cap \set{t < \tau}$,
one can observe that
$\partial_P E = \partial E \setminus (U \cap \set{t = \tau})$.
If $E = U$ then $\partial_P E = \partial E$. Note that we do not require that $E \neq U$, which is usually assumed, simply because it is unnecessary in this paper.
Adding this requirement does not change any of the results presented.
\end{remark}

Following \cites{K03,K07},
we define the viscosity solutions for \eqref{hs-f}.
In the following definitions,
$Q \subset \Rn\times\R$ is an arbitrary
parabolic neighborhood in the sense of
Definition~\ref{def:parabolic-nbd}
and
$f(x,t,p): Q \times \Rn \to \R$ satisfies assumptions (A1)--(A3)
(Section~\ref{sec:introduction}).
We do \emph{not} assume that $f$ is continuous in $p$.

\begin{definition}
\label{def:visc-test-sub}
We say that a locally bounded, non-negative
function $u: Q \to [0,\infty)$ is
a \emph{viscosity subsolution} of \eqref{hs-f} on $Q$ if
\begin{enumerate}
\romanlist
\item \emph{(continuous expansion)}
\[
\cl\Omega(u; Q) \cap Q \cap \set{t \leq \tau}
    \subset \cl{\Omega(u;Q) \cap \set{t < \tau}}
    \quad \text{for every $\tau > 0$},
\]

\item \emph{(maximum principle)}\\
for any $\phi \in C^{2,1}$
such that $u^* - \phi$ has a local maximum at
$(x_0, t_0) \in Q \cap \cl\Omega(u;Q)$
in $\cl\Omega(u;Q) \cap \set{t \leq t_0}$, we have
\begin{enumerate}[({ii}-1)]
\item if $u^*(x_0, t_0) > 0$ then $-\Delta \phi(x_0, t_0) \leq 0$,
\item if $u^*(x_0, t_0) = 0$ then either
    $-\Delta \phi(x_0, t_0) \leq 0$ or
    $D \phi(x_0, t_0) = 0$\\
    or
    $[\phi_t - f^*(x_0,t_0, D\phi(x_0, t_0)) \abs{D\phi}^2] (x_0, t_0) \leq 0$.
\end{enumerate}
\end{enumerate}
\end{definition}

\begin{remark}
The condition (i) in Definition~\ref{def:visc-test-sub} is necessary
to prevent a scenario where a `bubble' closes instantly;
more precisely, a subsolution cannot become instantly
positive on an open set surrounded by a positive phase,
or cannot fill the whole space instantly.
\end{remark}

The definition of a viscosity supersolution is similar.

\begin{definition}
\label{def:visc-test-super}
We say that a locally bounded, non-negative
function $u: Q \to [0,\infty)$ is
a \emph{viscosity supersolution} of \eqref{hs-f} on $Q$ if
\begin{enumerate}
\romanlist
\item (monotonicity of support)\\
    if $(\xi,\tau) \in \Omega(u_*; Q)$ then
    $(\xi, t) \in \Omega(u_*;Q)$
    for all $(\xi,t) \in Q$, $t \geq \tau$.
\item (maximum principle)\\
for any $\phi \in C^{2,1}$
such that $u_* - \phi$ has a local minimum at $(x_0, t_0) \in Q$
in $\set{t \leq t_0}$, we have
\begin{enumerate}[({ii}-1)]
\item if $u_*(x_0, t_0) > 0$ then $-\Delta \phi(x_0, t_0) \geq 0$,
\item if $u_*(x_0, t_0) = 0$ then either
    $-\Delta \phi(x_0, t_0) \geq 0$ or
    $D \phi(x_0, t_0) = 0$
    \\or
    $[\phi_t - f_*(x_0,t_0,D\phi(x_0, t_0)) \abs{D\phi}^2] (x_0, t_0)
        \geq 0$.
\end{enumerate}
\end{enumerate}
\end{definition}

\begin{remark}
The condition (i) in Definition~\ref{def:visc-test-super}
states that the support of a supersolution is nondecreasing.
It is a purely technical assumption which simplifies
the proof of the comparison theorem
by preventing an instantaneous disappearance of
a component of a positive phase.
Indeed, it can be shown that all solutions of the Hele-Shaw problem
have nondecreasing support.
This assumption can be removed by using the tools developed in
\cite{KP12}.
\end{remark}

Finally, we define a viscosity solution by combining the two previous definitions.

\begin{definition}
\label{def:visc-sol}
A function $u: Q \to [0,\infty)$ is a \emph{viscosity solution}
of \eqref{hs-f} on $Q$ if $u$ is both a viscosity subsolution
and a viscosity supersolution on $Q$.
\end{definition}

It will be useful to state an equivalent definition of viscosity solutions
via a comparison with barriers,
which in our case are strict classical sub- and supersolutions.
This definition seems more natural,
as it follows a more common pattern appearing
in the treatment of free boundary problems;
see \cites{ACSI,BS98,CS,CV,KP11,KP12}.

\begin{definition}
For a given closed set $K \in \Rn\times\R$,
we say that $\phi \in C^{2,1}_{x,t}(K)$
if there exists open set $U \supset K$
and $\psi \in C^{2,1}_{x,t}(U)$,
i.e., a function twice continuously differentiable in space
and once in time,
such that $\phi = \psi$ on $K$.
\end{definition}

\begin{definition}
\label{def:classical-sub}
Given a nonempty open set $U \subset \Rn\times\R$,
a function $\phi \in C(\cl U) \cap C^{2,1}_{x,t} (\cl\Omega(\phi; U))$
is called a \emph{subbarrier} of \eqref{hs-f} in $U$
if there exists a positive constant $c$ such that
\begin{enumerate}
\romanlist
\item
    $- \Delta \phi < -c$ on $\Omega(\phi; U)$, and

\item
    $\abs{D\phi^+} > c$ and
     $\phi^+_t - f_*(\cdot, \cdot, D \phi^+) \abs{D \phi^+}^2 < -c$
    on $\Gamma(\phi; U)$.
\end{enumerate}
Here $D\phi^+$ and $\phi_t^+$ denote the limits of $D\phi$ and $\phi_t$, respectively, on $\Gamma(\phi;U)$ from $\Omega(\phi;U)$.
\end{definition}

\begin{definition}
\label{def:classical-super}
Given a nonempty open set $U \subset \Rn\times\R$,
a function $\phi \in C(\cl U) \cap C^{2,1}_{x,t} (\cl\Omega(\phi;U))$
is called a \emph{superbarrier} of \eqref{hs-f} in $U$
if there exists a positive constant $c$ such that
\begin{enumerate}
\romanlist
\item
    $- \Delta \phi > c$ on $\Omega(\phi; U)$, and

\item
    $\abs{D\phi^+} > c$ and $\phi^+_t
        - f^*(\cdot, \cdot, D \phi^+) \abs{D \phi^+}^2 > c$
    on $\Gamma(\phi; U)$.
\end{enumerate}
\end{definition}

A notion of \emph{strict separation} is a crucial concept in the theory
and will be used multiple times
(our definition differs slightly from the one introduced in \cite{K03}).

\begin{definition}[Strict separation]
Let $E \subset \Rn \times \R$ be a parabolic neighborhood,
and $u, v : E \to \R$ be
bounded non-negative functions on $E$,
and let $K \subset \cl E$.
We say that $u$ and $v$ are strictly separated
on $K$ with respect to $E$,
and we write $u \prec v$ in $K$ w.r.t. $E$,
if
\begin{align*}
u^{*,E} &< v_{*,E} &&\text{in }&
    &K \cap \cl\Omega(u;E).
\end{align*}
\end{definition}

A definition of viscosity solutions using barriers follows.

\begin{definition}
\label{def:visc-barrier}
We say that a locally bounded, non-negative
function $u: Q \to [0,\infty)$ is
a \emph{viscosity subsolution} of \eqref{hs-f} on $Q$ if
for every bounded parabolic neighborhood $E = U \cap \set{t \leq \tau}$,
$E \subset Q$,
and every superbarrier $\phi$ on $U$
such that $u \prec \phi$ on $\partial_P E$ w.r.t. $E$,
we also have $u \prec \phi$ on $\cl E$ w.r.t. $E$.

Similarly, a locally bounded, non-negative
function $u: Q \to [0,\infty)$ is
a \emph{viscosity supersolution} if
Definition~\ref{def:visc-test-super}(i) holds and
for any
subbarrier $\phi$ on $U$
such that $\phi \prec u$ on $\partial_P E$ w.r.t. $E$,
we also have $\phi \prec u$ on $\cl E$ w.r.t. $E$.

Finally, $u$ is a viscosity solution if it is both
a viscosity subsolution and a viscosity supersolution.
\end{definition}

\bigskip
We finish this section by stating the equivalence
of the two definitions of viscosity solutions.

\begin{proposition}
\label{pr:visc-equivalency}
The definitions of viscosity subsolutions (resp. supersolutions)
in Definition~\ref{def:visc-test-sub} (resp. \ref{def:visc-test-super})
and in Definition~\ref{def:visc-barrier}
are equivalent.
\end{proposition}

Before proceeding with the proof,
let us state a useful property of a strict separation
on time dependent sets.

\begin{lemma}
\label{le:ordering-open}
Suppose that $E$ is a bounded parabolic neighborhood
and $u$, $v$ are non-negative locally bounded functions on $E$.
The set
\begin{align}
\label{ordering-set}
\Theta_{u,v;E} := \set{\tau: u \prec v
\text{ in $\cl E \cap \set{t \leq \tau}$
w.r.t. $E$}}
\end{align}
is open
and $\Theta_{u,v;E} = (-\infty, T)$ for some $T \in (-\infty, \infty]$.
\end{lemma}

\begin{proof}
By replacing $u$ by $u^{*,E}$ and $v$ by $v_{*, E}$, we may assume that $u$ is upper semi-continuous and $v$ is lower semi-continuous.
Let us choose $\tau \in \Theta_{u,v;E}$.
Clearly $s \in \Theta_{u,v;E}$ for all $s \leq \tau$.
Therefore, we only need to show that there exists $s > \tau$ such that
$s \in \Theta_{u,v;E}$.
We can assume that
$\set{t \leq \tau} \cap \cl E \neq \emptyset$,
otherwise the claim is trivial since $\cl E$ is compact.
Suppose that $u \not\prec v$
on $\cl E \cap \set{t \leq s}$ w.r.t. $E$ for all $s > \tau$.
Hence
there is a sequence $(x_k, t_k) \in \cl\Omega(u;E)$ such that
$t_k \searrow \tau$ and
$u(x_k, t_k) \geq v(x_k, t_k)$.
By compactness of $\cl E$,
we can assume that $(x_k, t_k) \to (\xi, \tau)$,
and we see that
$(\xi, \tau) \in \cl\Omega(u;E) \cap \set{t \leq \tau}$.
By semi-continuity we have $u(\xi, \tau) - v(\xi, \tau) \geq 0$,
a contradiction with $u \prec v$ in $\cl E \cap \set{t \leq \tau}$.
\qedhere\end{proof}

\begin{proof}[Proof of Proposition~\ref{pr:visc-equivalency}]
Let us prove the statement for subsolutions.
Suppose that $u$ is a subsolution on some parabolic neighborhood $Q$
in the sense
of Definition~\ref{def:visc-test-sub}.
By taking the semi-continuous envelope,
we can assume that $u \in USC(\cl Q)$.

Let
$E = U \cap \set{t \leq \tau} \subset Q$ be a bounded
parabolic neighborhood
and let $\phi$ be a superbarrier on $U$
such that $u \prec \phi$ on $\partial_P E$ w.r.t. $E$.
Consider the set $\Theta_{u, \phi, E}$,
introduced in \eqref{ordering-set}.
Let us define $\hat t = \sup \Theta_{u, \phi, E}$.
Since $E$ is bounded, we have $\hat t > -\infty$.
To show that $u$ is a subsolution
in the sense of Definition~\ref{def:visc-barrier},
it is enough to show that
$u \prec \phi$ in $\cl E$,
which is equivalent to
$\hat t = +\infty$.

Hence suppose that $\hat t < +\infty$.
Lemma~\ref{le:ordering-open} implies that
\begin{compactenum}[(a)]
\item
$u \not\prec \phi$ in $\cl E \cap \set{t \leq \hat t}$ w.r.t. $E$,
\item
$u \prec \phi$ in $\cl E \cap \set{t \leq s}$ for all $s < \hat t$,
and, finally,
\item
$\cl\Omega(u; E) \cap \set{t \leq \hat t} \subset \cl\Omega(\phi;E)$
due to (b) and the continuous expansion in
Definition~\ref{def:visc-test-sub}(i).
\end{compactenum}

By semi-continuity and compactness, $u - \phi$ has a local maximum
at some point $(\xi, \si)$ in $\cl\Omega(u;E) \cap \set{t \leq \hat t}$.
Then (a) implies that $u(\xi, \si) - \phi(\xi, \si) \geq 0$.
Consequently, since (b) holds,
we must have $\si = \hat t$,
and $u \prec \phi$ on $\partial_P E$ implies
$(\xi, \hat t) \in E$.
Finally, using (c) we conclude that
$(\xi, \hat t) \in E \cap \cl\Omega(u; E) \cap \cl\Omega(\phi;E)$.
Therefore either:
\begin{compactitem}
\item
$u(\xi,\hat t) > 0$,
and that is a contradiction with
Definition~\ref{def:visc-test-sub}(ii-1); or
\item $u(\xi, \hat t) = 0$,
but then $\phi(\xi,\hat t) = 0$ and consequently
$(\xi, \hat t) \in \Gamma(\phi;E)$,
which yields a contradiction with
Definition~\ref{def:visc-test-sub}(ii-2).
\end{compactitem}

The proof for supersolutions is analogous,
but more straightforward since the support of $u_*$ cannot
decrease due to Definition~\ref{def:visc-test-super}(i).

\bigskip
The direction from Definition~\ref{def:visc-barrier}
to Definitions~\ref{def:visc-test-sub} and \ref{def:visc-test-super}
is quite standard.
In particular,
Definition~\ref{def:visc-test-sub}(i)
can be shown by a comparison with barriers
as in the proof of Corollary~\ref{co:HS-expansion-speed}.
Furthermore,
any test function $\phi$ such that $u - \phi$ has a local maximum
at $(\hat x, \hat t)$
and violates either of
Definition~\ref{def:visc-test-sub}(ii)
can be used for a construction of a superbarrier
in a parabolic neighborhood of $(\hat x, \hat t)$
by considering
$\tilde \phi(x,t)
= (\phi(x,t) + |x - \hat x|^4 + |t - \hat t|^2 - \phi(\hat x, \hat t))_+$.
Once again,
the proof for supersolutions is similar.
\qedhere\end{proof}

\begin{definition}
For a given function $f$
and a nonempty parabolic neighborhood $Q \subset \Rd \times \R$,
we define the following classes of functions:
\begin{itemize}
\item $\supers(f,Q)$,
    the set of all viscosity supersolutions of the Hele-Shaw problem
    \eqref{hs-f} on $Q$;
\item $\subs(f,Q)$,
    the set of all viscosity subsolutions of \eqref{hs-f} on $Q$;
\item $\sol(f,Q) = \supers(f,Q) \cap \subs(f,Q)$,
    the set of all viscosity solutions of \eqref{hs-f} on $Q$.
\end{itemize}
\end{definition}

We have the following obvious inclusions:

\begin{lemma}
Let $Q \subset \Rn \times \R$ be a nonempty parabolic neighborhood.
If $f \leq g$ for given functions $f,g: Q \times \Rn \to (0,\infty)$ then
\begin{align*}
\supers(f, Q) &\supset \supers(g, Q) &&\text{and}&
\subs(f,Q) &\subset \subs(g, Q).
\end{align*}
\end{lemma}

Finally, we observe that subsolutions and supersolutions can be
``stitched'' together.

\begin{lemma}
\label{le:stitch}
Let
$Q_1, Q_2 \subset \Rn\times\R$ be two parabolic neighborhoods
such that $Q_1 \subset Q_2$
and let $u_i \in \subs(f, Q_i)$, $i = 1,2$,
for some $f(x,t,p) : Q_2 \times \Rn \to (0,\infty)$.
We define
\begin{align*}
u =
\begin{cases}
\max (u_1, u_2) & \text{in } Q_1,\\
u_2 & \text{in } Q_2 \setminus Q_1.
\end{cases}
\end{align*}
If $u_1^{*,Q_1} \leq u_2^{*,Q_2}$ on $(\partial_P Q_1) \cap Q_2$
and
\begin{equation}
\label{incl-on-pb}
\cl\Omega(u_1; Q_1) \cap \partial_P Q_1 \subset \cl\Omega(u_2;Q_2)
\end{equation}
then $u \in \subs(f, Q_2)$.

Let now $v_i \in \supers(f, Q_i)$, $i = 1,2$,
and we define
\begin{align*}
v =
\begin{cases}
\min (v_1, v_2) & \text{in } Q_1,\\
v_2 & \text{in } Q_2 \setminus Q_1.
\end{cases}
\end{align*}
If $(v_1)_{*,Q_1} \geq (v_2)_{*,Q_2}$ on $(\partial_P Q_1) \cap Q_2$,
then $v \in \supers(f, Q_1)$.
\end{lemma}

\begin{proof}
Let us show the proof for subsolutions.
We use Definition~\ref{def:visc-test-sub}.
Clearly $\Omega(u;Q_2) = \Omega(u_1;Q_1) \cup \Omega(u_2;Q_2)$
and therefore, by \eqref{incl-on-pb},
we have
\begin{align*}
\cl\Omega(u;Q_2) &= \cl\Omega(u_1;Q_1) \cup \cl\Omega(u_2;Q_2)
    \\&= (\cl\Omega(u_1;Q_1) \cap Q_1) \cup
        (\cl\Omega(u_1;Q_1) \cap \partial_P Q_1)
        \cup \cl\Omega(u_2;Q_2)
        \\&= (\cl\Omega(u_1;Q_1) \cap Q_1) \cup  \cl\Omega(u_2;Q_2).
\end{align*}
Thus Definition~\ref{def:visc-test-sub}(i) is clear.
To prove (ii),
suppose that $u^* - \phi$ has a local maximum at
$(\hat x,\hat t) \in \cl\Omega(u; Q_2)$
in $\cl\Omega(u; Q_2) \cap \set{t \leq \hat t}$.
We observe that, by hypothesis,
$u^* = \max(u_1^{*,Q_1}, u_2^{*,Q_2})$ on $Q_1$
and $u^* = u_2^{*,Q_2}$ on $Q_2 \setminus Q_1$.
Consequently, if $(\hat x, \hat t) \in Q_1$ then $u_1^* - \phi$
and/or $u_2^* - \phi$ have a local maximum
in the appropriate set,
and if $(\hat x, \hat t) \in Q_2 \setminus Q_1$ then $u_2^* - \phi$
has a local maximum.  Definition~\ref{def:visc-test-sub}(ii)
follows from this observation.

The proof for supersolutions is straightforward
since there is no restriction on their supports.
\qedhere\end{proof}

\subsection{Comparison principle}
\label{sec:comparison-principle}

A crucial result for a successful theory of viscosity solutions
is a comparison principle.
In this section we establish the comparison principle
between strictly separated viscosity solutions
on an arbitrary parabolic neighborhood.
We only give an outline of the proof
and point out the differences from previous results
in \cites{K03,K07,KP12}.

\begin{theorem}
\label{th:comparison}
Let $Q$ be a bounded parabolic neighborhood.
Suppose that $f : Q \times \Rn \to \R$ is a given function
that satisfies the assumptions (A1)--(A3) and is of the form either
$f(x,t,p) = g(x,t)$
or
$f(x,t,p) = f(p)$.
If $u \in \subs(f, Q)$ and $v\in \supers(f, Q)$
such that $u \prec v$ on $\partial_P Q$ w.r.t. $Q$,
then $u \prec v$ in $\cl Q$ w.r.t. $Q$.
\end{theorem}

\begin{proof}
\emph{Step 1.} We can assume that $u \in USC(\cl Q)$ and $v \in LSC(\cl Q)$.
Similarly to the proof of Proposition~\ref{pr:visc-equivalency},
we consider the set $\Theta_{u,v; Q}$,
defined in \eqref{ordering-set},
and introduce $\hat t_0 := \sup \Theta_{u,v; Q} > -\infty$.
Then $u \prec v$ on $\cl Q$ w.r.t. $Q$
is equivalent to $\hat t_0 = +\infty$.

Therefore suppose that $\hat t_0 < +\infty $.
Let us introduce the sup- and inf-convolutions
\begin{align*}
Z(x,t) &=  \pth{1+ \frac{3 L r}{m}} \sup_{\cl \Xi_r(x,t)}u,\\
W(x,t) &= \pth{1 - \frac{3 L r}{m}} \inf_{\cl \Xi_{r - \de t}(x,t)} v(x,t),
\end{align*}
with $0 < r \ll \frac{m}{3L}$ and $\de = \frac{r}{2T}$
where $T := \inf \set{\tau: Q \subset \set{t \leq \tau}}$.
The open set $\Xi_r(x,t)$
is defined as
\begin{align*}
    \Xi_r(x,t)
    = \set{(y,s) : (\abs{y - x}-r)_+^2 + \abs{s - t}^2 < r^2}.
\end{align*}
We refer the reader to \cites{K03,KP12}
for a detailed discussion of the properties of $Z$ and $W$.
In particular,
$Z$ and $W$ are well-defined on the closure
of the parabolic neighborhood $Q_r$,
\begin{align*}
Q_r := \set{(x,t): \cl\Xi_r(x,t) \subset Q},
\end{align*}
and $Z \in USC(\cl Q_r)$
and $W \in LSC(\cl Q_r)$.

Moreover,
$Z \in \subs(f, Q_r)$ and $W \in \supers(f, Q_r)$.
This is straightforward if $f(x,t,p) = f(p)$.
However, if $f$ depends on $(x,t)$ as in $f(x,t,p) = g(x,t)$
we have to use Proposition~\ref{pr:nonuniform-perturbation}
with factors $a = (1 \pm 3Lr/m)$, yielding the same conclusion.

Finally, and this is the main motivation for
the choice of the convolutions,
both $Z$ and $W$ have an important
interior/exterior ball property of their level sets
both in space for each time and in space-time;
see \cite{KP12}.

\emph{Step 2.} Since $u \prec v$ on $\partial_P Q$ w.r.t. $Q$,
we claim that
if $r$ is chosen small enough
we have $Z \prec W$ on $\partial_P Q_r$ w.r.t. $Q_r$.
Suppose that this is not true for any $r > 0$.
Let $r_k \to 0$ as $k \to \infty$.
There exists a sequence of $(x_k,t_k) \in \partial_P Q_{r_k}
\cap \cl\Omega(Z_{r_k};Q_{r_k})$
such that $Z_{r_k}(x_k,t_k) \geq W_{r_k}(x_k,t_k)$.
By definition of $Z$ and $W$ and semi-continuity,
there exist $(\xi_k, \si_k), (\zeta_k, \tau_k) \in \cl\Xi_{r_k}(x_k,t_k)$
such that $(\xi_k, \si_k) \in \cl\Omega(u;Q)$
and $u(\xi_k, \si_k) \geq v(\zeta_k, \tau_k)$.
By compactness of $\cl\Omega(u;Q)$,
for a subsequence, still denoted by $k$,
$(\xi_k,\si_k)$ converges to a point
$(\xi, \si)$, which must lie in $\partial_P Q \cap \cl\Omega(u; Q)$.
A simple observation $(\zeta_k, \tau_k) \in \cl\Xi_{2r_k}(\xi_k, \sigma_k)$ then implies that $(\zeta_k, \tau_k) \to (\xi, \si)$.
This is a contradiction with the semi-continuity of $u$ and $v$
and the fact that $u(\xi, \si) < v(\xi, \si)$.

\emph{Step 3.} We claim that $W$ is a supersolution of \eqref{hs-f} on $Q_r$ with the free boundary velocity increased by $\delta$, that is,
with $f^\de(x,t,p) = f(x,t,p) + \de \abs{p}^{-1}$.
As we already know that $W \in \supers(f, Q_r)$, we only need to verify Definition~\ref{def:visc-test-super}(ii-2).
Thus suppose that $\phi$ is a $C^{2,1}$-function
and $W - \phi$ has a strict local minimum $0$ at $(\xi, \si)$
in $\set{t \leq \si}$, and
$W(\xi, \si) = 0$, $-\Delta \phi(\xi, \si) < 0$ and $\abs{D \phi(\xi, \si)} > 0$.
By semi-continuity there exists $(y, s) \in \cl\Xi_{r - \de \si}(0,0)$
with
\[\pth{1 - \frac{3Lr}m}v(\xi + y, \si + s) = W(\xi, \si).\]
Moreover, the definition of $\phi$ and $W$ implies that
\begin{align}
\label{phi-v-with-shift}
    \phi(x,t) \leq W(x,t) \leq \pth{1 - \frac{3Lr}m}v(x + y + z, t + s),
\end{align}
for any $(y + z, s) \in \cl \Xi_{r - \de t}(0,0)$
and $(x,t)$ in a neighborhood of $(\xi, \sigma)$, $t \leq \sigma$.
Since the set $\cl \Xi_{r -\delta t}(0,0)$ is decreasing in $t$,
we can take any $z \in \cl B_{\de (\si - t)}(0)$.
Let us define $\nu = -\frac{D\phi}{\abs{D\phi}} (\xi, \si)$,
the unit outer normal to $\Omega_\si(\phi)$ at $\xi$.
With the particular choice $z = -\de (\si - t)\nu$,
we can rewrite \eqref{phi-v-with-shift} after a change of variables as
\begin{align*}
\psi(x,t) := \pth{1 - \frac{3Lr}m}^{-1}\phi(x - y + \de (\si + s - t) \nu, t - s) \leq
    v(x,t),
\end{align*}
for $(x,t)$ close to $(\zeta, \mu) := (\xi + y, \si + s)$, $t \leq \sigma + s$.
Additionally, the equality holds at $(\zeta, \mu)$
by the choice of $(y,s)$.
Thus $\psi$ is a test function for $v$ at $(\zeta, \mu)$ for Definition~\ref{def:visc-test-super}(ii-2),
and we have
\begin{align*}
[\phi_t - \de \abs{D\phi}](\xi,\si) &= \pth{1 - \frac{3Lr}m}\psi_t(\zeta, \mu)\\
& \geq
\pth{1 - \frac{3Lr}m} f_*\pth{\zeta, \mu, D\psi(\zeta, \mu)} \abs{D \psi(\zeta, \mu)}^2\\
&= \pth{1 - \frac{3Lr}m}^{-1} f_*\pth{\zeta, \mu, \pth{1 - \frac{3Lr}m}^{-1} D\phi(\xi, \sigma)} \abs{D \phi(\xi, \sigma)}^2.
\end{align*}
If $f(x,t,p) = g(x,t)$,
then we use the assumptions (A1)--(A2) that imply
$g(\zeta,\mu) \geq \pth{1 - \frac{3Lr}{m}}g(\xi, \si)$,
and if $f(x,t,p) = f(p)$ we use the assumption (A3)
to conclude that
\begin{align*}
[\phi_t - \de \abs{D\phi}](\xi,\si) \geq f_*(\xi, \si, D\phi(\xi,\sigma))
\abs{D\phi(\xi,\si)}^2.
\end{align*}
Therefore $W \in \supers(f^\de, Q_r)$.

\bigskip
\emph{Step 4.} We proceed with the proof of the comparison theorem.
Let us set \[\hat t = \sup \Theta_{Z, W; Q_r}\]
and observe that $\hat t \leq \hat t_0$.

As in the proof of Proposition~\ref{pr:visc-equivalency},
we have $\cl\Omega(Z) \cap \set{t \leq \hat t}
    \subset \cl{\Omega(W) \cap \set{t < \hat t}}$,
and there exists
$(\xi, \hat t) \in Q_r \cap \cl\Omega(Z) \cap \cl\Omega(W)$
such that $(\xi, \hat t)$ is a point of maximum of $Z - W$
in $\cl\Omega(Z) \cap \set{t \leq \hat t}$,
and $Z(\xi, \hat t) - W(\xi, \hat t) \geq 0$.
Since the set $\Omega^c_{\hat t}(W)$ has the interior
ball property of radius $r/2$,
we can show by a barrier argument
that in fact $Z = 0$ on $\Omega^c(W) \cap \set{t \leq \hat t}$,
see \cites{K03,KP12} for details.
Most of the arguments are analogous since $Z \in \subs(M, Q_r)$,
$W \in \supers(m, Q_r)$.
This together with the assumption that $\Omega(W)$ cannot decrease
in time due to Definition~\ref{def:visc-test-super}(i)
yields
$\Omega(Z) \cap \set{t \leq \hat t} \subset \Omega(W)$.

Suppose that $Z(\xi, \hat t) > 0$.
Let us denote by $U$ the connected component of $\Omega_{\hat t}(W)$
that contains $(\xi, \hat t)$.
Since $Z(\cdot, \hat t)$ is subharmonic in the (open) set
$U$,
$Z(\cdot, \hat t)$ cannot be identically zero on $\partial U$.
But the conclusion of the previous paragraph implies that
$Z \equiv 0$ on $(\partial U \times \set{\hat t}) \cap Q_r$.
Hence $(\partial U \times \set{\hat t}) \cap \partial_P Q_r \neq \emptyset$.
We have $0 < Z < W$ at the points of $\partial_P Q_r$ where $Z > 0$ by the assumption of strict separation.
Consequently, a straightforward argument of
adding a small harmonic function to
$Z$, using the comparison principle for subharmonic
and superharmonic functions,
yields $Z(\xi, \hat t) < W(\xi, \hat t)$, a contradiction with the previous paragraph.

Therefore $Z(\xi, \hat t) = W(\xi, \hat t) = 0$
and $Z \leq W$ on $\set{t \leq \hat t}$.
We arrive at the last step of the comparison theorem,
where we need to find some weak ordering of gradients
and the corresponding velocities, as in \cite{KP12}.

\emph{Step 5.} Hopf's lemma implies that the ``gradients'' of $Z$ and $W$ are
strictly ordered.
Indeed, let $\nu$ be the unique unit outer normal of
$\Omega_{\hat t}(Z)$ and $\Omega_{\hat t}(W)$ at $\xi$.
Since $Z(\cdot, \hat t)$ is subharmonic in $U$
and $W$ is superharmonic in
$U$,
and there is an interior ball of $\Omega_{\hat t}(Z)$ at $\xi$,
and $W(\cdot, \hat t) \not\equiv Z(\cdot, \hat t)$ in $U$,
we can apply the Hopf's lemma in $U$ at $\xi \in \partial U$
and conclude that
\begin{align}
\label{gradient-order}
    \alpha := \limsup_{h\to 0} \frac{Z(\xi - h\nu, \hat t)}{h}
        < \liminf_{h\to 0} \frac{W(\xi - h\nu, \hat t)}{h} =: \beta.
\end{align}
The assumption on $f$ implies
\begin{align}
\label{speed-order}
    f^*\pth{\xi, \hat t, -\al \nu} \abs{-\al \nu}
        \leq f_*\pth{\xi, \hat t, -\beta \nu} \abs{-\be \nu}.
\end{align}

It can be showed as in \cite{K03} (see also \cite{KP11,KP12})
that the sets $\cl\Omega(Z)$ and $\Omega^c(W)$
have interior space-time balls at $(\xi, \hat t)$
with space-time slopes that represent the velocities
of the free boundaries at $(\xi, \hat t)$
which we denote as $m_Z$ and $m_W$.
A barrier argument yields $0 < \de \leq m_W \leq m_Z < \infty$,
using the result of Step 3.

\emph{Step 6.} Now we proceed as in \cite{KP12}:
There are points $(\xi_u, t_u)$
and $(\xi_v, t_v)$
on the free boundaries of $u$ and $v$, respectively,
such that $(\xi_u, t_u), (\xi_v, t_v) \in \partial \Xi_r(\xi, \hat t)$.
Since $Z \leq W$ for $t\leq \hat t$ and $Z > 0$ on $\Xi_r(\xi_u, t_u)$,
we can deduce that $v > 0$ on $\bigcup \cl \Xi_{r-\delta t} (x,t)$,
where the union is over all points $(x,t) \in \Xi_r(\xi_u, t_u)$, $t \leq \hat t$.
Therefore as in \cite[Lemma~3.23]{KP12}, given arbitrary $\eta > 0$ small
we can put a radial test function under $v$ in a neighborhood of $(\xi_v, t_v)$
that touches $v$ from below at $(\xi_v, t_v)$ with free boundary velocity $m_Z - \delta + \eta$
and gradient $-\pth{1 -\frac{3Lr}m}^{-1} (\beta - \eta)\nu$ at $(\xi_v, t_v)$.
The definition of supersolution then implies
\begin{align*}
m_Z - \delta + \eta&\geq f_*\pth{\xi_v, t_v, -\pth{1 -\frac{3Lr}m}^{-1} (\beta - \eta) \nu} \pth{1 -\frac{3Lr}m}^{-1}(\beta - \eta),
\intertext{which can be estimated using (A1)--(A2) or (A3), depending on the form of $f$ as in Step 3, as}
&\geq f_*(\xi, \hat t, -(\beta - \eta)\nu) (\beta - \eta).
\end{align*}
Sending $\eta\to0$, the lower semi-continuity of $f_*$ implies
\begin{align*}
m_Z - \delta \geq f_*(\xi, \hat t, - \beta \nu) \beta.
\end{align*}
A similar, but more direct proof using the fact that $Z = 0$ on $\cl \Xi_r(\xi, \hat t)$ establishes
\begin{align*}
m_Z \leq f^*(\xi, \hat t, -\alpha \nu) \alpha.
\end{align*}

When we combine these two inequalities through \eqref{speed-order}, we arrive at
\begin{align*}
m_Z \leq f^*\pth{\xi, \hat t, - \al \nu}\al \leq f_*\pth{\xi, \hat t, -\be \nu}\be \leq m_Z - \de,
\end{align*}
a contradiction.
Therefore $Z$ and $W$ cannot cross, which in turn implies
that $u$ and $v$ cannot cross and hence
$u \prec v$ in $\cl Q$ w.r.t. $Q$.
\qedhere\end{proof}

\subsection{Well-posedness}
\label{sec:well-posedness}

We will use the comparison principle, Theorem~\ref{th:comparison},
to establish the well-posedness
of the Hele-Shaw-type problem \eqref{hs-f}, Theorem~\ref{th:well-posedness}.
Since Theorem~\ref{th:comparison} requires
strictly separated boundary data,
it does not provide uniqueness of solutions
for general boundary data.
However,
it is possible to prove uniqueness
when the boundary data features some special structure;
see \cites{K03,K07}.
We show the well-posedness theorem for boundary data
strictly increasing in time.

\begin{proof}[Proof of Theorem~\ref{th:well-posedness}]
In this proof we return to the parabolic cylinder
$Q = \Omega \times (0,T]$
introduced in Section~\ref{sec:introduction}.

\medskip
\noindent
\emph{Bounded support.}
First observe that exactly one of $\Omega$ and $\Omega_0$
is bounded and that both $\partial \Omega$
and $\partial \Omega_0$ are bounded.

Let $u$ be a bounded viscosity solution on $Q$
with the correct boundary data.
We will show that
$\cl\Omega(u;Q) \subset B_R(0) \times [0,T]$
for some large $R$.
If $\Omega$ is bounded, this is obvious.
Therefore assume that $\Omega$ is unbounded and therefore
$\Omega_0$ is bounded.
Let $d = \diam \Omega_0 > 0$,
set $K := \sup_Q u < \infty$ (by assumption)
and recall that $M$ is the upper bound on $f$.
We define $\mu := 2 \sqrt{2 n K M T}$
and $R := d + \mu$.
The claim now follows from Corollary~\ref{co:HS-expansion-speed}
applied with $E = \cl{\Omega_0}^c$.
We set $\Sigma := B_R(0) \cap \Omega$, which is clearly bounded.
Note that $\partial \Omega_0 \subset B_R(0)$.

\medskip
\noindent
\emph{Barriers.}
Before proceeding with the proof of existence,
let us first construct barriers for the Perron's method.
We write $\Omega_0^\rho := \Omega_0 + B_\rho(0)$, $\Omega_0^0 := \Omega_0$.
Since $\partial \Omega_0$ is $C^{1,1}$ and compact,
it is a set of a positive reach
and there exists $\rho_0 > 0$ such that $\partial \Omega_0^\rho \in C^{1,1}$
and $\partial \Omega_0^\rho \subset B_R(0)$ for $\rho \in [0, \rho_0]$.

For $\rho \in [0, \rho_0]$, $t \in [0,T]$,
let us define $u_t^\rho \in C(\Omega)$,
such that
$\Delta u_t^\rho = 0$ in $\Omega_0^\rho \cap \Omega$
with boundary data
$u_t^\rho(x) = \psi(x, t)$ on $\partial \Omega$,
$u_t^\rho = 0$ in $\Omega \setminus \Omega_0^\rho$.

Since $\partial \Omega \subset B_R(0)$ and
$\partial \Omega \in C^{0,1}$,
we can extend $\psi$ to $C(\Rn \times [0,T])$
as $\Delta \psi = 0$ in $(B_R(0) \setminus \partial \Omega) \times [0,T]$
and $\psi = 0$ in $B_R^c(0) \times [0,T]$.
Observe that $\psi \in \sol(M, \Sigma \times [0, T])$.

Now we introduce the barriers at the initial time $t = 0$.
By the Hopf's lemma and comparison for harmonic functions,
and the fact that $\psi$ is bounded from above and from zero
on $\partial \Omega \times [0, T]$,
there exists a constant $c > 0$ such that
$0 < c^{-1} \leq \abs{D(u_t^\rho)^+} \leq c$
on
$\partial \Omega_0^\rho$ for all $\rho \in [0, \rho_0/2]$, $t \in [0,T]$.
Therefore we can find $\omega > 0$ large enough
so that $U(x,t) := u_t^{\om^{-1} t}(x)$
is a subsolution $U \in \subs(m/2, \Omega \times (0, \de])$,
and $V(x,t) := u_t^{\om t}(x)$ is a
supersolution $V \in \supers(M, \Omega \times (0, \de])$
for some $\de \in (0, \rho_0 /\omega)$.
Moreover, $U, V \in C(\Omega \times [0,\de])$.
Let us extend $U$ and $V$ by defining $U(x,t) = u_t^{\om^{-1} \de}(x)$
and $V(x,t) = \psi(x,t)$ for $t \in (\de, T]$ (recall the extension
of $\psi$ above).
It is straightforward to show
that $U \in \subs(m/2, Q)$, $V \in \supers(M, Q)$
and that $U$ and $V$
satisfy the boundary condition on $\partial_P Q$
in the correct sense.
That is, $U^* = U_* = V^* = V_* = \psi$ on $\partial \Omega \times [0,T]$
and $U^* = U_* = V^* = V_* = u_0^0$.

\medskip
\noindent
\emph{Uniqueness.}
Let $u$ and $v$ be two solutions with the correct
boundary data.

We first check the value of $u$ and $v$ at $t = 0$.
In general, for any solution $u$,
we have $u^*(\cdot, t)$ is subharmonic in $\Omega$
and $u_*(\cdot, t)$ is superharmonic in $\Omega_t(u_*; Q)$
(see \cite{KP12} for a detailed proof).
Thus a simple argument,
using the facts that a support of a subsolution must expand continuously,
and that a support of a supersolution cannot decrease,
shows that
$u_*(\cdot, 0) = u^*(\cdot, 0) = v_*(\cdot, 0) = v^*(\cdot, 0) = u_0^0$
on $\Omega$.

Now we observe that the support of a supersolution $u$
is strictly increasing
at $t = 0$, i.e.,
$\cl\Omega_0(u_*;\cl Q) = \cl{\Omega \cap \Omega_0}
\subset \Omega_t(u_*;\cl Q)$ for any $t > 0$.
This follows from the comparison with a perturbation of $U$.
Specifically, fix a point $\zeta \in B_R(0) \setminus \cl\Omega$
and define $U_\kappa(x,t) = (U(x,t) + \kappa \abs{x - \zeta}^2
- 4 \kappa R^2)_+$.
Clearly $U_\kappa \rightrightarrows U$ uniformly on $\cl Q$
as $\kappa \to 0$ and $U_\kappa$ is a subbarrier
on $\Omega \times (0,\de]$ for $\kappa > 0$
small enough (recall that $U \in \subs(m/2, Q)$ with $m/2< m$).
Moreover, the support of $U$ strictly increases in the above sense.
Since $U_\kappa \prec u$ on $\partial_P Q$ for any $\kappa > 0$,
we conclude that the support of $u_*$ must strictly increase
by sending $\kappa \to 0$.

Because the comparison principle requires strictly separated boundary
data, we have to perturb the solutions
and create this separation.
We use a perturbation in time.
Such perturbation is more involved
because of the time-dependence of $f(x,t,p)$,
and in that case we will use the nonlinear
perturbation from Appendix~\ref{sec:nonlinear-perturbation}.
Thus let
\begin{compactitem}
\item
$\ta_\eta(t) = t - \eta$ if $f(x,t,p) = f(p)$, or let
\item
$\ta_\eta(t)$ be the function constructed
in Section~\ref{sec:nonlinear-scaling-sub} with $\tau = - \eta$
and $\rho = 0$ if $f(x,t,p) = f(x,t)$.
\end{compactitem}
Observe that, in either case, $\ta_\eta$
is well-defined on $[0,\infty)$, $\ta_\eta' \in (0,1]$
and therefore it is invertible,
$\ta_\eta(0) = -\eta$ and $\ta_\eta(t) \to t$ as $\eta \to 0$,
locally uniformly in $t$.

We define $w(x,t) = u(x,\ta_\eta(t))$.
As we showed above, the support of $v_*$
is strictly increasing at $t = 0$.
Furthermore, the boundary data $\psi$ on $\partial \Omega_0$
is strictly increasing in time
and $\ta_\eta(t) < t$ for all $t \geq 0$.
Since $\ta_\eta^{-1}(0) > 0$,
let us define
$Q_\eta := \Sigma \times (\ta_\eta^{-1}(0), T]$,
where $\Sigma$ was introduced above,
and we have for all small $\eta > 0$
\begin{align*}
w \prec v\quad \text{on } \partial_P Q_\eta.
\end{align*}
The comparison theorem~\ref{th:comparison} then implies
$w \prec v$ on $\cl Q_\eta$,
and thus $u^*(x,\ta_\eta(t)) \leq v_*(x,t)$.
By sending $\eta \to 0$, we conclude
$u_* \leq (u^*)_* \leq v_*$ on $Q$.
Futhermore, clearly
\begin{align*}
u^*(x,t) \leq v_*(x, \ta_\eta^{-1}(t))\quad
    \text{on } \cl Q \times [0, \ta_\eta^{-1}(T)],
\end{align*}
which yields $u^* \leq (v_*)^* \leq v^*$ on $Q$ after sending $\eta \to 0$.
The argument repeated with $u$ and $v$ swapped
implies that the solution is unique.

\medskip
\noindent
\emph{Existence.}
We use the classical Perron-Ishii method \cite{Ishii87}; see also \cites{CIL,K03}.
Let
\begin{align*}
u &= \sup \set{w \in \subs(f, Q): w \leq V}, &
v &= \inf \set{w \in \supers(f,Q): w \geq U}.
\end{align*}
Clearly $u \geq U$ and $v \leq V$.
Since $(V^*)_* = V$
and $(U_*)^* = U$,
a standard argument yields that $u, u^*,v, v_* \in \sol(f, Q)$.
Moreover,
because of the continuity of $U$ and $V$ at $\partial_P Q$
we conclude that all these solutions have the correct boundary data.
Therefore, by uniqueness,
if $w \in \sol(f, Q)$ with the correct boundary data, we have
$w_*=u_* = v_* = (u^*)_*$ and $w^* = u^* = v^* = (v_*)^*$, and thus
\[
(w^*)_* = w_*, \qquad (w_*)^* = w^*,
\]
which completes the proof.
\qedhere\end{proof}

The last step of the previous proof provides the following regularity.

\begin{corollary}
\label{co:regularity}
Let $u$ be the unique solution of the problem \eqref{hs-f}
with data satisfying the assumptions of Theorem~\ref{th:well-posedness}.
Then
\begin{equation}
\label{u-upper-lower}
(u_*)^* = u^* \qquad \text{and} \qquad (u^*)_* = u_*.
\end{equation}
and
\begin{equation}
\cl\Omega(u_*;Q) = \cl\Omega(u;Q)=\cl\Omega(u^*;Q).
\end{equation}
\end{corollary}

\begin{proof}
As was indicated above, \eqref{u-upper-lower}
was shown in the last step of the proof of Theorem~\ref{th:well-posedness}.

Let us point out that $\cl\Omega(u) = \cl\Omega(u^*)$ is always true.
It is enough to prove that $\cl\Omega(u^*) \subset \cl\Omega(u_*)$
since the other direction is obvious.
Suppose that $(\xi, \si) \in \cl\Omega(u^*) \setminus \cl\Omega(u_*)$.
Then there exists $\de > 0$ such that
$B_\de(\xi,\si) \cap \cl\Omega(u_*) = \emptyset$
and therefore $u_* = 0$ on $B_\de(\xi,\si)$.
However, \eqref{u-upper-lower} implies that
$u^* = (u_*)^* = 0$ on $B_\de(\xi, \si)$,
a contradiction with the choice of $(\xi, \si)$.
\qedhere\end{proof}

\section{Identification of the limit problem}
\label{sec:limit-problem}

This section is devoted to the key step in the proof of Theorem~\ref{th:intro-homogenization},
the identification of a viable candidate for the homogenized velocity $r(Du)$.
The results are quite technical, an unfortunate consequence of the employed arguments, and therefore we first give a brief overview of the content of this section.
We motivate the approach by considering the one-dimensional case in Section~\ref{sec:homogenizationin1D}.
The numerical results there suggest that the form of the homogenized velocity $r(Du)$ is qualitatively different from the time-independent case.
In particular, it indicates that a simple closed expression for $r(Du)$ is not available.

Therefore we proceed with the identification of $r(Du)$ by other means.
We follow the idea from \cite{CSW,K07} of using a class of certain obstacle problems.
They are introduced in Section~\ref{sec:obstacle-problem},
where for each scale $\e >0$, gradient $q$ and candidate $r$ of the free boundary velocity we set up a domain, an obstacle and define a subsolution and a supersolution of an obstacle problem.
The main feature of the solutions of the obstacle problems is their natural monotonicity with respect to a family of scalings and translations, Section~\ref{sec:monotonicity-obstacle}.
To quantify the viability of the homogenization velocity candidate, we introduce a new quantity,  called \emph{flatness}, in Section~\ref{sec:flatness}, and derive its basic properties.
It measures how far the obstacle problem solutions detach from the obstacle, that is, how ``flat'' they are.

The following three sections then cover three important results that build on the definitions above.
First, the \emph{local comparison principle}, Section~\ref{sec:local-comparison-principle}, improves on the theorem of the same name from \cite{K08}, and allows us to locally compare the subsolutions of the obstacle problems with the supersolutions, away from the domain boundary, even after the boundary values are no longer ordered, provided that the flatness of those solutions is bounded with a certain rate.
Section~\ref{sec:cone-flatness} then presents the new \emph{cone flatness} property, which is a consequence of our particular choice of the domain for the obstacle problems, and which guarantees that the free boundary of a given obstacle problem solution lies in between two nearby cones whose distance can be estimated.
Finally, the \emph{detachment lemma} of Section~\ref{sec:detachment}, a direct consequence of the cone flatness property, guarantees that if the flatness of a given obstacle problem solution exceeds a certain value, the solution is separated from a significant portion of the obstacle.

These three results motivate the introduction of the two candidates for the homogenized velocity in Section~\ref{sec:homogenized-velocity}.
We then deduce the semi-continuity of the candidates and their coincidence up to the points of discontinuity, as explained in Section~\ref{sec:semi-continuity}.
The main tools are the natural monotonicity of the obstacle problem, the local comparison principle and the detachment lemma.
In Section~\ref{sec:time-independent} we make a brief detour to show that if the scaling in time is available, as is the case in the time-independent medium, we can prove that the homogenized velocity is in fact continuous and positively one-homogeneous, which recovers the results of \cite{K08} via our methods.

\subsection{Homogenization in one dimension}
\label{sec:homogenizationin1D}

Let us first explore the homogenization of \eqref{HSt} in one dimension.
For simplicity,
we take $\Omega = (0, \infty)$ and $\Omega_0 = (-\infty, x_0)$
for some $x_0 > 0$.
Furthermore, let $g(x,t): \R \times \R \to [m,M]$
be a positive, $L$-Lipschitz, $\Z^2$-periodic function.
Finally, we assume that the boundary data on $\partial \Omega = \set{0}$
are given by a function $\psi(0,t) = \psi(t) > 0$,
$\psi \in C([0,\infty))$.

Let $u$ be the unique solution of \eqref{HSt} with the above data.
We observe that, in this simple setting, $\Omega(u; \Omega) = (0, x(t))$,
where $x(t)$ is the position of the free boundary
$\partial \Omega(u)$ at time $t$,
i.e., $\partial \Omega(u) = \set{(x(t),t) : t \geq 0}$.
Therefore the normal velocity is simply $V_\nu = x'$.
Moreover, the harmonic function $u(\cdot, t)$
must be linear on $(0,x(t))$ with the slope $\abs{Du(x(t), t)} = \psi(t)/x(t)$.
Therefore, for every fixed $\e > 0$,
the free boundary condition in \eqref{HSt}
is equivalent to a simple ODE of the form
\begin{align}
\label{eq:boundaryode}
\begin{cases}
x_\e'(t) = g\pth{\frac{x_\e(t)}{\e}, \frac{t}{\e}} \frac{\psi(t)}{x_\e(t)},\\
x_\e(0) = x_0.
\end{cases}
\end{align}
Due to the ODE homogenization result of \cite{IM10} (see also \cite{Piccinini1}),
$x_\e \to x$ locally uniformly,
where $x(t)$ is the solution of
\begin{align*}
\left\{
\begin{aligned}
x'(t) &= f (x(t), t),\\
x(0) &= x_0.
\end{aligned}
\right.
\end{align*}
The arguments in this paper then imply that $f$ is of the form
\[
f(x(t), t) = r(\psi(t)/x(t)).
\]
The function $r(q)$ can be found numerically by scaling \eqref{eq:boundaryode} and thus solving
\begin{align}\label{eq:boundaryodescaled}
\left\{
\begin{aligned}
x'(t) &= q g\pth{x(t), t},\\
x(0) &= x_0,
\end{aligned}
\right.
\end{align}
in some $t \in [0, T]$ for $T$ large. Then $r(q)$ can be approximated by $x(T)/T$. The limit $\lim_{T \to \infty} \frac{x(T)}{T}$ is independent of $x_0$ since we can squeeze any solution for $x_0 \neq 0$ between $x(t)$ and $x(t) + 1$ using the comparison principle, and use the periodicity of $g$.

Consider the medium
\begin{align*}
g(x,t) = \sin^2 ( \pi(x-t)) + 1.
\end{align*}
Then for $q \in [1/2, 1]$,
$x(t) = t + \pi^{-1} \arccos \sqrt{q^{-1} - 1}$
is a solution of \eqref{eq:boundaryodescaled}.
Thus we see that $r(q) = 1$ for $q \in [1/2, 1]$.
If the boundary propagates in the opposite direction,
i.e., when $g = \sin^2(\pi(-x-t)) + 1$,
this effect does not appear and $r(q)$ is strictly increasing.
The situation can be even more complicated;
see Figure~\ref{fig:pinningintervalmultiple} for an example of this phenomenon.
\begin{figure}[t]
\centering
\fig{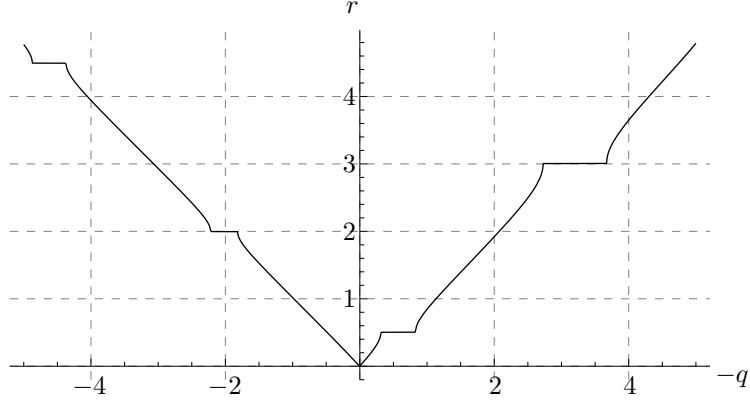}{4.25in}
\caption{The graph of $r(q)$ for $g(x,t) = \sin(2 \pi (x-3t)) \sin(2 \pi (2t + x)) + 11/10$}
\label{fig:pinningintervalmultiple}
\end{figure}

\subsection{Obstacle problem}
\label{sec:obstacle-problem}

This section is devoted to the introduction of the class of obstacle problems whose solutions are the key tool for the identification of the homogenized free boundary velocity $r(Du)$ in Theorem~\ref{th:intro-homogenization}.

The obstacle problems lead to the following class of solutions:
for every scale $\e >0$, gradient $q \in \Rn \setminus \set0$ and candidate $r > 0$ of the homogenized velocity,
we define
the solutions of an obstacle problem on $\Rn\times\R$ as
\begin{subequations}
\label{obstaclesolution}
\begin{align}
\label{eq:obstaclesupersolution}
\uu_\eqr &= \pth{\inf\set{
    u :
        u \geq P_{q,r} \text{ on $\Rn\times\R$},
        u \in \supers(g^\e,Q_q)
        }}_*,\\
\ou_\eqr &= \pth{\sup\set{
    u :
        u \leq P_{q,r} \text{ on $\Rn\times\R$, }
        u \in \subs(g^\e, Q_q)
        }}^*.
\end{align}
\end{subequations}
As a mnemonics, notice that $u$ is \emph{above} the bar in $\uu_\eqr$ because $\uu_\eqr$ is \emph{above} the obstacle.
The obstacles $\Pqr$ are the planar traveling wave solutions of the homogenized Hele-Shaw problem, and they are introduced in Section~\ref{sec:planar-solutions} below.
The domain of the obstacle problem, $Q_q$, requires a special attention.
It is constructed in Section~\ref{sec:domain-geometry}.
After the definition has been clarified, the basic properties of the solutions $\uu_\eqr$ and $\ou_\eqr$ are discussed in Section~\ref{sec:obstaclesolutions-properties}.

\subsubsection{Planar solutions}
\label{sec:planar-solutions}

The homogenized Hele-Shaw problem with boundary velocity law
$V_\nu = r(Du)$
admits simple traveling wave solutions.
We call them planar solutions since each
is a positive part of a linear function
with gradient $q \in \Rn\setminus\set0$
moving in the direction $\nu := -q/\abs q$
with the normal velocity $V_\nu = r(q)$.
These solutions serve as the obstacles for the obstacle
problem \eqref{obstaclesolution} that will be used to find a candidate of
the homogenized free boundary velocity $r(q)$.

For any $q \in \Rn\setminus\set0$ and $r > 0$ we define
\begin{align*}
P_{q,r}(x,t) := (\abs{q} r t + x \cdot q)_+
= \abs q \pth{rt - x\cdot \nu}_+,
    \qquad \nu := \frac{-q}{\abs{q}},
\end{align*}
where $(s)_+ := \max(s, 0)$ is the positive part.
We interpret $r$ as the velocity of the free boundary
\begin{align*}
\Gamma(\Pqr) = \set{(x,t) : rt = x\cdot \nu}
\end{align*}
and $q \neq 0$ as the gradient $q = D \Pqr$.
Thus $\nu$ is the unit outer normal
vector to the free boundary at every time.

The following observation is a
trivial consequence of the nondegeneracy assumption
\eqref{g-bound}.
\begin{proposition}
\label{pr:planar-solution-range}
$\Pqr \in \subs(g^\e)$ if
\begin{align*}
r &\leq m \abs{q}
\intertext{and $\Pqr \in \supers(g^\e)$ if}
r &\geq M \abs{q}.
\end{align*}
\end{proposition}
\begin{proof}
Use \eqref{g-bound}.
\qedhere\end{proof}

In view of Proposition~\ref{pr:planar-solution-range},
we will often impose the following restriction on the values
of $r$ and $q$ throughout the paper:
\begin{align}
\label{rqrestriction}
q \neq 0, \qquad m \leq \frac{r}{\abs{q}} \leq M.
\end{align}

\begin{proposition}
\label{pr:obstacle-ordering}
For given $q \neq 0$ and $r > 0$ we have
\begin{align*}
P_{q,r} (x - y, t - \tau) &\leq P_{q,r}(x,t) \text{ for all $x, t$}
    &\quad& \text{if and only if}\quad y \cdot \nu \leq r \tau,
\intertext{and}
P_{q,r} (x - y, t - \tau) &\geq P_{q,r}(x,t) \text{ for all $x, t$}
    &&\text{if and only if}\quad y \cdot \nu \geq r \tau.
\end{align*}
\end{proposition}

\begin{proof}
The statement is obvious from the identity
\begin{align*}
\abs{q}(r (t - \tau) - (x-y) \cdot \nu)
=  \abs q(r t - x \cdot \nu) + \abs{q} (y \cdot \nu - r \tau),
\end{align*}
and from the definition of $\Pqr$.
To prove the only-if direction, consider $x = 0$ and $t$ sufficiently large so that the first two terms above are positive.
\qedhere\end{proof}

\subsubsection{Domain}
\label{sec:domain-geometry}
\begin{figure}
\centering
\fig{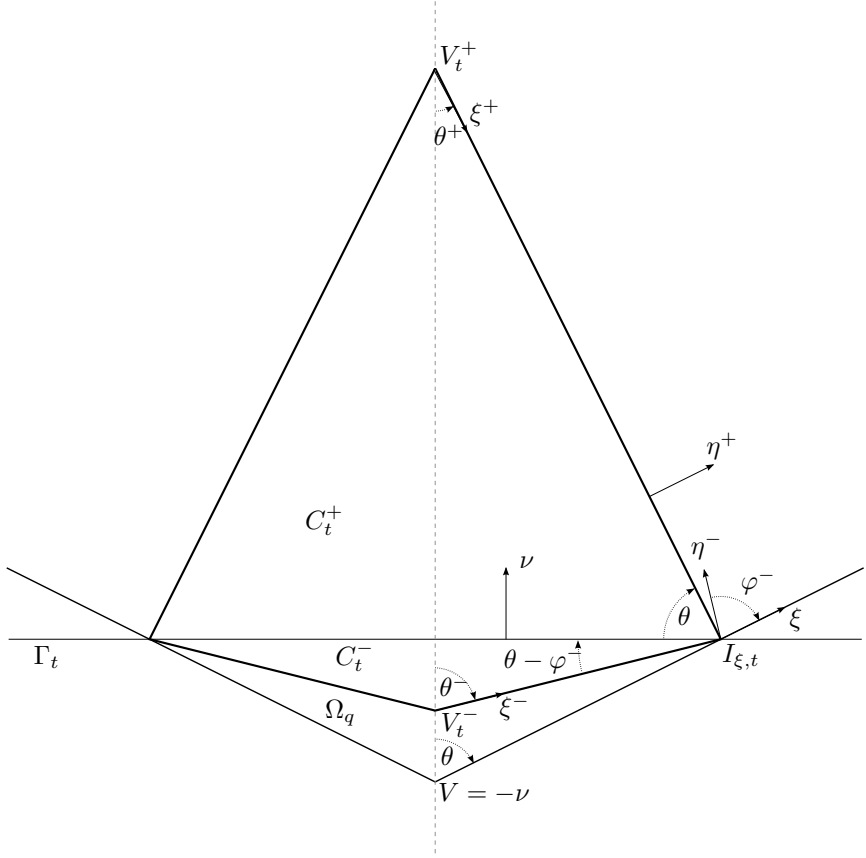}{4.5in}
\caption{The cone $\Omega_q$ and the support of
the sub/supersolution boundaries at a time $t$}
\label{fig:obstacledomain}
\end{figure}

In this section
we construct a class of space-time cylinders $Q_q$
that serve as the domains for the obstacle problem
introduced in \eqref{obstaclesolution}.
The base of the cylinder is chosen to be a cone with a prescribed opening angle and an axis in the direction that coincides with the normal of the free boundary of the obstacle.
This particular choice of geometry will let us control
how fast the solution of the obstacle problem detaches
from the obstacle at the lateral boundary of the domain $Q_q$.
It is achieved by a comparison with the class of planar subsolutions and supersolutions $R^\pm_\xi$
constructed below.
In fact, this is the main motivation for this technical construction.
With these barriers at our disposal, we can show that the free boundaries of the solutions $\uu_\eqr$ and $\ou_\eqr$ must lie inside an intersection of certain cones; see Proposition~\ref{pr:obstacle-sols}.
The ability to control the boundary behavior
will be necessary for the extension of the monotonicity property
of the obstacle problem beyond the straightforward
inclusion of domains.
One of the main consequences of this extra monotonicity is
the cone flatness property, Proposition~\ref{pr:cone-flatness}.

Throughout the rest of this section,
we fix $q\in\Rn\setminus\set0$ and $r>0$
and let $\nu = \frac{-q}{\abs{q}}$.
For simplicity, we shall denote
\begin{align*}
\Gamma_t = \Gamma_t(\Pqr).
\end{align*}

Since the homogenization result is trivial
when $m = M$,
we may assume throughout the rest of the paper that $0 < m < M$.
We define angles $\ta, \ta^+ \in (0,\frac\pi2)$,
and $\ta^- \in (\ta, \frac\pi2)$ as
\begin{subequations}
\label{eq:choiceoftas}
\begin{align}
\label{eq:choiceofta}
\ta &= \arccos \sqrt{\frac{m}{M}}\\
\label{eq:choiceoftaplus}
\ta^+ &= \ha\pi - \ta\\
\label{eq:choiceoftaminus}
\ta^-&= \ha\pi + \ta - \vp^-, && \text{where } \vp^- := \arccos\frac{m}{M}.
\end{align}
\end{subequations}
Note that indeed $\ta^- \in (\ta, \frac\pi2)$
since $\vp^- \in (\ta, \frac\pi2)$
by $\cos \vp^- = \frac{m}{M} < \sqrt{\frac{m}{M}} = \cos \ta$.

Let ${\rm Cone}_{p, \ta}(x)$ be the open cone with axis
in the direction of $p\in\Rn$, opening angle $2\ta$
and vertex $x$, that is,
\begin{align*}
{\rm Cone}_{p, \ta}(x)
&:= \set{y : (y - x) \cdot p > \abs{y - x} \abs{p} \cos \ta}.
\end{align*}

Angles $\ta$, $\ta^+$ and $\ta^-$ then define cones $\Omega_q$,
$C_t^+$ and $C_t^-$ for every $t \geq 0$,
\begin{align*}
\Omega_q &:= {\rm Cone}_{\nu, \ta}(V),\\
C_t^- &:= {\rm Cone}_{\nu, \ta^-}(V^-_t),\\
C_t^+ &:= {\rm Cone}_{-\nu, \ta^+}(V^+_t),
\end{align*}
where $V = -\nu$ and the vertices $V^\pm_t$ are uniquely determined by
the equality
\begin{align}
\label{same-base}
\Omega_q \cap \Gamma_t = C_t^- \cap \Gamma_t = C_t^+ \cap \Gamma_t,
\end{align}
see Figure~\ref{fig:obstacledomain}.
Finally, we define the space-time cylinder
\begin{align*}
Q_q := \Omega_q \times (0, \infty).
\end{align*}

\begin{remark}
\label{re:speedVplust}
Since the positions of the vertices $V^\pm_t$ of the cones $C^\pm_t$
clearly depend linearly on time,
we can extend their definition to all $t \in \R$
as
\[
V^\pm_t = V^\pm_0 + r^\pm_V t \nu \qquad t \in \R,
\]
where $r^\pm_V$ are constants and we call them the velocities of $V^\pm_t$.
Note that $V = V^+_{-r^{-1}} = V^-_{-r^{-1}} = - \nu$
due to \eqref{same-base} and the fact that $V = -\nu \in \Gamma_t(\Pqr)$
at $t = -\ov r$.
The velocities of $r^\pm_V$ can be found explicitly.
Since the cones share a base, \eqref{same-base},
we infer that (see Figure~\ref{fig:obstacledomain})
\[(r_V^+ - r)\tan \ta^+ = r \tan \ta = (r - r_V^-) \tan \ta^-.\]
Since $\ta^+ = \frac{\pi}{2} - \ta$,
$\tan \ta^+ = \frac{1}{\tan \ta}$ holds and we have
\begin{align*}
r_V^+ = r(1 + \tan^2 \ta) = \frac{r}{\cos^2 \ta} = \frac{M}{m}r > r.
\end{align*}
We can similarly express the velocity $r^-_V$ as
\begin{align*}
r_V^- = \pth{1 - \frac{\tan \ta}{\tan \ta^-}} r = c_{\frac Mm} r \in (0, r).
\end{align*}
\end{remark}

The reason for this choice of domain $Q_q$ is the result in Lemma~\ref{le:matchingplanarsolutions} below.
But we begin with the following geometrical result.
Let us introduce the set of all directions
of rays in $\partial \Omega_q$:
\begin{align}\label{eq:Xi}
    \Xi = \set{\xi \in \Rd : \quad \abs{\xi} = 1, \quad \xi \cdot \nu = \cos \ta}.
\end{align}

\begin{proposition}
\label{pr:cone-planes}
For any given $x \in \partial \Omega_q$, $x \neq V$,
there exists unique $\xi \in \Xi$
such that $x \in L_\xi \subset \partial \Omega_q$,
where $L_\xi$ is the ray
\begin{align}\label{eq:Lxi}
    L_\xi := \set{\si \xi +V: \si \geq 0}.
\end{align}
Furthermore, there exist unique $\xi^-$ and $\xi^+$
such that $\abs{\xi^\pm} = 1$,
\begin{align*}
L^\pm_{\xi,t}
    := \set{\si \xi^\pm + V^\pm_t: \si \geq 0} \subset \partial C_t^\pm
\end{align*}
and
\begin{align*}
L^\pm_{\xi,t} \cap L_\xi = \set{I_{\xi,t}} \qquad \text{for all $t \geq 0$}
\end{align*}
for some point $I_{\xi,t}$.
Let $\eta^\pm_\xi$ be the unit normal to $\partial C_t^\pm$
on $L^\pm_{\xi,t}$ with $\eta^\pm \cdot \nu > 0$.
There exist unique constants $\mu^\pm$, $r^\pm$ and $T^\pm$
depending only on $(q, r, m, M)$,
and independent of $\xi$, such that
\begin{align}
\label{Rxi-touches-Ct}
\Gamma_t(R^\pm_\xi) \cap \partial C^\pm_t &= L^\pm_{\xi,t}
&& \text{for all $t \geq 0$,}\\
\label{eq:planarsolutionsmatch}
R^\pm_\xi &= P_{q,r} && \text{on $L_\xi \times [0, \infty)$}
\end{align}
where
\begin{align}
\label{eq:Rpmxi}
R^\pm_\xi (x,t):= P_{-\mu^\pm \eta_\xi^\pm, r^\pm}(x, t - T^\pm).
\end{align}
\end{proposition}

\begin{proof}
The existence of $\xi, \xi^+, \xi^-$ is
straightforward.

$\mu^\pm$ and $r^\pm$ can be expressed using the following geometric
considerations.

By definition,
\begin{align*}
\cos \ta = \xi \cdot \nu.
\end{align*}
We introduce $\vp^\pm \in (0,\pi/2)$ via $\cos \vp^\pm = \xi \cdot \eta^\pm$ and note that
\begin{align}
\label{eq:taandvp}
\ta^+ &= \ha\pi + \vp^+ - \ta, & \ta^- &= \ha{\pi} + \ta - \vp^-.
\end{align}

Observe that the point $I_{\xi,t} \in L_\xi \cap \Gamma_t$ moves in the direction $\xi$ along $L_\xi$ with the velocity $r_I$ given as
\begin{align*}
r_I = \frac{r}{\cos \ta}.
\end{align*}

Since $R^\pm_\xi$ propagates in the direction $\eta^\pm$,
$\abs{\eta^\pm} = 1$,
and
\[
\set{I_{\xi,t}} = L^\pm_{\xi,t} \cap L_\xi
= \Gamma_t(R^\pm_\xi) \cap \partial C_t^\pm \cap L^\xi
\]
by condition \eqref{Rxi-touches-Ct},
we must have $I_{\xi,t} \in \Gamma_t(R^\pm_\xi)$
for all $t \geq 0$
and therefore
\begin{align}
\label{eq:rpm}
r^\pm = r_I (\xi \cdot \eta^\pm)
    =r_I \cos \vp^\pm =  r \frac{\cos \vp^\pm}{\cos \ta}.
\end{align}

The slope $\mu_L$ of $P_{q,r}$ on $L_\xi$ is
\begin{align*}
\mu_L = \abs{q} \cos \ta.
\end{align*}
In \eqref{eq:planarsolutionsmatch}
we require that this is also the slope of $R^\pm_\xi$ on $L_\xi$, i.e.,
\begin{align}
\label{eq:mupm}
\mu^\pm =  \frac{\mu_L}{\cos \vp^\pm} = \frac{\abs{q} \cos \ta}{\cos \vp^\pm}.
\end{align}

Finally, the constants $T^\pm$ are then fixed
by requiring $V^\pm_{T^\pm} = 0$,
and therefore, by Remark~\ref{re:speedVplust},
\[
T^\pm = \ov{r^\pm_V} - \ov r.
\]
It is straightforward to check that
$R^\pm_\xi$ defined with such unique choice of
$\mu^\pm$, $r^\pm$ and $T^\pm$
satisfy \eqref{Rxi-touches-Ct} and \eqref{eq:planarsolutionsmatch}.
\qedhere\end{proof}

We complete the family of $R_\xi^\pm$ defined in \eqref{eq:Rpmxi}
by introducing two special planar functions:
\begin{align}\label{eq:Rpm0}
\begin{aligned}
R^+_0(x,t) &= P_{q,\max(M\abs{q},r)}(x,t),\\
R^-_0(x,t) &= P_{q,\min(m\abs{q},r)}(x,t).
\end{aligned}
\end{align}

\begin{corollary}
\label{co:values-Rxi}
We have
\begin{align*}
R^-_\xi &\leq P_{q,r} \leq R^+_\xi &
&\text{on $Q_q$ for all $\xi \in \Xi \cup \set0$},
\end{align*}
with equality
\begin{align*}
R^\pm_\xi &= P_{q,r} &
&\text{on $L_\xi \times [0,\infty)\subset \partial Q_q$
for $\xi \in \Xi$},\\
R^\pm_0 &= P_{q,r} && \text{at $t = 0$},
\end{align*}
where $L_\xi$ was introduced in \eqref{eq:Lxi}.

Moreover,
\begin{align}
\label{Rxi-outside}
\sup_{\xi\in\Xi} R^-_\xi \geq \Pqr \geq \inf_{\xi\in\Xi} R^+_\xi
\qquad \text{on $\Omega_q^c \times [0,\infty)$.}
\end{align}

Finally,
\begin{align}
\label{Rxi-cone}
\bigcap_{\xi\in\Xi} \Omega_t(R_\xi^+) &= C^+_t, &
\bigcap_{\xi \in\Xi} \Omega_t^c(R_\xi^-) &= \overline{C^-_t}
\qquad \text{for all $t\geq 0$.}
\end{align}
\end{corollary}

\begin{proof}
Since the statement is obvious for $\xi = 0$ by the definition
of $R^\pm_0$ in \eqref{eq:Rpm0},
we fix $\xi \in \Xi$.
Let $\eta$ be the unit normal vector to $\partial \Omega_q$
on $L_\xi \subset \partial \Omega_q$.
Clearly $\eta \cdot \xi = 0$.
Define the half-space $H:= \set{x : (x - V) \cdot \eta > 0}$
and observe that $\Omega_q \subset H$
and $\partial \Omega_q \cap \partial H = L_\xi$.

Let us first show that $\Pqr = R^\pm_\xi$ on $\partial H$.
We note that if $x \in \partial H$ we can express
$x - V= \si \xi + y$, where $\si \in \R$ and
$y \cdot \xi = y \cdot \eta = 0$.
Since $\set{\nu, \eta^+, \eta^-} \subset \Span \set{\xi, \eta}$,
and $\Pqr = R^\pm_\xi$ on $L_\xi$ by Proposition~\ref{pr:cone-planes},
linearity implies
\begin{align*}
\Pqr(x,t) &= \Pqr(V + \si\xi + y,t) = \Pqr(V+ \si\xi,t)
\\&=
R^\pm_\xi(V+\si\xi,t) = R^\pm_\xi(x,t) \qquad (x,t)\in \partial H \times \R.
\end{align*}
Additionally, by linearity, since $V^\pm_t \in H$ for all $t\geq 0$,
and
\[V^+_t \in \cl\Omega_t(R^+_\xi) \setminus \cl\Omega_t(\Pqr),
\qquad
V^-_t \in \cl\Omega_t(\Pqr) \setminus \cl\Omega_t(R^-_\xi),
\]
we must have
\[
R^-_\xi \leq \Pqr \leq R^+_\xi \qquad
\text{on $H \times [0,\infty) \supset Q_q$}
\]
and
\[
R^-_\xi \geq \Pqr \geq R^+_\xi \qquad
\text{on $H^c \times [0,\infty)$}.
\]
Since for every $(x,t) \in \Omega_q^c \times [0,\infty)$
there exists $\xi \in \Xi$ such that $x \in H^c$,
\eqref{Rxi-outside} is also proved.

To show one inclusion ($\supset$) of \eqref{Rxi-cone},
we recall that by \eqref{Rxi-touches-Ct} and the choice $\eta^\pm \cdot \nu > 0$,
we clearly have
\[
C^+_t \subset \Omega_t(R^+_\xi), \qquad
\cl{C^-_t} \subset \Omega^c_t(R^-_\xi), \qquad t \geq 0.
\]

To show the other inclusion ($\subset$) of \eqref{Rxi-cone},
we only need to show that for any $x \in \Rn$
there exists $\xi \in \Xi$ that generates
$\xi^\pm, \eta^\pm$ as in Proposition~\ref{pr:cone-planes}
and $x$
can be expressed as
\begin{align}
\label{express-x}
x = \al^\pm \xi^\pm + \be^\pm\eta^\pm + V^\pm_t,
\end{align}
and therefore clearly for such $\xi$ we have
\begin{align*}
x \notin C^+_t \quad \Leftrightarrow \quad \be^+ \geq 0
\quad \Leftrightarrow \quad R^+_\xi(x,t) = 0
\quad \Leftrightarrow \quad x \notin \Omega_t(R^+_\xi),\\
x \notin \cl{C^-_t} \quad \Leftrightarrow \quad \be^- < 0
\quad \Leftrightarrow \quad R^-_\xi(x,t) > 0
\quad \Leftrightarrow \quad x \notin \Omega^c_t(R^-_\xi).
\end{align*}
To find such $\xi$ for a given $x \in \Rn$,
first observe that there exists $y \in \partial \Omega_q$
such that $x = y + a \nu$ for some $a \in \Rn$.
Since $y \in \partial \Omega_q$, there exists $\xi \in \Xi$
such that $y \in L_\xi$
and therefore $x = \si \xi + a \nu + V$.
Finally, we can write $x$ as \eqref{express-x}
since $\set{\xi, \nu, V^\pm_t - V} \subset \Span \set{\xi^\pm, \eta^\pm}$.
This concludes the proof.
\qedhere\end{proof}

The following geometrical observation,
which shows a connection between the cones $C^\pm_t$
and the planar functions $R^\pm_\xi$,
will be useful later
in deriving monotonicity properties of the solutions
of the obstacle problems.
\begin{proposition}
\label{pr:cone-plus-generator}
For fixed $a > 0$, the following are equivalent
(all with superscript either $+$ or $-$):
\begin{enumerate}
\item
$\pm a R_\xi^\pm(a^{-1} x,a^{-1} t)
    \leq \pm R_\xi^\pm(x - y, t - \tau)$ for all $\xi \in \Xi$,
$(x, t) \in \Rn \times \R$;
\item
    $a C^\pm_{a^{-1} t} \subset C^\pm_{t - \tau} + y$ for some $t \in \R$;
\item
    $y \in \cl{\cone_{\pm\nu,\ta^\pm}(r_V^\pm \tau \nu + (a-1) V_0^\pm)}$
    where $r_V^\pm$ were introduced in Remark~\ref{re:speedVplust}.
\end{enumerate}
\end{proposition}

\begin{proof}
Clearly $a R^\pm_\xi(a^{-1}x, a^{-1}t)$ is just a translation
of $R^\pm_\xi(x,t)$ by scaling invariance,
therefore we can compare the planar solutions
in (a)
by comparing their supports.
In fact,
we have
\[aV^\pm_{a^{-1}t} \in \Gamma_t(a R^\pm_\xi(a^{-1}\cdot, a^{-1}\cdot))\]
by construction
and
\[
\Omega_t(R^\pm_\xi(\cdot - y, \cdot - \tau)) = \Omega_{t-\tau}(R^\pm_\xi) + y,
\]
for all $\xi \in \Xi$.
Therefore by Corollary~\ref{co:values-Rxi}
it is straightforward that (a) is equivalent to
\begin{align*}
aV^\pm_{a^{-1}t} \in \cl{C^\pm_{t-\tau}} + y \qquad \text{for some $t \in \R$.}
\end{align*}
This is clearly equivalent to (b)
since $aV^\pm_{a^{-1}t}$ is the vertex of $aC^\pm_{a^{-1}t}$.
And it is also equivalent to (c) if we take $t = 0$
and compute
\begin{align}
\label{vertex-inclusion}
a V_0^\pm \in \cl{C^\pm_{-\tau}} + y = \cl{\cone_{\mp\nu,\ta^\pm}(V^\pm_{-\tau} + y)}.
\end{align}
Since $V^\pm_{-\tau} = V_0^\pm - r_V^\pm \tau \nu$
by Remark~\ref{re:speedVplust},
we can rewrite \eqref{vertex-inclusion} as (c).
\qedhere\end{proof}

The main motivation for such an involved choice of the domain
for the obstacle problem is the following observation
on the properties of functions $R^\pm_\xi$.

\begin{lemma}
\label{le:matchingplanarsolutions}
Suppose that $q$ and $r$ satisfy the condition \eqref{rqrestriction}.
Then the choice of the angles $\ta, \ta^\pm$ in \eqref{eq:choiceoftas},
depending only on $\frac{m}{M}$,
guarantees that
\begin{align*}
R^+_\xi &\in \supers(M) \subset \supers(g^\e)
&&\text{and} & R^-_\xi &\in \subs(m)\subset\subs(g^\e) &
&\text{for all $\xi \in \Xi \cup \set0$},
\end{align*}
where $R^\pm_\xi$ were introduced in \eqref{eq:Rpmxi}
and \eqref{eq:Rpm0}.
\end{lemma}
\begin{proof}
The statement is obvious for $R_0^\pm$, by definition.
For any $\xi \in \Xi$, by Proposition~\ref{pr:planar-solution-range},
we only need to verify $r^+ \geq M \mu^+$ and $r_- \leq m \mu^-$.
Recall that $q$ and $r$ satisfy \eqref{rqrestriction}.

We first observe that $\cos \vp^+ = 1$ by \eqref{eq:taandvp} and
\eqref{eq:choiceoftaplus}.
Then, in order,
\eqref{eq:rpm}, \eqref{eq:mupm}, \eqref{eq:choiceofta},
and finally \eqref{rqrestriction} lead to
\begin{align*}
\frac{r^+}{\mu^+}
    = \frac{r}{\abs{q}} \frac{\cos^2 \vp^+}{\cos^2 \ta}
    = \frac{r}{\abs{q}} \frac{M}{m} \geq M,
\end{align*}
which verifies that $R_\xi^+ \in \supers(M)$.

Similarly, \eqref{eq:rpm}, \eqref{eq:mupm}, \eqref{eq:choiceofta}, \eqref{eq:choiceoftaminus}
and \eqref{rqrestriction} yield
\begin{align*}
\frac{r^-}{\mu^-}
    = \frac{r}{\abs{q}} \frac{\cos^2 \vp^-}{\cos^2 \ta}
    = \frac{r}{\abs{q}} \frac{m}{M} \leq m,
\end{align*}
which verifies that $R_\xi^- \in \subs(m)$.
\qedhere\end{proof}

\subsubsection{Properties of the obstacle solutions}
\label{sec:obstaclesolutions-properties}

The basic properties of the solutions $\uu_\eqr$ and $\ou_\eqr$ of the obstacle problem introduced in \eqref{obstaclesolution} are
gathered in the following proposition.

\begin{proposition}
\label{pr:obstacle-sols}
For any $q \in \Rd\setminus\set0$, $r > 0$ satisfying \eqref{rqrestriction} and any $\e > 0$, the following statements apply.
\begin{enumerate}
\item
$\uu_\eqr \in \supers(g^\e, Q_q)$
and
$\uu_\eqr \in \subs(g^\e, Q_q \setminus (\Gamma(\uu_\eqr)\cap \Gamma_{q,r}))$;
\item
$\ou_\eqr \in \subs(g^\e, Q_q)$
and
$\ou_\eqr \in \supers(g^\e, Q_q \setminus (\Gamma(\ou_\eqr)\cap \Gamma_{q,r}))$;
\item
$(\uu_\eqr)^* \geq \uu_\eqr \geq P_{q,r}$
and
$(\ou_\eqr)_* \leq \ou_\eqr \leq P_{q,r}$ in $\Rn\times\R$,
with equality on $(\Rn\times\R)\setminus Q_q$;
\item
$\Omega_t(\uu_\eqr) \subset C^+_t \cup \Omega_t(\Pqr)$
and $\Omega_t^c(\ou_\eqr) \subset \cl{C^-_t} \cup \Omega^c_t(\Pqr)$;

\item
    $\uu_\eqr \in \subs(M, \Rn\times\R)$
    and
    $\ou_\eqr \in \supers(m,\Rn\times\R)$;
\item
$\uu_\eqr \leq \inf_{\xi\in\Xi\cup\set0} R^+_\xi$
and $\lu_\eqr \geq \sup_{\xi\in\Xi\cup\set0} R^-_\xi$
on $Q_q$.
\end{enumerate}
\end{proposition}

\begin{proof}
(a) and (b) are standard; see \cite[Lemma~4]{K07}.

To prove (c), we recall the definition of $R_\xi^\pm$
in \eqref{eq:Rpmxi} and \eqref{eq:Rpm0}.
Lemma~\ref{le:matchingplanarsolutions}
and Corollary~\ref{co:values-Rxi}
yield that $\max(\Pqr, R_\xi^+) \in \supers(g^\e, Q_q)$
and $\min(\Pqr, R_\xi^-) \in \subs(g^\e, Q_q)$ and thus, by definition,
\begin{align}
\label{Rxi-bounds-u}
\uu_\eqr \leq \max(\Pqr, R_\xi^+), \qquad
\lu_\eqr \geq \min(\Pqr, R_\xi^-)
\qquad \text{for all $\xi \in \Xi \cup \set0$,}
\end{align}
which proves (f).
Therefore Corollary~\ref{co:values-Rxi}
implies that
\begin{align*}
(\uu_\eqr)^* = \uu_\eqr = (\ou_\eqr)_* = \ou_\eqr = \Pqr
\qquad \text{on $(\Rn\times\R)\setminus Q_q$},
\end{align*}
and (c) is proved.

The result of (d) follows from \eqref{Rxi-bounds-u},
using the support property of $R^\pm_\xi$ in \eqref{Rxi-cone}.

Since $q$ and $r$ satisfy \eqref{rqrestriction},
(e) follows from (a)--(b) and Proposition~\ref{pr:planar-solution-range}.
\qedhere\end{proof}

\begin{remark}
\label{rem:extension}
Fix $T > 0$ and a pair $q,r$ that satisfies \eqref{rqrestriction}.
We shall show that $\uu_\eqr$ coincides with the function defined on $Q^T = Q_q \cap \set{t < T}$,
\begin{align*}
v = \pth{\inf \set{u \in \supers(g^\e, Q^T):
    u \geq P_{q,r} \text{ on $Q^T$}}}_{*,\cl{Q^T}},
\end{align*}
which is the solution of the obstacle problem considered in \cite{K07}.
An analogous result holds for $\lu_\eqr$.

Indeed, $v \leq \uu_\eqr$ trivially, but also,
since $v \leq R_\xi^+$ on $\cl Q^T$,
we can define for any $s< T$ the function
\begin{align*}
w(x,t) =
\begin{cases}
v(x,t) & t\leq s,\\
 R_0^+(x,t) & t > s.
\end{cases}
\end{align*}
Clearly $w \in \supers(g^\e,Q_q)$
and hence, by definition,
$\uu_\eqr \leq w = v$ for $t \leq s$ for any $s < T$.
\end{remark}

\subsection{Monotonicity of the obstacle problem}
\label{sec:monotonicity-obstacle}

The solutions $\uu_\eqr$ and $\ou_\eqr$ of the obstacle problem introduced in \eqref{obstaclesolution} admit a natural monotonicity property with respect to the hyperbolic scaling and a change of scale $\e$.
Moreover, the periodicity of the medium also yields the monotonicity with respect to translations by multiplies of $\e$, both in space and time.
Since we have a control on how fast the free boundaries of $\uu_\eqr$ and $\ou_\eqr$ detach from the obstacle at the parabolic boundary of the domain $Q_q$, we can extend the allowed range of translations.
This is the main result of this section.
The second result is the monotonicity of $\uu_\eqr$ and $\ou_\eqr$ in time.

\begin{proposition}[Monotonicity]
\label{pr:monotonicity}
Let $q$ and $r$ satisfy \eqref{rqrestriction},
and let $a > 0$.
Then
\begin{align*}
    \uu_\eqr(x,t) \leq a^{-1}\uu_{a\e;q,r}(a(x - y), a(t - \tau))
\end{align*}
for all $(x,t) \in \Rn\times\R$
and for all $(y, \tau) \in \e\Z^{n+1}$ such that
\begin{align}
\label{y-in-taplus-cone}
y\cdot \nu &\geq \max (r\tau, M\abs{q} \tau), &
&\text{and}&
y &\in \cl {\cone_{\nu,\ta^+}(r^+_V \tau \nu + (a - 1) V^+_0)},
\end{align}
where $r_V^\pm$, $V_0^\pm$ and $\ta^\pm$ were defined
in Section~\ref{sec:domain-geometry}.

Similarly,
\begin{align*}
    \ou_\eqr(x, t) \geq a^{-1}\ou_{a\e;q,r}(a (x-y),a (t-\tau))
\end{align*}
for all $(x,t) \in \Rn\times\R$
and for all $(y, \tau) \in \e\Z^{n+1}$ such that
\begin{align}
\label{y-in-taminus-cone}
y\cdot \nu &\leq \min (r\tau, m\abs{q} \tau), &
&\text{and}&
y &\in \cl {\cone_{-\nu,\ta^-}(r^-_V \tau \nu + (a - 1) V^-_0)}.
\end{align}
\end{proposition}

\begin{proof}
We shall show the ordering for $\ou_\eqr$,
the proof for $\uu_\eqr$ is analogous.

Fix $a > 0$ and $(y, \tau)$
that satisfies the hypothesis \eqref{y-in-taplus-cone}
and define the cylinder \[\Sigma = Q_q \cap (a^{-1}Q_q + (y, \tau)).\]
Let us, for simplicity,
denote \[v^a(x,t) := a^{-1}\ou_{a\e;q,r}(a(x-y), a(t-\tau)).\]
Due to the $\e\Z^{n+1}$-periodicity of $g^\e$,
we observe that $v^a \in \subs(g^\e, a^{-1}Q_q +(y,\tau))$.
We immediately have
\begin{align}
\label{obs-va-order}
\Pqr(x,t) \geq \Pqr(x-y,t-\tau) = a^{-1} \Pqr(a(x-y),a(t-\tau))
\geq v^a(x,t)
\end{align}
by \eqref{y-in-taminus-cone},
the scale invariance and Proposition~\ref{pr:obstacle-ordering}.

Our goal is to show that
\begin{align}
\label{uorder-complement}
\text{$\lu_\eqr \geq v^a$ on $\Sigma^c\quad$
and $\quad\cl\Omega(v^a)\cap \Sigma^c \subset \cl\Omega(\lu_\eqr)$}
\end{align}
since then we can apply Lemma~\ref{le:stitch}
with $Q_1 = \Sigma$, $Q_2 = Q_q$,
$u_1 = v^a$ and $u_2 = \lu_\eqr$
to conclude that $u \in \subs(g^\e, Q_q)$,
where $u$ is defined in Lemma~\ref{le:stitch}.
Because clearly $u \leq \Pqr$ by \eqref{obs-va-order},
we have by definition of $\lu_\eqr$ that
$\lu_\eqr \geq u \geq v^a$ on $\Sigma$
and therefore $\lu_\eqr \geq v^a$ on $\Rn\times\R$.

Let us prove \eqref{uorder-complement}.
We can write $\Sigma^c = A_1 \cup A_2 \cup A_3$
where
$A^1 := (Q_q)^c$,
$A_2 := Q_q \cap ((a^{-1} \Omega_q + y)^c \times [\tau,\infty))$
and
$A_3 := Q_q \cap \set{t < \tau}$.
We therefore prove $\lu_\eqr(x,t) \geq v^a(x,t)$
for all $(x,t) \in \Sigma^c$ by considering the following three cases:

\begin{itemize}
\item
$(x,t) \in A_1$:
This case is very simple since
Proposition~\ref{pr:obstacle-sols}(c)
and \eqref{obs-va-order},
yield
\begin{align*}
\lu_\eqr(x,t) = \Pqr(x,t) \geq v^a(x,t).
\end{align*}
\item
$(x,t) \in A_2$:
Observe that \eqref{y-in-taminus-cone} implies
Proposition~\ref{pr:cone-plus-generator}(c)
and therefore Proposition~\ref{pr:cone-plus-generator}(a)
holds.
We use
Proposition~\ref{pr:obstacle-sols}(f),
Proposition~\ref{pr:cone-plus-generator}(a)
and \eqref{Rxi-outside}, all properly rescaled, and estimate
\begin{align*}
\lu_\eqr(x,t) &\geq \sup_{\xi\in\Xi} R^-_\xi(x,t)
\geq \sup_{\xi\in\Xi} a^{-1}R^-_\xi(a(x-y), a(t -\tau))
\\&\geq a^{-1}\Pqr(a(x-y), a(t-\tau)) \geq v^a(x,t).
\end{align*}
\item
$(x,t) \in A_3$:
Since $0 \leq t \leq \tau$,
Proposition~\ref{pr:obstacle-sols}(f)
and \eqref{y-in-taminus-cone}
yield
\begin{align*}
\lu_\eqr(x,t) &\geq R^-_0(x,t) = \abs{q} (m\abs{q} t - x\cdot\nu)_+
\\&\geq \abs{q}(r(t- \tau) - (x - y)\cdot \nu)_+
= a^{-1}\Pqr(a(x-y), a(t-\tau))
\\&\geq v^a(x,t).
\end{align*}
\end{itemize}

Therefore we have shown the first part of \eqref{uorder-complement},
and the second part can be inferred from the above estimates as well.
Consequently, the argument using Lemma~\ref{le:stitch} described above
applies and we can conclude that
$\lu_\eqr \geq v^a$ on $\Rn\times\R$.
\qedhere\end{proof}

The second important result of this section is the almost obvious fact
that the solutions of the obstacle problems are increasing in time.
However, in contrast to \cite{K07},
Proposition~\ref{pr:monotonicity} only implies
the monotonicity in time for multiplies of $\e$
due to the time-dependence of $g$.
Here we present an argument using a nonlinear scaling of the solutions,
which also provides a lower bound
on the speed of the free boundary.

\begin{lemma}[Monotonicity in time]\label{le:monotonicity-in-time}
Let $r$ and $q$ satisfy \eqref{rqrestriction}
and let $\e > 0$. Then
\begin{align*}
\uu_\eqr(x - y, t_1) &\leq \uu_\eqr(x,t_2),\\
\ou_\eqr(x - y, t_1) &\leq \ou_\eqr(x,t_2)
\end{align*}
for any $t_1 <  t_2$ and $x \in \Rn$,
and $y \in \Rn$ such that $\abs{y} \leq \rho$
for some positive constant $\rho$ depending only on $m, L, \e, t_1, t_2$.
\end{lemma}

\begin{proof}
Let us fix $r$, $q$ and $\e$.
We first prove the result for $\uu := \uu_\eqr$.
Since we have to compare solutions
that are not necessarily translated by a multiply of $\e$,
we cannot use the monotonicity of
Proposition~\ref{pr:monotonicity} directly.
Instead,
we will compare $\uu$
with its nonuniform perturbation
using Proposition~\ref{pr:nonuniform-perturbation}.

First, assume that $0 < t_1 < t_2 < t_1 + \ga$
where $\gamma := \frac mL$.
We shall define
\begin{align}
\label{def-of-v-mont}
v(x,t) = \inf_{y \in \cl B_\rho(0)} \uu_\eqr(x - y, \ta(t)),
\end{align}
where $\rho$ will be a positive constant
and $\ta(t) := \ta(t;\bar\la) = \bar\la + f(t;\bar\la)$
for some $\bar\la$ determined below.
Function $f(t;\la)$ is the function constructed in
Section~\ref{sec:nonlinear-scaling-super}
for given parameters $\gamma$ (defined above), $a = 1$,
and $\la > 0$. Because $D\rho \equiv 0$ we set $\al = a = 1$.

With these parameters, the expression for $\ta(t;\la)$
is simply
\begin{align*}
\ta(t;\la) =
t - \gamma W\pth{-\frac{\la}{\ga} e^{-\frac \la \ga} e^{\frac t \ga}},
\end{align*}
where $W$ is the Lambert W function.
Therefore $\ta(t; \la)$ is well-defined on $t \in [0, t_\la]$
and $\la \in [0,\gamma]$,
where $t_\la := \log \frac \ga \la + \frac \la\ga -1$
(we set $t_0 = +\infty$).
Observe that $t_\la$ is strictly decreasing in $\la$
and positive for $\la \in (0, \ga]$ with
$t_\la = 0$ at $\la = \ga$ and $t_\la \to +\infty$
as $\la \to 0+$.
Therefore for every $t > 0$ there exists a unique $\la =: \la_t$
such that $t_\la = t$.
From the properties of $W$, we observe that $\ta(t; 0) = t$
and $\ta(t; \la_t) = t + \ga$ for any $t > 0$.
By continuity and the assumption $t_2 \in (t_1, t_1 + \ga)$,
there exists $\bar \la \in (0, \la_{t_1})$
such that $\ta(t_1; \bar\la) = t_2$.
Since $\bar\la < \la_{t_1}$ we must have $t_{\bar\la} > t_1$.

Let us
thus take $\ta(t) = \ta(t;\bar \la)$
and consider $\uu_\eqr(x-y, \ta(t))$
for some $y \in \Rn$ and $(x,t) \in \set{t \leq t_{\bar\la}}$.
It was shown in the proof of Proposition~\ref{pr:monotonicity}
that
$\uu_\eqr(x,t) \leq \uu_\eqr(x - y, t - \tau)$
on $(Q_q \cap (Q_q + (y,\tau))^c$
for all $(y,\tau)$ satisfying \eqref{y-in-taplus-cone}.
Let us set $\tau = - \bar\la$. There exists $\rho>0$
such that $(y,\tau)$ satisfies \eqref{y-in-taplus-cone}
for all $y \in \cl B_\rho(0)$.
Since $t + \bar \la\leq \ta(t)$ for all $t \in [0, t_{\bar\la}]$,
the same argument yields
$\uu_\eqr(x, t) \leq \uu_\eqr(x-y, \ta(t))$
on $(Q_q \cap (Q_q + (y,\tau))^c \cap \set{t \leq t_{\bar\la}}$
for $(y,\tau) \in \cl B_\rho(0) \times \set{-\bar\la}$.
Hence the argument using Lemma~\ref{le:stitch}
as in the proof of Proposition~\ref{pr:monotonicity}
(with Remark~\ref{rem:extension}
for extension of $v$ for $t > t_{\bar\la}$)
yields that $\uu_\eqr(x,t) \leq \uu_\eqr(x-y,\ta(t))$
for $(x,t) \in \set{t < t_{\bar\la}}$.
In particular $u(x,t_1) \leq u(x - y, \ta(t_1)) = u(x - y, t_2)$
for all $y \in \cl B_\rho(0)$,
and that is what we wanted to prove.

The proof for general $0 < t_1 < t_2$ can be done iteratively
by splitting $[t_1, t_2]$ into subintervals of length smaller than
$\ga$.

We skip the proof for $\ou_\eqr$, which follows the same idea,
but is simpler.
In particular, it is unnecessary to restrict
$t_2 < t_1 + \ga$
since $f$ constructed in Section~\ref{sec:nonlinear-scaling-sub}
can handle arbitrary tranlations.
\qedhere\end{proof}

For completeness, and as a consequence of the monotonicity in time in the previous lemma,
we also show that the solutions are harmonic in their positive sets.

\begin{proposition}
\label{pr:harmonic}
Let $r$ and $q$ satisfy \eqref{rqrestriction}
and let $\e > 0$.
The functions $\uu_\eqr(\cdot, t)$ and
$\pth{\ou_\eqr}_*(\cdot, t)$ are harmonic
in $\Omega_t(\uu_\eqr)\cap\Omega_q$
and $\Omega_t((\ou_\eqr)_*) \cap \Omega_q$,
respectively, for every $t$.
\end{proposition}

\begin{proof}
Since the claim is trivially true for $t \leq 0$,
we fix $\tau > 0$.
We first show that
$\uu_\eqr(\cdot, \tau)$ is harmonic
in the open set $E := \Omega_\tau(\uu_\eqr)\cap \Omega_q$.
For any $w \in LSC(\cl{E})$,
superharmonic in $E$, such that $w \geq \uu_\eqr(\cdot, \tau)$
on $\partial E$, we define a new function on $Q_q \cap \set{t \leq \tau}$
\begin{align*}
v(x, t) =
\begin{cases}
\min \pth{\uu_\eqr(x,t), w(x)} & x \in E,\\
\uu_\eqr(x, t) & \text{otherwise}.
\end{cases}
\end{align*}
By definition of $w$, we observe
that $v \in LSC(Q_q \cap \set{t \leq \tau})$
and $v \in \supers(g^\e, Q_q \cap \set{t \leq \tau})$.
Therefore $\uu_\eqr(x, \tau) \leq v(x, \tau) \leq w(x)$ in $E$
by definition of $\uu_\eqr$ and Remark~\ref{rem:extension}.
We recall that $\uu_\eqr(\cdot, \tau)$ is superharmonic in
$E$, and therefore it is the smallest superharmonic function,
i.e., it must be harmonic.

Now we show that $(\ou_\eqr)_*(\cdot, \tau)$ is harmonic
in the open set $E = \Omega_\tau((\ou_\eqr)_*)\cap\Omega_q$.
First we recall that $(\ou_\eqr)_*(\cdot, \tau)$
is superharmonic since $\ou_\eqr \in \supers(m, \Rn\times\R)$
by Proposition~\ref{pr:obstacle-sols}.
Let $w \in LSC(\cl E)$ be a superharmonic function in $E$,
such that $w \geq (\ou_\eqr)_*(\cdot, \tau)$ on $\partial E$.
By the monotonicity in time, Lemma~\ref{le:monotonicity-in-time},
and in particular the positive speed of expansion,
we observe that $\ou_\eqr(x, t) = 0$ for
$x \in \Gamma_\tau((\ou_\eqr)_*)$
and $t < \tau$.
Since $\ou_\eqr(\cdot, t)$ is subharmonic in $\Omega_q$ for all $t$,
we must have $\ou_\eqr(x,t) \leq w(x)$ for $t < \tau$.
Therefore, $(\ou_\eqr)_*(x, \tau) \leq w(x)$.
Again, $(\ou_\eqr)_*(\cdot, \tau)$ is the smallest superharmonic function
and thus harmonic.
\qedhere\end{proof}

\subsection{Flatness}
\label{sec:flatness}

In this section we introduce \emph{flatness},
the quantity that will be a measure
of how good our guess of the homogenized
velocity $r(q)$ is.
More specifically, it measures how far, for given $\e$, $q$ and $r$, the free boundaries of the obstacle problem solutions $\uu_\eqr$ and $\ou_\eqr$,
introduced in \eqref{obstaclesolution}, detach from the free boundary of the obstacle $\Pqr$ up to the given time.

For $\eta \in \R$, denote
\begin{align*}
P_{q,r}^\eta(x,t) := P_{q,r} (x - \eta \nu, t).
\end{align*}
Note that $\Pqr^\eta$ represents a translation of $\Pqr$
in the direction $\nu$ by distance $\eta$.

\begin{definition}
\label{def:flatness}
For $q \neq 0$, $r > 0$ and $\e > 0$,
we define the \emph{upper flatness} $\uPhi_\eqr(\tau)$,
i.e., the flatness of $\uu_\eqr$ up to time $\tau \in \R$
as
\begin{align*}
\uPhi_\eqr(\tau) =
    \inf \set{\eta:
        \uu_\eqr \leq P^\eta_{q,r} \text{ in } \set{t \leq \tau}}.
\end{align*}
Similarly, the \emph{lower flatness} $\lPhi_\eqr(\tau)$
is defined as the flatness of $\ou_\eqr$ up to time $\tau \in \R$,
\begin{align*}
\lPhi_\eqr(\tau) =
    \inf \set{\eta:
        \ou_\eqr \geq P^{-\eta}_{q,r} \text{ in } \set{t \leq \tau}}.
\end{align*}

We say that $\uu_\eqr$ (resp. $\ou_\eqr$) is $\eta$-flat up to time $\tau$
if $\uPhi_\eqr(\tau) \leq \eta$ (resp. $\lPhi_\eqr(\tau) \leq \eta$).
\end{definition}

The flatness is Lipschitz and monotone:

\begin{proposition}
\label{pr:phi-lipschitz}
Let $q \in \Rn \setminus \set0$, $r > 0$ and $\e > 0$.
Then for any $t \in \R$, $h > 0$,
\begin{align*}
\uPhi_\eqr(t) &\leq \uPhi_\eqr(t + h)
    \leq \uPhi_\eqr(t) + h(M\abs{q} - r)_+,\\
\lPhi_\eqr(t) &\leq \lPhi_\eqr(t + h)
    \leq \lPhi_\eqr(t) + h(r - m\abs{q})_+.
\end{align*}
In particular, $\lPhi_\eqr$ and $\uPhi_\eqr$ are
nondecreasing and Lipschitz continuous.
\end{proposition}

\begin{proof}
$\lPhi_\eqr$ and $\uPhi_\eqr$ are clearly nondecreasing
from their definition.

To show the upper bound for $\uPhi_\eqr(\tau + h)$ for any $\tau \in \R$,
$h > 0$,
we may assume that $r < M \abs q$, otherwise $\uu_\eqr = \Pqr$
and so $\uPhi_\eqr(\tau + h) = \uPhi_\eqr(\tau) = 0$.
We only need to compare $\uu_\eqr$ with an appropriate translation
of $R^+_0=P_{q,M\abs{q}}$, using the fact that
$\uu_\eqr(\cdot, \tau) \leq \Pqr^\eta(\cdot, \tau)$
and the identity
$P_{q, M\abs q}(\cdot, h) = \Pqr^{(M\abs q - r)h}(\cdot, h)$.
Similarly, the upper bound for $\lPhi_\eqr(\tau+h)$ follows
from comparison with a translation of $R^-_0 = P_{q,m\abs{q}}$.
\qedhere\end{proof}

Finally, we can always find a point on the free boundary of $\uu_\eqr$ or $\ou_\eqr$
whose spatial distance from the free boundary of the obstacle is exactly equal to the flatness at the given time.

\begin{lemma}
\label{le:furthest-point}
Let $q \in \Rn \setminus \set0$, $r > 0$ and $\e > 0$.
Let $\tau \in \R$.
Then for any $\eta \in (0, \uPhi_\eqr(\tau)]$ there exists
$(\zeta, \sigma) \in \Gamma(\uu_\eqr)$
such that
$\zeta \cdot \nu = r \sigma + \eta$
and $\uPhi_\eqr(t) \leq \eta$ for all $t \leq \si$.
Similarly,
for any $\eta \in (0, \lPhi_\eqr(\tau)]$ there exists
$(\zeta, \sigma) \in \Gamma(\ou_\eqr)$
such that
$\zeta \cdot \nu = r \sigma - \eta$
and $\lPhi_\eqr(t) \leq \eta$ for all $t \leq \si$.
\end{lemma}

\begin{proof}
We will show the proof for $\uu_\eqr$,
the proof for $\lu_\eqr$ is similar.

Let us define $T = \sup\set{t \leq \tau: \uPhi_\eqr(t) \leq \eta}$.
Clearly $T \in (0, \tau]$. Moreover, by continuity,
$\uPhi_\eqr(T) = \eta$.
We claim that we can find
$(\zeta, \sigma) \in \Omega^c(\Pqr^\eta) \cap \cl\Omega(\uu_\eqr)
\cap \set{t \leq T}$.
If such point does not exist,
we must have
$\dist(\Omega^c(\Pqr^\eta) \cap \set{0 \leq t \leq T},
\cl\Omega(\uu_\eqr)) > 0$
and therefore there exists $\de \in (0, \eta)$
such that
$\Omega^c(\Pqr^{\eta-\de}) \cap \cl\Omega(\uu_\eqr)
\cap \set{t \leq T} = \emptyset$.
Since $\Pqr^{\eta-\de}$ is harmonic in its positive set,
it follows that $\uu_\eqr \leq \Pqr^{\eta-\de}$ in $\set{t \leq T}$.
But that is a contradiction with
$\uPhi_\eqr(T) = \eta > \eta - \de$.
Therefore $(\zeta, \si)$ exists and
we observe that
$(\zeta, \si) \in \Gamma(\uu_\eqr) \cap \Gamma(\Pqr^\eta)
\cap \set{0 < t \leq T}$.
\qedhere\end{proof}

\subsection{Local comparison principle}
\label{sec:local-comparison-principle}

An important tool in the analysis of the behavior in the homogenization
limit is the \emph{local comparison principle}.
In contrast to the standard comparison principle, which requires
ordering on the parabolic boundary to guarantee ordering
in the whole cylinder,
the local comparison relies on the extra information about flatness
of solutions in the sense of Definition~\ref{def:flatness},
and therefore allows for solutions to cross on the lateral boundary
but still guarantees the ordering of free boundaries
for a short time
in a small region far from the boundary.
This is a generalization of a result that appeared in
\cite{K07,K08} (with $\be =1$).

\begin{theorem}[Local comparison principle]
\label{th:localComparison}
Let $\be \in (\frac45, 1)$, $r_1 > r_2 > 0$ and $q_1, q_2 \in \Rn$
such that $a^3 q_1 = q_2$ for some $a > 1$.\\
Then there exists
$\e_0 = \e_0(r_1,r_2, \abs{q_1}, \abs{q_2}, m, M, n, \be)$ such that
\[
\max \pth{\lPhi_{\e; q_1, r_1}(1), \uPhi_{\e; q_2, r_2}(1)} > \e^\be
\qquad \text{for } \e < \e_0.
\]
In other words,
$\ou_{\e; q_1, r_1}$ and $\uu_{\e; q_2, r_2}$ cannot be both $\e^\be$-flat up to time $t = 1$ for any $\e < \e_0$.
\end{theorem}

Before proceeding with the proof of Theorem~\ref{th:localComparison},
we first present a simple result on the nondegeneracy of $\uu_\eqr$.
This is an improvement of
Corollary~\ref{co:hs-nondegeneracy}
when we also know that $\uu_\eqr$ is $\e^\be$-flat.

\begin{lemma}
\label{le:nondegeneracyObstacleSupersolution}
For given $q \in \Rn \setminus \set0, r>0$,
there exists a positive constant $c = c(r, \abs{q}, M, n)$ such that if
$\uPhi_\eqr(T) \leq \eta$, for some $\e, \eta, T > 0$, then
\begin{align*}
\uu_\eqr(x,t) \geq c \frac{\rho^2}{\eta}
\end{align*}
for every $0 < \rho \leq \eta$ and $(x,t)$ such that $B_\rho(x) \subset \Omega_t(\uu_\eqr) \cap \Omega_q$, $t \leq T$.
\end{lemma}

\begin{proof}
Let $c_H = c_H(n)$ be the constant from Harnack inequality such that
if $\psi \geq 0$, harmonic in $B_1(0)$ then
$\inf_{\cl{B}_{1/2}(0)} \psi \geq c_H \sup_{\cl{B}_{1/2}(0)} \psi$.

Let us denote $u = \uu_\eqr$. Fix $(\xi, \tau)$ such that $B_\rho(\xi) \subset \Omega_{\tau}(u) \cap \Omega_q$.
Then we have two possible scenarios:
\begin{enumerate}
\item $\xi \cdot \nu \leq r \tau + \frac{\rho}{4}$: We estimate
\begin{align*}
u(\xi - \frac{\rho}{2} \nu, \tau) \geq P_{q,r}(\xi - \frac{\rho}{2}\nu, \tau) \geq \abs{q} \frac{\rho}{4} \geq \frac{\abs{q}}{4} \frac{\rho^2}{\eta}.
\end{align*}

\item $\xi \cdot \nu > r \tau + \frac{\rho}{4}$:
In this case necessarily $r < M \abs{q}$,
otherwise $u = \Pqr$, a contradiction.
We have $u \in \subs(M, \Rn\times \R)$
by Proposition~\ref{pr:obstacle-sols}
and $u$ is nondecreasing in time by Lemma~\ref{le:monotonicity-in-time}.
Due to the flatness assumption $\uPhi_\eqr(\tau) \leq \eta$,
$u(\cdot, t) = 0$ on $\cl{B}_{\rho/4}(\xi)$ for
$t \leq \tau - \frac{\eta}{r}$.
Consequently,
Corollary~\ref{co:hs-nondegeneracy} implies
\begin{align*}
\sup_{B_{\rho/4}(\xi)} u(\cdot, \tau) \geq \frac{\rho^2}{32n M \frac{\eta}{r}} = \frac{r}{32n M} \frac{\rho^2}{\eta}.
\end{align*}
\end{enumerate}
Since $u(\cdot, \tau)$ is harmonic in
$\Omega_\tau(u) \cap \Omega_q \supset B_{\rho}(\xi)$,
the Harnack inequality, properly rescaled, yields
\[
u(\xi, \tau) \geq c_H \min \pth{\frac{\abs{q}}{4}, \frac{r}{32nM}} \frac{\rho^2}{\eta}.\qedhere
\]
\qedhere\end{proof}

An important tool in the proof of the local comparison principle
is the nonuniform perturbation in space,
see Section~\ref{sec:nonlinear-perturbation}.
We will construct a particular radius $\rho(x)$
in the following lemma.

\begin{lemma}
\label{le:radial-nonlin-perturbation}
There exists $R > 0$ and a smooth,
radially symmetric function $\rho : \cl B_{2 R}(0) \setminus B_R(0)
\to [1,6]$
such that
$\rho \Delta \rho \geq (n-1) \abs{D \rho}^2$
in $B_{2 R}(0) \setminus \cl B_R(0)$,
$\rho = 1$ on $\partial B_{2R}(0)$ and
$\rho > 5$ on $\cl B_{R+1}(0) \setminus B_R(0)$.
\end{lemma}

\begin{proof}
Let $\vp$ be the radially symmetric smooth positive solution
of
\begin{align*}
\begin{cases}
\lap \vp^{2-n} = 0, & \text{ in } B_{2} \setminus \cl B_1\\
\vp = 1 & \text{ on } \partial B_{2},\\
\vp = 6 & \text{ on } \partial B_{1}.\\
\end{cases}
\end{align*}
It was observed in \cite[p. 148]{C87}
that
$\vp$ satisfies $\vp \lap \vp \geq (n-1) \abs{D\vp}^2$.
Moreover,
note that also the rescaled function $a \vp(b x)$ satisfies the same inequality for any $a, b$ positive.
By continuity, there exists $\de > 0$ such that $\vp > 5$
on $\cl B_{1 + \de} \setminus B_1$.
Define $R = 1/\de$ and
$\rho(x) = \vp(\de x)$
on $\cl B_{2R} \setminus B_R$.
It is straightforward to verify that $\rho$ satisfies
the properties asserted in the statement and the proof is finished.
\qedhere\end{proof}

The following technical lemma
shows
that the closure of a time-slice of $\Omega(\lu_\eqr)$
is the time-slice of the closure,
a fact
that will be used throughout the proof of the local comparison principle
and
that allows us to work with simple time-slices of
a subsolution with nondecreasing support.

\begin{lemma}
Let $u \in \subs(M, Q)$
be a bounded subsolution on $Q$
for some $M > 0$ and $Q = E \times (t_1,t_2)$,
$E \subset \Rn$ open,
and assume that $\Omega_t(u^*) \subset \Omega_s(u^*)$ for all
$t, s$ such that
$t_1 < t \leq s < t_2$.
Then for any $(x,t) \in Q$ we have
\begin{align*}
x &\in \cl{\Omega_t(u^*;Q)}&&\text{if and only if}
& (x,t) \in \cl\Omega(u;Q),
\end{align*}
or, equivalently,
\begin{align*}
\cl{\Omega_t(u^*;Q)} = \cl\Omega_t(u;Q).
\end{align*}
\end{lemma}

\begin{proof}
Since $\cl\Omega(u) = \cl\Omega(u^*)$,
we can WLOG assume that $u = u^*$.
Let $(\xi, \tau) \in Q$.
It is clear that if $\xi \in \cl{\Omega_\tau(u)}$
then $(\xi, \tau) \in \cl \Omega(u)$.
Now suppose that $(\xi, \tau) \in \cl\Omega(u) \cap Q$.
By Definition~\ref{def:visc-test-sub}(i)
and the monotonicity assumption,
$(\xi, \tau) \in \cl{\Omega(u) \cap \set{t < \tau}}
\subset \cl{\Omega_\tau(u)} \times \set{t \leq \tau}$.
\qedhere\end{proof}

\begin{proof}[Proof of Theorem~\ref{th:localComparison}]
The proof proceeds in a number of technical steps.

Suppose that such $\e_0$ does not exist and
both $\ou_{\e; q_1, r_1}$ and $\uu_{\e; q_2, r_2}$
are $\e^\be$-flat up to time $t = 1$
along some subsequence of $\e$ that converges to 0.
We will show that this leads to a contradiction
with the comparison principle for the Hele-Shaw problem
for $\e$ sufficiently small.

\parahead{(1) Setup}
Let $\al := \frac{4}{5}$, and observe that, since $\be \in (4/5,1)$,
\begin{align}
\label{alpha-beta-relation}
2 - 2 \be \leq 4 - 4 \be < \al < \be.
\end{align}

First,
for fixed $\e > 0$,
we introduce notation that will be used throughout the proof.
Symbol $c$ with a subscript will denote quantities independent of $\e$.

We shall use multiple perturbations of
the subsolution $\ou_{\e,q_1,r_1}$.
Throughout the proof,
$\nu = \frac{-q_1}{\abs{q_1}} = \frac{-q_2}{\abs{q_2}}$.
Let us denote $\Lambda := 6$ and define the constant
\begin{align}
\label{choice-gamma}
\ga = \frac{1}{\Lambda}\min\pth{\frac{m}{L}(a^{1/2} - 1), \ov2}
\end{align}
and a translation
\begin{align*}
\xi_0 \in \argmin_{\substack{\xi \in \e Z^n\\
                            \xi \cdot \nu \leq - \e}}
    \abs{\xi}.
\end{align*}
Clearly $\abs{\xi_0} \leq n \e$.
For $\la \in [0,\Lambda]$,
we define the perturbations
\begin{align*}
u^\la(x,t) :=
    \sup_{y \in \cl B_{\la \ga \e}(x)} a\lu_{\e,q_1,r_1}(y - \xi_0,t),
\end{align*}
where $\cl B_0(x) := \set{x}$ for simplicity.
The choice of $\gamma$ and $\xi_0$,
and Proposition~\ref{pr:nonuniform-perturbation}
imply that $u^\la \in \subs(g^\e, \tilde Q)$,
where $\tilde Q := \set{(x,t) : B_{(n+1)\e}(x) \times \set{t} \subset Q_q}$.
Furthermore,
since $\ga \leq \frac1{2\Lambda}$ we note that
$u^\la \leq P_{aq_1, r_1}$ for all $\la \in [0, \Lambda]$.

Next, we introduce the time
\begin{align*}
T := \frac{3 \e^\be}{r_1 -r_2} = c_T \e^\be > 0.
\end{align*}
If $\e$ is small enough so that
$T \leq 1$ then $u^\la$ and $v$ will be
$(\e^\be + n\e)$-flat and $\e^\be$-flat,
respectively, up to time $t = T$,
and if also
$\e^\be \geq n\e \geq \abs{\xi_0 \cdot \nu}$
then the free boundary of $u^\la$
must completely overtake the free boundary of $v$
by time $T$.

\newcommand{\Cyl}{\operatorname{Cyl}}

Finally, we choose the constants $\si_1, \si_2$ as
\begin{align*}
\si_1 &:= \frac{3a \abs{q_1} + 1}{\abs{q_2} - a \abs{q_1}} \e^\be
    = c_{\si_1} \e^\be,\\
\si_2 &:= T r_1 = c_{\si_2} \e^\be.
\end{align*}
Let us also denote $\si := \si_1 + \si_2$ and remark
that $\si = c_\si \e^\be$.
The motivation behind this particular choice of $\si_1$
is the fact that
\begin{align}
\label{eq:boundsigma1}
v(x,t) - u^\la(x,t) \geq P_{q_2, r_2}(x,t) -P_{aq_1,r_1}(x,t) \geq \e^\be
\end{align}
for $x \cdot \nu \leq - \si_1$, $t \in [0,T]$, $\la \leq \Lambda$.
Similarly,
$\si_2$ is motivated by
\begin{align}
\label{obstacle-reach}
{\cl\Omega_t(P_{a q_1,r_1}) \cup \cl\Omega_t(P^{\e^\be}_{q_2,r_2})}
    \subset \set{x : x \cdot \nu \leq \si_2}
\end{align}
for $t \in [0,T]$,
i.e. the free boundaries of $\ou_{\e;q_1,r_1}$ and
$\uu_{\e;q_2,r_2}$ will be contained in
$\set{x : 0 \leq x \cdot \nu \leq \si_2}$ for $t \in [0,T]$.

\parahead{(2) Fast-shrinking domain}
For a point $x \in \Rn$,
we define $x^\perp$ to be the
component of $x$ orthogonal to $\nu$,
\begin{align*}
x^\perp = x - (x \cdot \nu) \nu = (I - \nu \otimes \nu) x.
\end{align*}
By symbol $\Cyl(\rho; \eta_1, \eta_2)$ we shall
denote the open space cylinder of radius $\rho$
with axis parallel to $\nu$,
between the hyperplanes $\set{x : x \cdot \nu = \eta_1}$ and
$\set{x: x \cdot \nu = \eta_2}$,
\begin{align*}
\Cyl(\rho; \eta_1, \eta_2)
    = \set{x : \abs{x^\perp} < \rho,
        \ \eta_1 < x \cdot \nu < \eta_2}.
\end{align*}

Let us define a ``fast-shrinking'' domain $\Sigma$,
whose each time-slice $\Sigma_t$ is the cylinder
\begin{align*}
\Sigma_t :=
\begin{cases}
\Cyl(\rho_\Sigma(t); -\si_1, \si_2) & t \geq 0,\\
\emptyset & t < 0.
\end{cases}
\end{align*}
with radius
\begin{align*}
\rho_\Sigma(t) = \frac{\tan\ta}2 - \e^{-\al} t,
\end{align*}
where $\ta$ is the opening of the cone $\Omega_q$ defining $Q_q$,
introduced in \eqref{eq:choiceofta}.
In particular, the radius $\rho_\Sigma(t)$ of the cylinder $\Sigma_t$
shrinks with velocity
$\e^{-\al}$.
Furthermore,
if $\e$ is small enough,
we have $\Sigma + (\cl B_{(n+1)\e} \times \set0) \subset Q_q$,
which also implies
$u^\la \in \subs(g^\e, \Sigma)$ for $\la \in [0,\Lambda]$.
Additionally,
$\Sigma_t \neq \emptyset$ for all $t \in [0, T]$
since, as $\al < \be$ by \eqref{alpha-beta-relation},
\begin{align*}
\rho_\Sigma(T) &= \frac{\tan\ta}2 - \e^{-\al} T
    = \frac{\tan\ta}2 - c_T \e^{\be - \al} > 0
    &&\text{for $\e>0$ small.}
\end{align*}

Finally, we need to fix a small time step $\tau$,
\begin{align}
\label{lc-def-of-tau}
\tau := \frac{\ga^2 \e^2}{4n a \si \abs{q_1} M} \quad
        = c_\tau \e^{2-\be}.
\end{align}
Let us explain how such choice of $\tau$ guarantees
that, for $\la \in [0,\Lambda]$,
the support of $u^\la$
cannot expand by more than $\ga \e$
from time $\hat t - \tau$  to time $\hat t$
for any $\hat t \in [\tau, T]$ in the sense
\begin{align}
\label{expand-by-gae}
\cl\Omega_s(u^\la)
    &\subset \cl\Omega_{\hat t - \tau}(u^\la) + \cl B_{\ga\e}
&&s \in [\hat t-\tau, \hat t].
\end{align}
This follows from a straightforward application of
Corollary~\ref{co:HS-expansion-speed},
on the parabolic cylinder
\[E = \set{(x,t) : -\si_1 < x \cdot \nu,\ t > \hat t - \tau},\]
taking advantage of the bound
\begin{align}
\label{eq:boundulambda}
u^\la(x,t) &\leq P_{aq_1,r_1}(x,t) \leq a \abs{q_1} \si \quad =: c_K \e^\be,
\end{align}
valid for $x\cdot \nu \geq -\si_1$ and $t \in [0, T]$
due to \eqref{obstacle-reach}.

\parahead{(3) Boundary crossing point}
With these definitions at hand,
we proceed with the proof.
As we observed above,
since $n \e \leq \e^\be$ if $\e$ is small enough,
$u^\la$ is $(\e^\be + n\e)$-flat
for any $\la \in [0, \Lambda]$,
and $v$ is $\e^\be$-flat,
the choice of $\xi_0$ and $T$ guarantees
that the supports of $u^\la$ and $v$ are originally strictly ordered,
but
will cross each other
before the time $T$ in the domain $\Sigma$.
So let $\hat t$ be the first time that
the boundaries of $u^5$ and $v$ touch in $\Sigma$,
i.e.,
\begin{align*}
\hat t := \sup \set{s : \cl\Omega(u^5) \cap
    \Omega^c(v) \cap \cl\Sigma \cap \set{t \leq s} = \emptyset} < T.
\end{align*}
Observe that $\hat t \geq \tau$ since
$\cl\Omega_t(u^5)$ cannot advance more than $\ga\e$
in time $\tau$ by \eqref{expand-by-gae}.

We claim that the set $\cl\Omega_{\hat t}(u^5) \cap
    \Omega^c_{\hat t}(v) \cap \cl\Sigma_{\hat t}$
is nonempty.
Indeed, if that were not the case,
$\dist(\cl\Omega_{\hat t}(u^5),
    \Omega^c_{\hat t}(v) \cap \cl\Sigma_{\hat t}) > 0$
because the first set is closed and the second one is compact.
Therefore, there exists $\de > 0$ such that
the distance is positive for $t < \hat t + \de$ by
Corollary~\ref{co:HS-expansion-speed}
and the monotonicity of $\Omega^c_t(v) \cap \cl\Sigma_t$,
which yields a contradiction with the definition of $\hat t$.
Let us thus choose
\begin{align*}
\hat x \in \cl\Omega_{\hat t}(u^5) \cap
    \Omega^c_{\hat t}(v) \cap \cl\Sigma_{\hat t}.
\end{align*}
We will show that the existence of such point leads to a contradiction
with the comparison principle for the Hele-Shaw problem
since $u^5$ and $v$ were strictly separated on a much larger
domain $\cl \Sigma_{\hat t - \tau}$ at time $t = \hat t - \tau$.

To make this idea rigorous, we introduce the cylinders
\begin{align*}
S_i &= \Cyl(\rho_i; -\si_1, \si_2) & &\text{for } i=1,2,
\end{align*}
(see Figure~\ref{fig:localcomparison2})
\begin{figure}
\centering
\fig{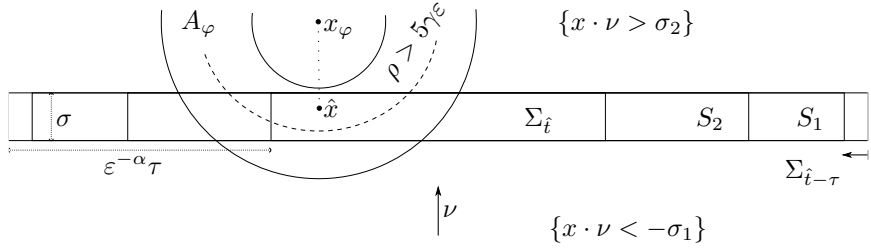}{4.5in}
\caption{Arrangement of cylinders}
\label{fig:localcomparison2}
\end{figure}
with radii
\begin{align*}
\rho_1 &:= \rho_\Sigma(\hat t - \tau) - 2 \ga \e, &
\rho_2 &:= \rho_1 - \pth{\frac{c_K}{c_N}\e^{2\be - 2} + 2}c_\si \e^\be,
\end{align*}
where constant $c_N$ is specified below in \eqref{cN}.
We point out that if $\e$ is small enough, we have
\begin{align}
\label{rho2-rhoSigma}
\rho_2 > \rho_\Sigma(\hat t) + 2 R \si,
\end{align}
where $R$ is the constant from
Lemma~\ref{le:radial-nonlin-perturbation},
since
\[
\rho_\Sigma(\hat t - \tau) - \rho_\Sigma(\hat t) = \tau \e^{-\alpha}
= c_\tau \e^{2-\al-\be} \gg \rho_{\Sigma}(\hat t - \tau) -\rho_2 + 2R\sigma
\qquad \text{if } \e \ll 1
\]
due to \eqref{alpha-beta-relation}.
In particular, we have
\begin{align*}
\Sigma_{\hat t - \tau} \supset S_1 \supset S_2 \supset \Sigma_{\hat t},
\end{align*}
see Figure~\ref{fig:localcomparison2}.

Now we consider function $u^2$.
First, clearly $\cl\Omega_t(u^5) = \cl\Omega_t(u^2) + \cl B_{3\ga\e}(0)$.
The definition of $\hat t$ yields
$\cl\Omega_t (u^5) \cap
    \Omega^c_t(v) \cap \cl\Sigma_t = \emptyset$
for $t < \hat t$,
and thus, in particular,
\begin{align*}
\pth{\cl\Omega_{\hat t - \tau} (u^2) + \cl B_{3\ga \e}(0)}
    \cap \Omega^c_{\hat t - \tau}(v)
    \cap \cl\Sigma_{\hat t - \tau} = \emptyset.
\end{align*}
As $\Omega^c_t(v)$ cannot expand
(due to the monotonicity, Lemma~\ref{le:monotonicity-in-time}),
and the positive set of $u^2$ cannot expand by more than $\ga \e$
in time $\tau$,
in the sense of \eqref{expand-by-gae},
we conclude that
\begin{align}
\label{u1-v-separated}
(\cl\Omega_t(u^2) + \cl B_{2\ga \e}(0))
    \cap \Omega^c_t (v) \cap \cl S_1 &= \emptyset
    &&t \in [0, \hat t].
\end{align}

Let us define the closed sets
\begin{align*}
A_t &:= \pth{\cl\Omega_{t}(u^2) + \cl B_{\ga\e}(0)}
    \cap \cl S_1
&&t \in [0, \hat t].
\end{align*}
Due to the nondegeneracy of $v$ (Lemma~\ref{le:nondegeneracyObstacleSupersolution} applied with $\eta = \e^\be$ and $\rho = \ga \e$), we observe that
\begin{align}
\label{cN}
v(x,t) \geq c(r_2, \abs{q_2}, M, n) \ga^2 \e^{2-\be} \quad = c_N\e^{2-\be} \qquad x \in A_t, \ t \in [0, \hat t],
\end{align}
where $c(r_2, \abs{q_2}, M, n)$ is the constant
from Lemma~\ref{le:nondegeneracyObstacleSupersolution}.
Consider the function
$w^t(x) = \frac{u^2(x,t) - v(x,t)}{c_N \e^{2-\be}}$
defined on $A_t$.
Function $w^t$ is subharmonic on $A_t$ for all $t \in [0, \hat t]$
since $v(\cdot, t)$ is superharmonic in $\interior A_t$ as $A^t \subset \Omega_t(v) \cap \Omega_q$,
and $u^2(\cdot, t)$ is clearly subharmonic in $\interior A_t$.

Our goal is to apply Lemma~\ref{le:subharmonicOnThinDomain}
to function $w^t$.
Since $\e$ is small,
we have $\e^{\be} \geq c_N \e^{2-\be}$
and thus $w^t \leq -1$ on $\partial A_t \cap \set{x : x \cdot \nu = -\si_1}$
by \eqref{eq:boundsigma1}.
Additionally,
$u^2 = 0$ on $\partial A_t \cap S_1$ by definition of $A_t$
and therefore
$w^t \leq -1$ on $\partial A_t \cap S_1$.
Finally, by \eqref{eq:boundulambda},
\begin{align*}
w^t \leq
    \frac{u^2}{c_N \e^{2-\be}}
    \leq \frac{c_K}{c_N} \e^{2\be -2} \qquad \text{on } A_t.
\end{align*}
Therefore Lemma~\ref{le:subharmonicOnThinDomain}
implies that $w^t < 0$
on the set $S_1 \cap \set{x : \abs{x^\perp} \leq \rho_1'}$,
where $\rho_1'$ is given by
\begin{align*}
\rho_1' = \rho_1 - c_T
    \pth{\frac{c_K}{c_N}\e^{2\be - 2} + 2}c_\si \e^\be,
\end{align*}
and $\rho' \geq \rho_2$ by the choice of $\rho_1$ and $\rho_2$.
We conclude that $w^t < 0$ in $\cl{S_2} \cap A_t$
for all $t \in [0, \hat t]$,
and therefore,
with \eqref{u1-v-separated},
\begin{align}
\label{u2-prec-v}
u^2 &\prec v &&\text{ in $\cl S_2 \times [0, \hat t]$}.
\end{align}

Let $R$ and $\rho(x)$ be from
Lemma~\ref{le:radial-nonlin-perturbation} and define
\begin{align*}
x_\vp := \hat x^\perp + (\si_2 + R\si) \nu
\end{align*}
and
\begin{align*}
\vp(x) := \ga \e \rho \pth{\frac{x - x_\vp}{\si}}
    \qquad \text{on } A_\vp := B_{2R\si}(x_\vp) \setminus \cl B_{R\si}(x_\vp).
\end{align*}
Let us set
$c_\rho = \max_{\cl B_{2R}(0) \setminus B_R(0)} \abs{D\rho} < \infty$.
Clearly
\begin{align*}
\max_{\cl A_\vp} \abs{D\vp}
    = \frac{\ga \e}{\si} c_\rho
    = \frac{c_\rho}{c_\si} \ga \e^{1-\be}.
\end{align*}
Hence for $\e$ small enough we have
\begin{align}
\label{smallness-Dvp}
(1 - \max_{\cl A_\vp} \abs{D\vp})^{-2} \leq a^{1/2}.
\end{align}
Moreover, $\si_2 < \si$ yields
$\hat x \in \cl B_{(R+1) \si}(x_\vp) \setminus B_{R\si}(x_\vp)$
and hence,
by the definition of $\rho$,
we also have $\vp(\hat{x}) > 5\ga\e$.

\parahead{(4) Non-uniform perturbation}
In the final step, we consider the nonuniform perturbation $u^\vp$
of $\ou_\eqr$,
\begin{align*}
u^\vp(x,t) := \sup_{y \in \cl B_{\vp(x)}(x)}
                a \ou_{\e;aq_1, r_1}(y - \xi_0, t)
    \qquad (x,t) \in A_\vp \times (0, \hat t].
\end{align*}
This function is a subsolution $u^\vp \in \subs(g^\e, Q_\vp)$
on $Q_\vp := A_\vp \times (0,\hat t]$
for all $\e$ small enough
so that \eqref{smallness-Dvp} holds,
due to Proposition~\ref{pr:nonuniform-perturbation}
and the choice of $\gamma$ in \eqref{choice-gamma}.

We want to show that $u^\vp \prec v$ on
the parabolic boundary
$\partial_P (A_\vp \times (0, \hat t])$
to be able to use the comparison theorem for the Hele-Shaw problem:
\begin{compactitem}
\item
$u^\vp \prec v$ on $A_\vp \times \set{0}$
by the choice of $\xi_0$,
which implies $u^\Lambda \prec v$ on $A_\vp \times \set{0}$,
and $u^\vp \leq u^\Lambda$.
\item
$\cl\Omega(u^\vp; Q_\vp)
    \cap \pth{\partial B_{R\si}(x_\vp) \times [0,\hat t]} = \emptyset$
since $\cl B_{R \si} (x_\vp) \subset
\set{x: x\cdot\nu\geq \si_2}$, recalling \eqref{eq:boundulambda}
and \eqref{obstacle-reach}.
\item
To show $u^\vp \prec v$ on
$\partial B_{2R\si}(x_\vp) \times [0, \hat t]$ w.r.t. $Q_\vp$,
we split $\partial B_{2R\si}(x_\vp)$ into three parts.
By construction of $S_2$
and \eqref{rho2-rhoSigma},
we have
\[
\partial B_{2R\si}(x_\vp) \subset \set{x : x\cdot \nu \leq - \si_1}
\cup S_2 \cup \set{x: x\cdot\nu \geq \si_2}.
\]
But the strict separation $u^\vp \prec v$
follows from $u^\vp = u^1$ on $\partial B_{2R\si(x_\vp)} \times [0,\hat t]$,
and then \eqref{eq:boundsigma1} for the first part,
\eqref{u2-prec-v} for the second part,
and \eqref{eq:boundulambda} with \eqref{obstacle-reach}
for the third part.

\end{compactitem}
Therefore $u^\vp \prec v$ on $\cl A_\vp \times [0,\hat t]$
by the comparison theorem, Theorem~\ref{th:comparison}.
This is however a contradiction with
\begin{align*}
\hat x \in \cl \Omega_{\hat t}(u^\vp) \cap \Omega^c_{\hat t}(v),
\end{align*}
since $u^\vp(\hat x, \hat t) > 0$
due to $\hat x \in \cl \Omega_{\hat t}(u^5)$
and $\vp(\hat x) > 5 \ga \e$.
This finishes the proof of the local comparison principle.
\qedhere\end{proof}

\subsection{Cone flatness}
\label{sec:cone-flatness}

In this part
we show that the positive and the zero sets of the obstacle solutions $\uu_\eqr$ and $\ou_\eqr$
cannot form long, thin fingers in the direction orthogonal to the obstacle, and, on the contrary, that their free boundaries are in fact in between
two cones that are $\sim \e \abs{\ln \e}^{1/2}$ apart.
We refer to this property as \emph{cone flatness}.
It is a consequence of the monotonicity of the obstacle problem and our particular choice of the domain $Q_q$
for the obstacle problem that allows us to control how fast the free boundaries of $\uu_\eqr$ and $\ou_\eqr$ detach from the obstacle at the boundary of the domain $Q_q$; see also the discussion in previous sections.

We have the following result:

\begin{proposition}[Cone flatness]
\label{pr:cone-flatness}
Let $q$ and $r$ satisfy \eqref{rqrestriction}
and let $T > 0$.
There exist positive constants $K$ and $\e_0$,
both depending only on $(n, m, M, \abs q, T)$,
such that
for any $(\zeta, \si) \in \Omega^c(\uu_\eqr)$, $\si \leq T$ and $\e < \e_0$
\begin{align*}
    \cone_{\nu, \ta^+}(\zeta + K \e \abs{\ln \e}^{1/2} \nu)
        \times (-\infty, \si] \subset \Omega^c(\uu_\eqr),
\end{align*}
and similarly for any $(\zeta,\sigma) \in \Omega^c(\ou_\eqr)$
and $\e < \e_0$
\begin{align*}
    \cone_{\nu, \ta^-}(\zeta + K \e \abs{\ln \e}^{1/2} \nu)
        \times (-\infty, \si]
        \subset \Omega^c(\ou_\eqr).
\end{align*}
(Note the different opening angles $\ta^+$ and $\ta^-$, introduced in Section~\ref{sec:domain-geometry}.)
\end{proposition}

\begin{proof}
Before proceeding with the detailed proof let us give its brief outline. We use $u$ to denote either $\uu_\eqr$ or $(\ou_\eqr)_*$.
\begin{compactenum}[Step 1.]
\item A supersolution $u \in \supers(m, Q_q)$, harmonic in its positive set,
can be bounded from above using Proposition~\ref{pr:solUpperBoundZeroSet}.
\item The monotonicity property,
Proposition~\ref{pr:monotonicity}, implies that $u$
is in fact small in the whole cone.
\item A barrier prevents $u$ from becoming positive
    too far from the boundary of the cone.
\end{compactenum}

\medskip
\emph{Step 1.} For the sake of brevity, let us denote $\uu = \uu_\eqr$.
By Remark~\ref{rem:grid-approx-cone}
there exists $\la = \la(n, \ta^+) > 0$
such that
\begin{align}
\label{cone-ball-cover}
{\cone_{\nu,\ta^+}(x)} \subset
\pth{\cone_{\nu,\theta^+}(x)}^{\grid\e} + B_{\la\e}(0)
 \qquad \text{for any $x \in \Rn$, $\e > 0$,}
\end{align}
where the notation $\grid\e$
is introduced in Definition~\ref{def:grid}.

Pick $(\zeta, \si) \in \Omega^c(\uu) \cap \set{t \leq T}$.
Note that $\uu$ is in $\supers(m, Q_q)$,
nondecreasing in time by Lemma~\ref{le:monotonicity-in-time}
and harmonic in $\Omega_t(\uu; Q_q)$ by
Proposition~\ref{pr:harmonic}.
Therefore
Proposition~\ref{pr:solUpperBoundZeroSet} yields
\begin{align}
\label{rational-bound}
\uu(z,t) &\leq \frac Cm \frac{\abs{z - \zeta}^2}{\si - t}
    \leq \frac Cm \frac{\la^2 \e^2}{\si - t}
&& \text{for all } z \in \cl B_{\la \e} (\zeta),\ t < \si.
\end{align}

\emph{Step 2.} By the monotonicity result of Proposition~\ref{pr:monotonicity},
we have
\begin{align}
\label{epsilon-monotonicity}
\uu(z + y,t) \leq \uu(z, t - \e)
\end{align}
for all $(z,t) \in \Rn\times\R$ and
\begin{align}
\label{y-in-cone}
y \in
    \pth{{\cone_{q,\theta^+}} (c\e \nu)}^\gride,
\end{align}
where $c = \max (M\abs q, r_V^+)$.

\eqref{cone-ball-cover} implies
${\cone_{\nu,\theta^+}} (c\e \nu + \zeta)
\subset \zeta + \pth{{\cone_{\nu,\theta^+}} (c\e \nu)}^\gride + B_{\la \e}(0)$
and therefore for any
$x \in \cone_{\nu,\theta^+} (c\e \nu + \zeta)$
there exist $z \in B_{\la \e}(\zeta)$ and $y$ satisfying \eqref{y-in-cone},
and,
combining \eqref{epsilon-monotonicity} and \eqref{rational-bound},
we obtain
\begin{align}
\label{rational-bound-in-cone}
\uu(x,t) = \uu(z + y, t) \leq
\uu(z,t - \e)
\leq \frac{C \la^2}{m} \frac{\e^2}{\e + \si - t}.
\end{align}

\emph{Step 3.} Let us now choose $\mu$ satisfying
\begin{align*}
\frac{m\mu^2}{2nMC\la^2\e^2} = -2 \ln \e,
\end{align*}
i.e.,
\begin{align*}
\mu = \tilde K \abs{\ln \e}^{1/2} \e.
\end{align*}
Hence, provided that $\e < \e_0 = \e_0(T)$,
we have
\begin{align}
\label{eps-and-T}
\e \pth{e^{\frac{m \mu^2}{2nM C\la^2\e^2}} - 1} =
\e (\e^{-2} - 1) > T.
\end{align}
We choose $K$ such that for all $\e < \e_0$
\begin{align*}
\dist\pth{\cone_{\nu,\ta^+}(K \abs{\ln\e}^{1/2}\e \nu),
\partial \cone_{\nu,\ta^+}(c\e\nu)} > \mu.
\end{align*}
Consequently,
whenever $\xi \in \cone_{\nu, \ta^+}(\zeta + K \e \abs{\ln \e}^{1/2} \nu)$
we also have
$\cl B_\mu(\xi) \subset \cone_{\nu, \ta^+}(\zeta + c\e\nu)$
and therefore the bound \eqref{rational-bound-in-cone}
for
$\uu$
holds in $\cl B_\mu(\xi) \times [0, \si]$.
Since $\si \leq T$ and \eqref{eps-and-T} holds,
we set $A = \frac{C\la^2\e^2}{m}$ and
apply Corollary~\ref{co:rational-contract-bound}
to conclude that $\uu(\xi, t) = 0$ for $t \in [0,\si]$.
This finishes the proof for $\uu_\eqr$.

The proof for $\ou = \ou_\eqr$ is analogous since
$\ou \in \supers(m, Q_q)$ due to Proposition~\ref{pr:obstacle-sols}(e),
$(\ou_\eqr)_*$ is harmonic in the positive phase
by Proposition~\ref{pr:harmonic},
and the correct monotonicity holds due to
Proposition~\ref{pr:monotonicity}
with a cone ${\cone_{\nu,\theta^-}}$.
\qedhere\end{proof}

\begin{remark}
\label{rem:reversed-cone-flatness}
The statement of Proposition~\ref{pr:cone-flatness}
can be reversed in the following sense:
For $q$, $r$, $T$, $K$ and $\e_0$ from Proposition~\ref{pr:cone-flatness}
we have that
if
$(\zeta,\sigma) \in \Omega(\uu_\eqr)$, $\si\leq T$
and $\e < \e_0$
then
\begin{align}
\label{reversed-flatness}
\cone_{-\nu, \ta^+}(\zeta - K \e \abs{\ln\e}^{\frac12}\nu) \times
[\si, +\infty) \subset \Omega(\uu_\eqr).
\end{align}
The proof is very simple:
suppose that the inclusion \eqref{reversed-flatness} is violated,
i.e., we have $(\zeta, \si) \in \Omega(\uu_\eqr)$
for which there exists $(y,s)$
that belongs to the left-hand side of \eqref{reversed-flatness}
but belongs also to $\Omega^c(\uu_\eqr)$.
Since $s \geq \si$, the monotonicity in time, Lemma~\ref{le:monotonicity-in-time},
implies that $(y, \si) \in \Omega^c(\uu_\eqr)$
and we can apply Proposition~\ref{pr:cone-flatness}
and conclude that
$(\xi, \sigma) \in \Omega^c(\uu_\eqr)$
since $\xi\in
 \cone_{\nu, \ta^+}(y + K \e \abs{\ln \e}^{1/2} \nu)$.
But that is obviously a contradiction.
A similar reflection argument applies for $\lu_\eqr$ as well.
\end{remark}

Since it is easier to work with a ball instead of a cone,
we also formulate the following consequence of
Proposition~\ref{pr:cone-flatness}.

\begin{corollary}
\label{co:ball-outlier}
Let $q$ and $r$ satisfy \eqref{rqrestriction}
and let $\la > 0$
and $\be \in [0, 1)$.
There exists $\e_0 = \e_0(n, m, M, q, \la,\be)$
such that for all $\e < \e_0$
the following statements hold:
\begin{itemize}
\item
if $\uPhi_\eqr(1) > \frac34 \e^\be$
then there exists $(\zeta,\si)$ such that
$\zeta \cdot \nu = r \si + \frac12 \e^\be$,
$\uPhi_\eqr(\si + \la \e) \leq \e^\be$
and
\begin{align*}
B_{\la \e}(\zeta, \si) \subset
\Omega(\uu_\eqr) \cap
Q_q \cap C^+ \cap \set{t \leq 1};
\end{align*}
\item
if $\lPhi_\eqr(1) > \frac34 \e^\be$
then there exists $(\zeta,\si)$ such that
$\zeta \cdot \nu = r\si - \frac12 \e^\be$,
$\lPhi_\eqr(\si + \la \e) \leq \e^\be$
and
\begin{align*}
B_{\la \e}(\zeta, \si) \subset
\Omega^c(\lu_\eqr) \cap
Q_q \cap C^+ \cap \set{t \leq 1}.
\end{align*}
\end{itemize}
\end{corollary}

\begin{proof}
Fix $q$ and $r$.
For $T = 1$
we have constants $\e_0$ and $K$ from Proposition~\ref{pr:cone-flatness}.
Let us define
\begin{align}
\label{def-eta}
    \eta := \inf \set{s : B_\la(s \nu) \subset \cone_{\nu,\ta^+}(0)}
        = \frac{\la}{\sin \ta^+}.
\end{align}
Because $\be <1$,
we may assume, making $\e_0$ smaller if necessary,
that for all $\e < \e_0$
\begin{align}
\label{shazam}
\max \set{r\la\e + \eta \e + K \e \abs{\ln\e}^{\frac12},
2\la \e  (M-m)\abs q}
< \frac18 \e^\be.
\end{align}
Lemma~\ref{le:furthest-point}
provides a point $(\hat \zeta,\hat \si) \in \partial \Omega(\uu_\eqr)$
such that $\hat \zeta = r \hat \sigma + \frac58 \e^\be$
and $\uPhi_\eqr(\hat\si) = \frac58 \e^\be$.
We define
\begin{align*}
\zeta &= \hat \zeta + (r\la\e - \frac18 \e^\be) \nu, &
\sigma &= \hat \sigma + \la \e.
\end{align*}
Clearly $\zeta \cdot \nu = r \sigma + \frac12 \e^\be$.
Furthermore,
due to \eqref{shazam}, the definition of $\eta$ in
\eqref{def-eta}
and Remark~\ref{rem:reversed-cone-flatness},
we have
\begin{align*}
B_{\la \e} (\zeta, \si)
\in \cone_{-\nu,\ta^+} (\hat \zeta - K \e \abs{\ln\e}^{\frac12}\nu)
\times [\hat \si, +\infty)
\subset \Omega(\uu_\eqr),
\end{align*}
see Figure~\ref{fig:ball-outlier}.
\begin{figure}
\centering
\fig{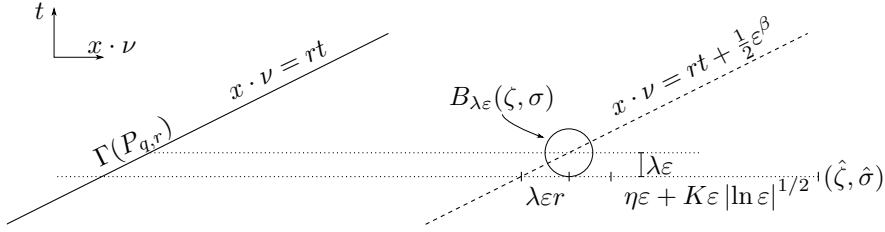}{4.5in}
\caption{Construction of the outlying ball in Corollary~\ref{co:ball-outlier}}
\label{fig:ball-outlier}
\end{figure}
Since $B_{\la\e}(\zeta,\si) \subset \Omega^c(\Pqr)$,
Proposition~\ref{pr:obstacle-sols}(d)
implies $B_{\la\e}(\zeta,\sigma) \subset Q_q \cap C^+$.

Finally,
Proposition~\ref{pr:phi-lipschitz}
and \eqref{shazam}
imply that
\begin{align*}
\uPhi_\eqr(\si+\la\e) =
\uPhi_\eqr(\hat\si+2\la\e)
&\leq \uPhi_\eqr(\hat\si) + 2\la\e(M\abs q - r)_+
\\&\leq \frac58\e^\be + 2\la\e(M-m)\abs q < \frac34 \e^\be.
\end{align*}
In particular, $\uPhi_\eqr(\si + \la\e) \leq \e^\be$
and $\si + \la \e \leq 1$, that is,
$B_{\la\e}(\zeta,\sigma) \subset \set{t\leq 1}$.
This finishes the proof for $\uu_\eqr$.
The proof for $\lu_\eqr$ is analogous.
\qedhere\end{proof}

\subsection{Detachment lemma}
\label{sec:detachment}

The main consequence of the cone flatness property of the previous section (Proposition~\ref{pr:cone-flatness}), combined
with the monotonicity of solutions
across scales (Proposition~\ref{pr:monotonicity}),
is the detachment lemma, presented in this section.
Heuristically speaking, we can conclude that if
the solution of the obstacle problem is not $\e^\be$-flat on some scale,
the boundary of the solution on a smaller scale is completely detached
from the obstacle
by a fraction of $\e^\be$
on a substantial portion of $Q_q$.

\begin{lemma}[Detachment]
\label{le:detachment-ball}
For every $q$ and $r$ satisfying \eqref{rqrestriction}
and $\be \in (0,1)$,
there exist $\ga \in (0,1)$ and $\e_0 > 0$ such that
for $\mu := \frac{1}{4\sqrt{1 +r^2}}$:
\begin{itemize}
\item
if
$\lPhi_\eqr(1) > \frac34\e^\be$ for some $\e < \e_0$
then
\begin{align}
\label{detachment-expansion-set-lu}
    \lu_{\ga \e; q,r} &= 0
    &\text{in} \quad &
    B_{\ga\mu \e^\be} (r \nu, 1);
\end{align}
\item
if
$\uPhi_\eqr(1) > \frac34\e^\be$ for some $\e < \e_0$
then
\begin{align}
\label{detachment-expansion-set-uu}
    \uu_{\ga \e; q,r} &> 0
    &\text{in} \quad &
    B_{\ga\mu\e^\be} (r \nu, 1).
\end{align}
\end{itemize}
\end{lemma}

\begin{proof}
We give the proof for $\uu_\eqr$.
The proof for $\ou_\eqr$ is analogous.
The basic idea of the proof is quite straightforward.
If $\uu_\eqr$ is not $\e^\be$-flat for sufficiently small $\e$,
the cone flatness proposition implies
that there must be a ball of sufficiently large radius $\la \e$
with a center
still $\frac{1}{2}\e^\be$ away from the obstacle in the direction of $\nu$,
and $\uu_\eqr$ is positive on this ball.
Then by looking at a smaller scale $\ga \e$,
$\ga < 1$,
and using the monotonicity property,
we can cover the ball given in \eqref{detachment-expansion-set-uu}
by the translates of the rescaled ball of radius $\ga \la \e$
and conclude that $\uu_{\ga \e; q, r}$ must be positive in that set.
These translates cover the ball $B_{\gamma\mu\e^\be}(r\nu,1)$.

For any $\ga \in (0,1)$, let us define the sets
\begin{align*}
E &:= \set{(x,t) : x\cdot\nu = rt + \frac12 \e^\be} \cap
Q_q \cap \set{t \leq 1}, \\
H &:= \set{(x,t) : x\cdot \nu \leq r t},\qquad
S_\ga :=  \bra{(1- \gamma) \cl C^+} \cap \set{t \geq 0}.
\end{align*}
Recall that if $-(y,\tau) \in H \cap S_\ga$
then $(y,\tau)$ satisfies \eqref{y-in-taplus-cone} (with $a = \ga$).

We set $\la = \sqrt{n + 1}$ and
note that the choice of $\la$
guarantees $\Z^{n+1} + B_\la = \R^{n+1}$.
We set $\mu = \frac1{4\sqrt{1 + r^2}}$
and observe that
\begin{align*}
B_{\ga\mu\rho}(r\nu,1) \subset
B_{\ga\mu\rho}(r\nu,1) + H \subset H + \frac{\gamma\rho}4 (\nu,0)
\end{align*}
for any $\rho > 0$.

Take $\e_0$ from Corollary~\ref{co:ball-outlier}.
By making $\e_0$ smaller if necessary,
we can assume that
\begin{align*}
\la \e \leq \mu \e^\be \qquad \text{for } \e < \e_0.
\end{align*}
This guarantees,
since $-\gamma E \subset H - \frac\gamma2\e^\be(\nu,1)$,
that
\begin{align}
\label{gamma-cons-1}
B_{\ga\mu \e^\be}(r\nu, 1) + B_{\ga \la \e}(0,0)
        - \ga E
        =B_{\ga\mu \e^\be + \ga \la \e}(r\nu, 1)
        - \ga E
        \subset H
\end{align}
for all $\ga > 0$, $\e < \e_0$.

We claim that we can find $\ga \in (0,1)$ such that,
for all $\e \in (0,\e_0)$,
\begin{align}
\label{gamma-cons-2}
    B_{\ga\mu \e^\be}(r\nu, 1) + B_{\ga \la \e}(0,0)
        - \ga E
        \subset S_\ga.
\end{align}
This is indeed possible by choosing $\ga$ small,
because
the set on the left-hand side is bounded,
contained in a ball $B_{\ga R}(r\nu, 1)$ for some $R > 0$,
independent of $\e < 1$,
and $(r \nu, 1)$ is in the interior of $S_\ga$ for all $\ga < 1$.
Such choice of $\ga$ and $\e_0$
that provides \eqref{gamma-cons-1} and \eqref{gamma-cons-2}
then assures by Lemma~\ref{le:grid-cover} that
\begin{align}
\label{choice-gamma-cons}
B_{\ga \mu \e^\be} (r\nu, 1) &\subset (S_\gamma \cap H)^{\grid\ga\e}
+ B_{\ga\la\e}(0,0)
+\ga(\zeta,\si)
\end{align}
for any $\e< \e_0$ and $(\zeta, \si) \in E$,
where the notation $\grid\ga\e$ is introduced in
Definition~\ref{def:grid}.

By Corollary~\ref{co:ball-outlier} we can find $(\zeta,\si) \in E$ such that
\begin{align}
\label{z-sigma}
\uu_\eqr &> 0 \text{ in }
B_{\la\e}(\zeta,\si).
\end{align}
Additionally, any $(y,\tau) \in -(S_\ga \cap H)^{\grid\ga\e}$
satisfies the condition \eqref{y-in-taplus-cone}
with $a = \ga$
and also $(y,\tau) \in \ga\e \Z^{n+1}$.
Hence, by the monotonicity proposition~\ref{pr:monotonicity},
\begin{align*}
\uu_{\ga \e; q, r} (x,t)
    \geq \ga\uu_\eqr\pth{\frac{x + y}{\ga}, \frac{t + \tau}{\ga}}
    > 0
\end{align*}
for all $(x,t)$ such that
$\ga^{-1}\pth{x + y, t + \tau} \in B_{\la \e}(\zeta, \si)$;
this can be rewritten as
\begin{align*}
(x,t) \in (S_\gamma \cap H)^{\grid\ga\e} + B_{\ga \la \e}(0,0)
+ \ga (\zeta,\si).
\end{align*}
Therefore $\uu_{\ga\e;q,r} > 0$ in $B_{\ga\mu\e^\be}(r\nu,1)$
by \eqref{choice-gamma-cons}.
\qedhere\end{proof}

\subsection{Homogenized velocity}
\label{sec:homogenized-velocity}

By now we have laid enough groundwork to be able to introduce reasonable candidates for
the homogenized velocity $r(Du)$.
The two main items in our toolbox, the local comparison principle (Theorem~\ref{th:localComparison})
and the detachment lemma (Lemma~\ref{le:detachment-ball}),
both work at the critical flatness rate $\e^\beta$ for $\beta$ in the interval $(\frac45, 1)$.

We therefore fix an exponent $\be \in (\frac45, 1)$ in the rest of the paper
and make the educated guess that if $r$ is a good candidate for the homogenized velocity at
a given gradient $q \in \R^n \setminus \set0$,
then the flatness of both $\uu_\eqr$ and $\ou_\eqr$ decreases \emph{faster} than $\e^\beta$ as $\e \to 0+$.
This consideration leads us to the two candidates for the homogenized velocity:
given $q \in \Rn \setminus \set0$,
we define the lower and upper homogenized velocities
\begin{subequations}
\label{homogenized-velocities}
\begin{align}
\lr(q) &:= \sup \set{r > 0 :
    \limsup_{\e\to0} \e^{-\be} \lPhi_\eqr(1) < 1},\\
\ur(q) &:= \inf \set{r > 0 :
    \limsup_{\e\to0} \e^{-\be} \uPhi_\eqr(1) < 1}.
\end{align}
\end{subequations}
If $q = 0$ we set $\lr(0) = \ur(0) = 0$.
Heuristically speaking,
if $r > \lr(q)$ then
$\ou_\eqr$ does not have $\e^\be$-flat free boundary
up to time $1$ for a subsequence of $\e$ converging to $0$.
Similarly, if $r < \ur(q)$ then
$\uu_\eqr$ does not have $\e^\be$-flat free boundary up to time $1$
for a subsequence of $\e$ converging to $0$.

In this and in the following subsection
we show that the two candidates $\lr$ and $\ur$ give rise to the same homogenized velocity $r$
(in the sense of Proposition~\ref{pr:r-semicontinuity}).
In this section we show that they are well defined and satisfy some basic monotonicity properties.

\begin{proposition}
\label{pr:bound-on-r}
For all $q \in \Rn$, the homogenized velocities
$\overline{r}(q)$ and $\underline{r}(q)$
introduced in \eqref{homogenized-velocities}
are well defined.
Moreover,
\begin{align*}
\overline{r}(q) &\in [m\abs{q}, M\abs{q}],
    & \underline{r}(q) &\in [m\abs{q}, M\abs{q}].
\end{align*}
\end{proposition}

\begin{proof}
We will show the statement for $\lr(q)$,
the proof for $\ur(q)$ is similar.

First, we consider $r \in (0, m\abs{q}]$.
Due to Proposition~\ref{pr:planar-solution-range},
$\Pqr \in \subs(g^\e, Q_q)$
and therefore $\lu_\eqr = \Pqr$, by definition.
In particular, $\lPhi_\eqr  \equiv 0$
and we conclude that $\lr(q) \geq m\abs{q}$.

On the other hand, if $r > M \abs{q}$,
we can find $r_2$, $q_2$ such that
$q_2 = a^3 q$ for some $a > 1$ and $M \abs{q_2} = r_2 < r_1$.
Since $P_{q_2,r_2} \in \supers(g^\e, Q_q)$,
we now have $\uu_\eqr = P_{q_2, r_2}$, which is $\e^\be$-flat
for all $\e > 0$.
Therefore, by the local comparison, Theorem~\ref{th:localComparison},
there exists $\e_0$ such that
$\ou_\eqr$ cannot be $\e^\be$-flat for any $\e < \e_0$.
We conclude that $\lr(q) \leq M \abs{q}$.
\qedhere\end{proof}

$\lr$ and $\ur$ also enjoy the following monotonicity property.

\begin{proposition}
For any $a > 1$ and $q \in \Rn$ we have
\begin{align*}
\lr(q) &\leq \lr(a q) & \ur(q) &\leq  \ur(aq).
\end{align*}
\end{proposition}

\begin{proof}
The statement is trivial for $q = 0$.
Hence we assume that $q \neq 0$
and $a > 1$.
Suppose that $\lr(q) > \lr(aq)$.
By definition of $\lr$, there exists $r$ with
$\lr(q) \geq r > \lr(aq)$ and $\e_0 > 0$ such that
$\limsup_{\e\to0}\e^{-\be}\lPhi_\eqr(1) < \ta < 1$
for some $\ta$.
Observe that due to Proposition~\ref{pr:nonuniform-perturbation}
we have $a \ou_\eqr \in \subs(g^\e, Q_q)$,
and clearly $a \ou_\eqr \leq a P_{q,r} = P_{aq,r}$.
Therefore, by definition of $\ou_{\e;aq,r}$ and flatness of $\ou_\eqr$,
\[
P_{aq, r}^{-\e^\be \ta} = a P_{q,r}^{-\e^\be\ta} \leq a \ou_\eqr
    \leq \ou_{\e;aq, r} \qquad \text{in } \set{t \leq 1},\ \e \ll 1,
\]
and we conclude that
$\limsup_{\e\to0}\e^{-\be}\lPhi_{\e;aq,r}(1) \leq \ta < 1$ and therefore $r \leq \lr(aq)$, a contradiction.
The proof for $\ur$ is analogous.
\qedhere\end{proof}

However,
the local comparison principle,
Theorem~\ref{th:localComparison},
also provides some ordering between $\ur$ and $\lr$.

\begin{corollary}
\label{co:mixed-monotone}
For any $q \in \Rn$ and $a > 1$ we have
\begin{align*}
    \lr(q) \leq \ur(aq).
\end{align*}
\end{corollary}

\begin{proof}
Again, the ordering is obvious if $q = 0$.
Therefore let us assume that $q \neq 0$ and $a > 1$.
Suppose that the ordering does not hold, i.e.,
$\ur(aq) < \lr(q)$.

Set $q_1 = q$ and $q_2 = a q$.
By the definition of $\lr$ and $\ur$,
there are $r_1$ and $r_2$ with
\begin{align*}
    \ur(a q) < r_2 < r_1 < \lr(q)
\end{align*}
such that
\[
\limsup_{\e\to0} \e^{-\be} \max \pth{\lPhi_{\e; q_1, r_1}(1), \uPhi_{\e; q_2, r_2}(1)} < 1.
\]
But that is in contradiction with the conclusion of
Theorem~\ref{th:localComparison}.
\qedhere\end{proof}

\subsection{Semi-continuity}
\label{sec:semi-continuity}

In this section we prove that
$\ur(q) \leq \lr(q)$,
and that in fact $\lr$ and $\ur$ are semi-continuous.

\begin{proposition}
\label{pr:r-semicontinuity}
The functions $\lr$ and $\ur$ are semi-continuous on $\Rn$ and
\begin{align*}
\lr = \ur^* &\in USC(\Rn), & \ur = \lr_* &\in LSC(\Rn).
\end{align*}
Furthermore, $a \mapsto \lr(a q)$ and $a \mapsto \ur(aq)$ are nondecreasing
on $(0, \infty)$ for any $q \in \Rn$.
(Recall that we set $\lr(0) = \ur(0) = 0$.)
\end{proposition}

To prove this proposition, we start with the following geometric result.

\begin{lemma}
\label{le:halfContinuity}
For every $q \in \Rn$, $\eta > 0$ there exists $\de > 0$ such that
\begin{align*}
\underline{r}(p)& \leq \overline{r}(q) + \eta, &
\underline{r}(q)& \leq \overline{r}(p) + \eta,
\end{align*}
for all $p \in \Rn$, $\abs{p - q} < \de$.
\end{lemma}

\begin{proof}
By the continuity of $\lr$ and $\ur$ at $0$ (Proposition~\ref{pr:bound-on-r}),
we can assume that $q \neq 0$.

We will prove the second inequality.
The proof of the first one is analogous.
Let us further assume that $\ur(q) - \eta > 0$.
We set $r = \ur(q) - \eta/2$ and $\hat r = \ur(q) - \eta = r - \eta/2 < r$.

Let $\la := \sqrt{n+1}$.
By the definition of $\ur(q)$,
there exists $\e >0$ small enough
such that Corollary~\ref{co:ball-outlier} applies
and there exists $(\zeta, \si)$
such that $B_{\la\e}(\zeta, \si) \subset \Omega(\uu_\eqr) \cap \set{0 \leq t \leq 1}$
and $\zeta \cdot \nu = r \si + \e^\be/2$.
Moreover, $\uPhi_\eqr(\si + \la \e) \leq \e^\be$.

Suppose that the conclusion does not hold and there exists
a sequence $\set{q_k}_k \subset \Rn$ such that $q_k \to q$
as $k \to \infty$ and $\ur(q) > \lr(q_k) + \eta$.
By the definition of $\lr$, there exists
$\e_k \in (0, \e)$, $\e_k \to 0$ such that
$\lPhi_{\e_k;q_k, \hat r} \geq \frac34 \e_k^\be$.
We will show that this leads to a contradiction with the comparison principle
for $k$ large enough.

Recall that we set $\nu = -q/\abs q$ and $\nu_k = -q_k/\abs{q_k}$.
By Lemma~\ref{le:furthest-point},
there exists $(\zeta_k,\si_k) \in \Gamma(\ou_{\e_k;q_k,\hat r})$
such that $\zeta_k \cdot \nu_k = \hat r \si_k - \frac12\e_k^\be$
and $\lPhi_{\e_k;q_k,\hat r}(\si_k) = \frac12 \e_k^\be$.

We introduce the scaling factor $\ga_k := \e_k/\e$.
By the choice of $\la$ we can find $(\xi_k, \tau_k) \in \e \Z^{n+1}$
such that
\begin{align*}
(X_k, T_k) := (\xi_k, \tau_k) + \ga_k^{-1} (\zeta_k, \si_k)
    \in B_{\la\e} (\zeta, \si).
\end{align*}
By compactness and taking a subsequence, we can assume that $(X_k, T_k) \to (X,T) \in \cl B_{\la\e}(\zeta, \sigma)$.
Let $v_k$ be the rescaled solution
\begin{align*}
v_k(x,t) = \ga_k^{-1} \ou_{\e_k;q_k,\hat r} (\ga_k(x - \xi_k), \ga_k(t - \tau_k)).
\end{align*}
Observe that $(X_k, T_k) \in \Gamma(v_k)$
and$(X_k, T_k) \in B_{\la \e}(\zeta, \si) \subset \Omega(\uu_\eqr)$.
If we prove that $\uu_\eqr \leq v_k$ in $B_{\la \e}(\zeta, \si)$
for large $k$,
we get a contradiction.

To that end, we will use the monotonicity of the obstacle problem
and the fact that $\lu_{\e_k;q_k,\hat r}$
is a solution away from the boundary of the obstacle.
Let $\Sigma_k := \ga_k^{-1} Q_{q_k} + (\xi_k, \tau_k)$.
Furthermore, let us define the translated planar solutions
\begin{align*}
P_k(x,t) &:= P_{q_k, \hat r}(x - \xi_k, t - \tau_k)&
    P^{-\rho_k}_k &:= P^{-\rho_k}_{q_k, \hat r}(x - \xi_k, t - \tau_k),
\end{align*}
where $\rho_k := \frac12 \ga_k^{-1} \e_k^\be = \frac12 \e_k^{\be -1} \e$.
Note that $\rho_k \to \infty$ since $\e_k \to 0$ and $\be < 1$.

We observe that
$v_k \in \sol(g^\e, \Sigma_k \setminus \Gamma(P_k))$.
Therefore it is enough to show that
\begin{compactenum}
\item
$v_k \geq P_{q,r}$
in $Q_q \cap \set{t \leq T_k}$,
\item
$\cl\Omega(\uu_\eqr)\cap \set{t \leq T_k} \cap \Gamma(P_k) = \emptyset$, and
\item
$\uu_\eqr \leq (v_k)_*$
on $\partial_P (Q_q \cap \Sigma_k \cap \set{t \leq T_k})$.
\end{compactenum}

By the properties of the obstacle problem,
Proposition~\ref{pr:obstacle-sols},
we have
\begin{align*}
\uu_\eqr &= \Pqr \quad\text{on } \partial Q_q &
    (v_k)_* = P_k \quad\text{on } \partial \Sigma_k.
\end{align*}
Moreover, the flatness assumptions provide the bounds
\begin{align*}
\uu_\eqr &\leq \Pqr^{\e^\be} & v_k&\geq P_k^{-\rho_k},
    \quad \text{ both in }
    \set{t \leq T_k}.
\end{align*}

Therefore, since
$\partial_P (Q_q \cap \Sigma_k \cap \set{t \leq T_k})
\subset (\partial Q_q \cup \partial \Sigma_k) \cap \cl Q_q$,
we only have to show that, for large $k$,
\begin{align}
\label{order1}
P_k > \Pqr^{\e^\be} \quad\text{in } \cl Q_q \cap \set{t\leq T_k}
    \cap \cl\Omega(\Pqr^{\e^\be})
\end{align}
and
\begin{align}
\label{order2}
P_k^{-\rho_k} > \Pqr \quad\text{in } \cl Q_q \cap \set{t\leq T_k}
    \cap \cl\Omega(\Pqr),
\end{align}
and (a)--(c) shall follow.
A straightforward calculation yields
\[
P_k^{-\rho_k}(x,t) = P_{q_k, \hat r}(x - X_k, t - T_k)
    \to P_{q,\hat r}(x - X, t - T), \quad {k \to\infty}
\]
where the convergence is locally uniform in $\Rn \times \R$.
A similar calculation yields
\[
P_k(x,t) = P^{\rho_k}_{q_k, \hat r}(x - X_k, t - T_k) \to \infty,
\quad k \to \infty
\]
locally uniformly, since $\rho_k \to \infty$.
Finally, we realize that
\[
P_{q,\hat r}(x - X, t - T) > \Pqr \quad\text{in } \cl Q_q \cap \set{t\leq \si + \la\e}
    \cap \cl\Omega(\Pqr),
\]
due to $\hat r < r$
and $X \cdot \nu \geq r T + \e^\be/4$,
and
we conclude that the orderings \eqref{order1} and \eqref{order2}
hold for sufficiently large $k$.
Therefore by the stitching lemma, Lemma~\ref{le:stitch},
and the obstacle monotonicity
we have that $\uu_\eqr \leq v_k$ in $B_{\la \e}(\zeta, \si)$
for large $k$ but that is a contradiction with the facts
that $\uu_\eqr > 0$ on $B_{\la\e}(\zeta,\si)$
and $(X_k, T_k) \in B_{\la\e}(\zeta,\si) \cap \Gamma(v_k)$.
The proof of the lemma is finished.
\qedhere\end{proof}

We can now proceed with the proof of semi-continuity.

\begin{proof}[Proof of Proposition~\ref{pr:r-semicontinuity}]
Let us fix $q \in \R$ and $\eta > 0$,
and let $\de > 0$ be from Lemma~\ref{le:halfContinuity}.
We will show that $\lr$ is USC at $q$.
First, we observe that whenever $p \in \Rn$ and $\abs{p - q} < \de$,
there exists $a > 1$ such that $\abs{a p - q} < \de$.
Therefore, for such $p$ and $a$,
Corollary~\ref{co:mixed-monotone}
and Lemma~\ref{le:halfContinuity} yield
\[
\lr(p) - \eta \leq \ur(a p) -\eta \leq \lr(q).
\]
Since $\de > 0$ exists for arbitrary $\eta > 0$, we conclude that $\lr$
is $USC$ at $q$ for all $q \in \Rn$, and thus $\lr \in USC(\Rn)$.
We can similarly show that $\ur \in LSC(\Rn)$.

Furthermore, we have $\ur \leq \lr$ on $\Rn$ by Lemma~\ref{le:halfContinuity}
which together with the semi-continuity implies
\[
\ur \leq \lr_*, \qquad\qquad \ur^* \leq \lr.
\]
On the other hand, Corollary~\ref{co:mixed-monotone} yields
$\lr(a^{-1} q) \leq \ur(q) \leq \lr(q) \leq \ur(aq)$
for any $a > 1$ and therefore
\[
\lr_*(q) \leq \liminf_{a \to1+} \lr(a^{-1} q) \leq \ur(q)
    \leq \lr(q) \leq \limsup_{a \to1+} \ur(aq) \leq \ur^*(q)
    \qquad q \in \Rn.
\]

Finally, the monotonicity of $\lr(aq)$ and $\ur(aq)$
in $a$ is a simple consequence of
\[
\ur(a_1 q) \leq \lr(a_1 q) \leq \ur(a_2 q) \leq \lr(a_2 q)
    \qquad 0 < a_1 < a_2,\ q \in \Rn.
\]
This finishes the proof.
\qedhere\end{proof}

\subsection{Time-independent medium}
\label{sec:time-independent}

In this section, we show that in a time-independent medium,
i.e., when $g(x,t) = g(x)$,
we can in fact prove that $\lr$ and $\ur$ are continuous and therefore coincide, and that they are positively one-homogeneous.
We thus obtain the same results as in \cite{K07}.
We are able to achieve this because of the key feature
of a time-independent medium, that is,
the availability of scaling in time
independently of space.
Indeed,
it is trivial to check that
if $u \in \sol(g^\e)$ then also $[(x,t) \mapsto a u(x, a t)] \in \sol(g^\e)$
for any $\e > 0$ and $a > 0$.

The main result of this section is the following proposition.
\begin{proposition}
\label{pr:timeind}
Suppose that $g(x,t) = g(x)$.
Then
\begin{align}
\label{lrur-coincide}
\lr = \ur \in C(\Rn).
\end{align}
and
\begin{align*}
\lr(q) = \ur(q) = \abs{q} \ur\pth{\frac{q}{\abs{q}}}
\qquad \text{for all $q \in \Rn \setminus \set0$.}
\end{align*}
\end{proposition}

Let us first show the continuity.
We will need the following scaling result.

\begin{lemma}
\label{le:flatness-scaling}
Suppose that $g(x,t) = g(x)$.
Then for any $q \in \Rn$ and $\e, r > 0$
we have
\begin{align}
\label{flatness-scaling}
\uPhi_{\e; aq, ar} (t) \leq a\uPhi_{a^{-1}\e;q,r}(t)
\qquad \text{for all $a>1$ and $t\in\R$.}
\end{align}
\end{lemma}

\begin{proof}
Since the statement is trivial for $q = 0$,
we assume $q \neq 0$.
Fix $a > 1$ and define
\[
v(x,t) = a^2 \uu_{a^{-1} \e; q, r}(a^{-1} x, t).
\]
Since $\uu_{a^{-1} \e; q, r} \in \supers(g^{a^{-1} \e}, Q_q)$,
the scaling of the problem and the independence of $g$ on time
yield
that $v \in \supers(g^\e,a Q_q)$.
Moreover,
\[
v(x,t) \geq a^2 P_{q,r}(a^{-1}x,t) = P_{aq,ar}(x,t).
\]
Therefore, by the definition of $\uu_{\e; aq, ar}$
and $Q_q \subset a Q_q$, we get
\[
\uu_{\e; aq, ar} \leq v \qquad \text{in $\Rn\times\R$.}
\]
By the definition of flatness, we have
for any $\tau \in \R$ and $t \leq \tau$
\begin{align*}
\uu_{\e; aq, ar}(x,t) \leq
v(x,t) &\leq a^2P_{q,r}(a^{-1} x - \uPhi_{a^{-1}\e;q,r}(\tau)\nu, t)
\\&= P_{aq,ar}(x - a\uPhi_{a^{-1}\e;q,r}(\tau)\nu, t).
\end{align*}
Thus \eqref{flatness-scaling} follows.
\qedhere\end{proof}

\begin{proof}[Proof of Proposition~\ref{pr:timeind}, continuity]
Fix $q \in \Rn$.
Since $\lr(0) = \ur(0)$ by definition,
we shall assume that $q \neq 0$.
By Lemma~\ref{le:halfContinuity} and
Corollary~\ref{co:mixed-monotone}
we infer
\[
\ur(q) \leq \lr(q) \leq \ur(a q) \qquad \text{for all } a > 1.
\]
Now fix $\de > 0$.
By definition of $\ur(q)$, there exists
$r \in [\ur(q), \ur(q) + \de/2)$
such that $\limsup_{\e\to0} \e^{-\be} \uPhi_\eqr(1) = \ta < 1$.
Now Lemma~\ref{le:flatness-scaling} for any $a \in (1, \ta^{1/(\be-1)})$
yields
\begin{align*}
\limsup_{\e\to0} \e^{-\be} \uPhi_{\e; aq, ar}(1)
\leq \limsup_{\e\to0} a^{1-\be}(a^{-1} \e)^{-\be} \uPhi_{a^{-1}\e;q,r}(1)
= a^{1-\be}\ta < 1.
\end{align*}
Therefore if also $a < 1 + \de/2r$,
the definition of $\ur(aq)$ implies
\begin{align*}
\ur(aq) \leq ar \leq r+\de/2 \leq \ur(q) + \de.
\end{align*}
Since $\de > 0$ was arbitrary,
we have $\lr(q) = \ur(q)$.
The semi-continuity result of Proposition~\ref{pr:r-semicontinuity}
implies $\lr = \ur \in C(\Rn)$.
\qedhere\end{proof}

For the proof of the one-homogeneity,
we will use the following scaling result.

\begin{lemma}
\label{le:phi-scaling}
Suppose that $g(x,t) = g(t)$.
Then for any $q \in \Rn$ and $\e, r > 0$ we have
\begin{align}
\label{uPhi-scaling}
\uPhi_\eqr(t) = \uPhi_{\e; aq, ar}(a^{-1} t) \qquad
\text{for all $a > 0$, $t \in\R$.}
\end{align}
The same holds for $\lPhi$.
\end{lemma}

\begin{proof}
Let us first observe that a simple calculation yields
\begin{align}
\label{pqr-scaling}
a \Pqr^\eta(x, at) = P_{aq,ar}^\eta(x,t)
\qquad \text{for all $q \in \Rn$, $a,r>0$.}
\end{align}
Fix $a > 0$ and consider $v(x,t) = a \uu_\eqr(x,at)$.
Since $g$ is time-independent, we have
$v \in \supers(g^\e, Q_q)$
and \eqref{pqr-scaling} yields
$v \geq P_{aq,ar}$. Therefore,
by definition of $\uu_{\e;aq,ar}$,
we obtain
\begin{align}
\label{uu-1}
\uu_{\e;aq,ar}(x,t) \leq v(x,t) = a \uu_\eqr(x,at).
\end{align}
Similarly, we set $w(x,t) = a^{-1} \uu_{\e; aq, ar} (x,a^{-1} t)$
and observe that $w \in \supers(g^\e,Q_q)$
and $w \geq P_{q,r}$.
Therefore
\begin{align}
\label{uu-2}
\uu_\eqr (x,t) \leq w(x,t) = a^{-1} \uu_{\e; aq, ar} (x,a^{-1} t).
\end{align}
Since $a > 0$ was arbitrary, \eqref{uu-1} and \eqref{uu-2} yield together
\begin{align*}
\uu_\eqr(x,t) = a^{-1} \uu_{\e; aq, ar} (x,a^{-1} t).
\end{align*}
The scaling of the planar solutions in \eqref{pqr-scaling}
and the definition of flatness thus yield \eqref{uPhi-scaling}.

The same statement for $\lPhi$ can be obtained analogously.
\qedhere\end{proof}

We can now finish the proof of the main result of this section.

\begin{proof}[Proof of Proposition~\ref{pr:timeind}, homogeneity]
Let us fix $q \in \Rn \setminus\set0$
and $a \in (0,1)$.
Then Proposition~\ref{pr:phi-lipschitz} and Lemma~\ref{le:phi-scaling}
yield
\[
\lPhi_{\e;aq,ar}(1) \leq \lPhi_{\e;aq,ar}(a^{-1})
\leq \lPhi_\eqr(1).
\]
Thus the definition of $\lr$ implies
\[
\lr(aq) \geq a \lr(q),
\]
and an analogous argument with $\uPhi$ yields
\[
\ur(aq) \leq a \ur(q).
\]
Since $\ur = \lr$ by \eqref{lrur-coincide},
we can conclude that
\begin{align}
\label{urlr-scale}
\ur(aq) = \lr(aq) = a\ur(q) = a \lr(q)\qquad\text{for all $a \in (0,1)$.}
\end{align}
Since $q \in \Rn\setminus\set0$ was arbitrary,
we can define $q' = aq$ and show that
\[
a^{-1}\ur(q') = a^{-1}\lr(q') = \ur(a^{-1}q') = \lr(a^{-1}q')\qquad\text{for all $a \in (0,1)$.}
\]
Therefore \eqref{urlr-scale} holds for all $q \in \Rn$ and $a > 0$.
In particular, if $q \neq 0$, for $a = 1/\abs{q}$ which concludes the proof.
\qedhere\end{proof}

\section{Convergence in the homogenization limit}
\label{sec:convergence}

In the view of Proposition~\ref{pr:r-semicontinuity},
we define the homogenized Hele-Shaw problem
\begin{align}
\label{hs-homogenized}
\begin{cases}
-\Delta u = 0, & \text{in } \set{u > 0} \cap Q,\\
V = \overline{r}(D u), & \text{on } \partial \set{u > 0} \cap Q,
\end{cases}
\end{align}
with the appropriate domain $Q:= \Omega \times (0,T]$
and the initial and boundary data $\Omega_0, \psi$
as in Theorem~\ref{th:well-posedness}.
The viscosity solutions of this problem were introduced in
Section~\ref{sec:viscosity-solutions}.

We show that the solutions of the inhomogeneous Hele-Shaw problem \eqref{HSt}
converge as $\e \to 0$
to the unique solution of the homogenized problem \eqref{hs-homogenized}
with the same initial and boundary data,
in the sense of half-relaxed limits (denoted $\halfliminf$ and $\halflimsup$ here, Definition~\ref{def:half-relaxed} below).
Since the solution of the homogenized problem might be discontinuous,
a half-relaxed limit is the standard mode of convergence
in such setting.
In particular, we prove the following homogenization result.

\begin{theorem}[Homogenization]
\label{th:homogenization}
Suppose that the domain $Q := \Omega \times (0,T]$
and the initial and boundary data satisfy the assumptions of
Theorem~\ref{th:well-posedness}.
Let $u^\e \in \sol(g^\e, Q)$ be the unique solution of
problem \eqref{HSt} for $\e>0$
and let $u \in \sol(\lr(p)/\abs{p}, Q)$ be the unique solution
of \eqref{hs-homogenized}.

Then $u_* = \halfliminf_{\e\to0} u^\e$
and $u^* = \halflimsup_{\e\to0} u^\e$.
Moreover, $\partial \Omega((u^\e)_*;Q)$ converge to
$\partial \Omega(u_*;Q)$ uniformly in Hausdorff distance.
\end{theorem}

Before proceeding with the proof of Theorem~\ref{th:homogenization},
we recall the notion of half-relaxed limits and the notion of convergence in
Hausdorff distance.

\begin{definition}
\label{def:half-relaxed}
Suppose that $\set{v_\e}_{\e > 0}$
is a family of functions
$v_\e : E \to \R$ on some set $E \subset \Rn\times\R$.
For $(x,t) \in \cl E$ we define
the \emph{upper and lower half-relaxed limits}
\begin{align*}
\halflimsup_{\e\to0} v_\e(x,t) &:=
    \lim_{\eta\to0}
    \sup \set{v_\e(y,s): (y,s) \in E,\ \abs{y - x} + \abs{s - t} + \e \leq \eta}
\\
\halfliminf_{\e\to0} v_\e(x,t) &:=
    \lim_{\eta\to0}
    \inf \set{v_\e(y,s): (y,s) \in E,\ \abs{y - x} + \abs{s - t} + \e \leq \eta}.
\end{align*}
\end{definition}
Let us recall that $v_\e \rightrightarrows v$ uniformly as $\e\to0$
on a compact set $K \subset E$, where $v \in C(K)$, if
$ \halflimsup_{\e\to0} v_\e = \halfliminf_{\e\to0} v_\e = v$ on $K$.

\begin{definition}
\label{def:Hausdorff-convergence}
Let $E_\e, E \subset \Rn\times\R$ for $\e > 0$
be compact sets.
We say that $E_\e$ converges to $E$ uniformly in Hausdorff distance
as $\e \to 0$
if for every $\de > 0$ there exists $\e_0 > 0$ such that
\[
    E \subset E_\e + B_\de(0,0), \quad E_\e \subset E + B_\de(0,0)
    \quad \text{for } \e < \e_0.
\]
\end{definition}

The following characterization is well known \cite{CC06}.
\begin{proposition}
\label{pr:hausdorff}
Compact sets $E_\e \subset \Rn\times\R$
converge to a compact set $E\subset\Rn\times\R$ uniformly in Hausdorff
distance if and only if the following two statements are true:
\begin{compactenum}[(i)]
\item
if $(x_\e,t_\e) \in E_\e$ for all $\e > 0$, then every limit point of
$\set{(x_\e,t_\e)}_{\e>0}$ lies in $E$;
\item for every $(x,t) \in E$ there exists $(x_\e,t_\e) \in E_\e$
such that $(x_\e,t_\e) \to (x,t)$ as $\e\to0$.
\end{compactenum}
\end{proposition}

To prove Theorem~\ref{th:homogenization},
we first address the convergence of the supports,
which is important
for the stability of subsolutions of one-phase problems
such as \eqref{hs-f}.
Nondecreasing-in-time subsolutions
in fact satisfy $\cl\Omega(\overline u) = \limsup_{\e\to0} \cl\Omega(u^\e)$,
see \cite[Lemma 4.1]{K08},
but we will only need the following simpler statement
(we give the proof here for completeness).

\begin{lemma}
\label{le:convergence-of-supports}
Let $v^\e\in\subs(M, Q)$
be a locally uniformly bounded sequence
of subsolutions on the space-time cylinder
$Q = \Omega \times (0, T]$
for open domain $\Omega$ and $T > 0$,
and let $\overline v = \halflimsup_{\e\to0} v^\e$.
Moreover,
suppose that $(v^\e)^*$ are nondecreasing in time
and
$\cl\Omega_0(v^\e;Q) \subset \cl\Omega_0(\overline v;Q)$
for all $\e > 0$.

If $\e_k \to 0$ as $k\to\infty$,
$(x_k, t_k) \in \cl\Omega(v_{\e_k}; Q)$ for all $k$
and $(x_k, t_k) \to (\hat x, \hat t)$
then $(\hat x, \hat t) \in \cl\Omega(\overline v; Q) \cup \partial_P Q$.
\end{lemma}

\begin{proof}
Suppose that $\e_k \to 0$ as $k \to \infty$ and
$(x_k, t_k) \in \cl\Omega(v_{\e_k})$ such that
$(x_k, t_k) \to (\hat x, \hat t) \in \cl Q$.

We argue by contradiction and assume that
$(\hat x, \hat t) \notin \cl\Omega(\overline v) \cup \partial_P Q$.
Let us denote $H_\rho(\xi,\si) = B_\rho(\xi,\si) \cap \set{t \leq T}$ in the following.
In particular, there exists $\de > 0$ such that
$\cl H_{2\de}(\hat x, \hat t) \subset Q \setminus \cl \Omega(\overline v)$.
Since $(v_\e)^*$ is nondecreasing in time,
so is $\overline v$ and thus
\[
\cl B_{\de}(x_k) \cap \cl\Omega_0(v_\e) \subset
\cl B_{2\de}(\hat x) \cap \cl\Omega_0(\overline v) = \emptyset
\]
for $k$ large
so that
$\cl B_\de(x_k) \subset B_{2\de}(\hat x)$.
Hence the nondegeneracy property of subsolutions,
Corollary~\ref{co:hs-nondegeneracy}, applies and,
used with the monotonicity in time,
yields the estimate
\[
\sup_{H_{2\de}(\hat x, \hat t)} v_{\e_k} \geq
\sup_{H_{\de}(x_k, t_k)} v_{\e_k} \geq \frac{\de^2}{2nM t_k}
\to \frac{\de^2}{2nM\hat t} > 0 \quad \text{as } k \to \infty.
\]
But this is a contradiction with the property of $\halflimsup$,
namely that $v_{\e_k}$ converges to $0$ uniformly on
$\cl H_{2\de}(\hat x, \hat t)$ as $k \to \infty$.
The proof is complete.
\qedhere\end{proof}

\begin{corollary}
\label{co:finite-speed-stable}
Let $v_\e$ be a sequence of subsolutions
as in Lemma~\ref{le:convergence-of-supports}.
Then $\overline v = \halflimsup_{\e\to0}v_\e$
satisfies Definition~\ref{def:visc-test-sub}(i).
\end{corollary}

\begin{proof}
Suppose that $\overline v$ does not satisfy
Definition~\ref{def:visc-test-sub}(i)
and there exists
\[(\xi, \tau) \in (\cl\Omega(\overline v) \cap Q)
\setminus \cl{\Omega(\overline v) \cap \set{t < \tau}}.\]
Then for some $\delta > 0$ we have
\begin{align}
\label{ball-outside-support}
\cl B_{2\de}(\xi,\tau) \cap \set{t\leq T} \subset
Q \setminus \cl{\Omega(\overline v) \cap \set{t < \tau}}.
\end{align}
Let $K = \sup_{\Omega\times[0,\tau]} v_\e < \infty$
and find $h \in (0, \de)$ such that
\[
\sqrt{2nKM h} < \de.
\]
By the definition of $\halflimsup$ there exists
sequences $\e_k \to0$ and $(\xi_k, \tau_k) \to (\xi, \tau)$
as $k \to \infty$ such that
$(\xi_k, \tau_k) \in \Omega(v_{\e_k})$.
Therefore by Corollary~\ref{co:HS-expansion-speed}
we conclude that there exist $(x_k, t_k) \in \cl\Omega(v_{\e_k}) \cap \cl B_{2\de}(\xi,\tau) \cap \set{t \leq \tau - h}$.
By compactness, $(x_k, t_k) \to (\hat x, \hat t)
\subset \cl B_{2\de}(\xi,\tau) \cap \set{t \leq \tau- h}$
as $k \to \infty$
up to a subsequence,
and Lemma~\ref{le:convergence-of-supports} implies
that $(\hat x, \hat t) \in \cl \Omega(\overline v)$.
But that is a contradiction with \eqref{ball-outside-support}.
\qedhere\end{proof}

In the following proposition we show that the half-relaxed limits of the solutions of the inhomogeneous
Hele-Shaw problem \eqref{HSt} are solutions
of the homogeneous problem \eqref{hs-homogenized}.

\begin{proposition}
\label{pr:homogenization-interior}
Let $v_\e$ be a locally uniformly bounded sequence
of subsolutions of \eqref{HSt}
on the space-time cylinder $Q = \Omega \times (0,T]$
satisfying also the hypothesis of Lemma~\ref{le:convergence-of-supports}.
Then the upper half-relaxed limit
\begin{align*}
\overline{v} = \halflimsup_{\e \to 0} v_\e
\end{align*}
is a viscosity subsolution of the homogenized problem
\eqref{hs-homogenized} on $Q$.

Similarly, if $v_\e$ is a locally uniformly bounded sequence of supersolutions of \eqref{HSt} on $Q$ then the lower half-relaxed limit
\begin{align*}
\underline{v} = \halfliminf_{\e \to 0} v_\e
\end{align*}
is a supersolution of \eqref{hs-homogenized} on $Q$.
\end{proposition}

\begin{proof}
We will prove that $\overline{v}$
is a subsolution of \eqref{hs-homogenized}.
A parallel, simpler argument yields that $\underline{v}$ is a supersolution;
in this case there is no need to study the convergence of supports.

We shall use Definition~\ref{def:visc-test-sub}.
Its part (i) follows from Corollary~\ref{co:finite-speed-stable}.
To prove that $\overline v$ satisfies Definition~\ref{def:visc-test-sub}(ii) we shall argue by contradiction:
if $\overline{v}$ fails to be a subsolution at some point $P$ of
its free boundary,
it is possible to compare $v_\e$ with a \emph{rescaled translation} of the subsolution
of the obstacle problem
$\ou_{\hat\e;q,r}$ for some small $\e,\hat\e$ and appropriate $q$, $r$,
in the neighborhood of $P$
and it can be shown that this subsolution
must be detached enough from its obstacle to prevent the free boundary
from reaching $P$ in the limit $\e\to0$.
We accomplish this by translating the rescaled $v_\e$ into $w_\e$ in \eqref{w-epsilon} in such a way that the point $P$ is moved into the neighborhood of point $(r\nu,1)$ where the function $\ou_{\hat\e;q,r}$ is 0 due to the detachment lemma (Lemma~\ref{le:detachment-ball}).

First, we can assume that $v_\e$ are $USC$
since that does not change $\halflimsup$.
Let us therefore suppose that $\overline{v}-\phi$ has a local maximum at
$(\hat x, \hat t)$
in
$\cl\Omega(\overline v) \cap \set{t \leq \hat t}$
for some $\phi \in C^{2,1}_{x,t}$.
If $\overline{v}(\hat x,\hat t) > 0$ then a standard
stability argument
yields that $-\Delta \phi(\hat x, \hat t) \leq 0$
and thus $\overline{v}$
satisfies Definition~\ref{def:visc-test-sub}(ii-1).

Hence suppose that
$\overline{v}$ does not satisfy Definition~\ref{def:visc-test-sub}(ii-2)
and
$\overline{v}(\hat x, \hat t)=0$,
$-\Delta \phi(\hat x, \hat t) > 0$,
$q := D \phi(\hat x, \hat t) \neq 0$
and $\phi_t > \overline{r}(D\phi) \abs{D\phi}$ at $(\hat x, \hat t)$.
In this case we can find $r \in \R$
such that
\begin{align*}
    \frac{\phi_t}{\abs{D\phi}} > r > \overline{r}(q)
        = \overline{r}(D\phi) \qquad \text{at } (\hat x, \hat t).
\end{align*}
Let us first observe that we can assume
that $\phi \in C^\infty$.
In fact,
let us take
\begin{align*}
\psi(x,t) = \abs{q} r (t - \hat t) + q \cdot (x - \hat x)
    + \ov2 A (x - \hat x) \cdot (x - \hat x),
\end{align*}
where $A = D^2\phi(\hat x, \hat t) + \eta I$,
with $\eta > 0$ chosen small so that $-\Delta \psi = -\trace A
    = -\Delta \phi(\hat x, \hat t) - n \eta > 0$.
By the Taylor's theorem, there exists $\de > 0$
such that
$\phi(x,t) - \phi(\hat x, \hat t) \leq \psi(x,t)$
for $\abs{x - \hat x} \leq \de$, $\hat t - \de \leq t \leq \hat t$,
and the equality holds only at $(\hat x,\hat t)$.
Therefore $\overline{v} - \psi$ has
a strict maximum $0$ at $(\hat x, \hat t)$
in $\cl\Omega(\overline{v}) \cap \set{t \leq \hat t}
\cap \cl B_\de(\hat x,\hat t)$.

Consequently,
along a subsequence of $\e\to0$,
using Lemma~\ref{le:convergence-of-supports},
there exist $(x_\e,t_\e)$ such that
$(x_\e, t_\e) \to (\hat x, \hat t)$ as $\e\to0$
and $v_\e - \psi$ has a local maximum at $(x_\e, t_\e)$
in $\cl\Omega(v_\e) \cap \set{t \leq t_\e} \cap \cl B_\de(\hat x, \hat t)$.
To simplify the notation in the rest of the proof,
$\e$ only takes on the values from this subsequence.
Since $-\Delta \psi(x_\e, t_\e) > 0$,
we must have $v_\e(x_\e,t_\e) = 0$ and therefore
\begin{align}
\label{max-estimate}
v_\e(x,t) \leq \bra{\psi(x,t)- \psi(x_\e, t_\e)}_+
\qquad \text{on $\cl B_\de(\hat x, \hat t) \cap \set{t \leq t_\e}$.}
\end{align}

By the definition of $\lr(q)$ and the choice of $r > \lr(q)$,
there exists arbitrarily small $\e'$ such that
$\lPhi_{\e';q,r}(1) > \frac34 \e'^\be$.
Hence Lemma~\ref{le:detachment-ball} implies that
there exists arbitrarily small $\hat\e = \ga \e'$ such that
$\ou_{\hat \e; q, r} = 0$ in
$B_{c \hat \e^{\be}}(r\nu, 1)$, where $c$ is a positive constant
independent of $\hat \e$.
Let us fix one such $\hat \e$ small enough so that
\begin{align}
\label{choice-of-hate}
c \hat \e^{\be - 1} > r + 2 \sqrt{n} + 1,
\end{align}
and  set $h := \e/\hat \e$.

Denoting $q_\e = D\psi(x_\e, t_\e)$,
we can estimate
\begin{align}
\label{psi-taylor}
\begin{aligned}
\psi(x,t) - \psi(x_\e, t_\e) &=
    \psi_t(x_\e, t_\e)(t - t_\e)
    + D\psi(x_\e, t_\e) \cdot (x - x_\e)\\
    &\quad + \ov2 D^2\psi(x_\e, t_\e)(x -x_\e) \cdot (x - x_\e)\\
    &\leq \abs{q} r (t - t_\e)
     + q \cdot (x - x_\e)
     \\&\quad+ \abs{q_\e - q} \abs{x - x_\e}
     + \norm{A} \abs{x - x_\e}^2.
\end{aligned}
\end{align}

We introduce $E:= Q_q \cap C^+ \cap \set{t \leq 2}$
and
its diameter $d := \diam E$,
and set (recall that $\nu := -q/\abs{q}$)
\begin{align}
\label{ye-taue}
\begin{aligned}
\tilde y_\e &= x_\e - h r \nu
    + \norm{A} h^2 d^2 \abs{q}^{-1} \nu
    + \abs{q_\e - q}  h d \abs{q}^{-1}\nu
    + (r + \sqrt{n}) \e \nu,\\
\tilde \tau_\e &= t_\e - h.
\end{aligned}
\end{align}
Then we choose $(y_\e, \tau_\e) \in \e\Z^{n+1}$ such that
\begin{align*}
(y_\e, \tau_\e) \in \argmin_{(x,t) \in \e \Z^{n+1}}
                    \bra{\abs{x - \tilde y_\e} +
                    \abs{t - \tilde \tau_\e}}.
\end{align*}
Clearly
\begin{align}
\label{grid-error}
\abs{y_\e - \tilde y_\e} &\leq \sqrt{n} \e, &
\abs{\tau_\e - \tilde \tau_\e} &\leq \e
\end{align}
and therefore, recalling \eqref{choice-of-hate}
and $h = \e/\hat\e$,
we have for all $\e$ sufficiently
small
\begin{align*}
\dist((y_\e + h r \nu, \tau_\e + h), (x_\e, t_\e))
    &\leq (r + 2\sqrt{n} + 1) \e
    \\&\quad + \abs{q_\e - q}  h d \abs{q}^{-1}
             + \norm{A} h^2 d^2 \abs{q}^{-1}
    \\&\leq c \hat \e^{\be - 1} \e = c h \hat\e^\be.
\end{align*}
since $q_\e \to q$ as $\e \to 0$.
In particular,
\begin{align}
\label{xe-te-in-ball}
(x_\e, t_\e) \in B_{c h \hat\e^\beta} (y_\e + h r\nu, \tau_\e + h)
             \subset h E + (y_\e, \tau_\e)
\end{align}
for sufficiently small $\e$.
From this and the definition of $d = \diam E$,
we infer that $\dist((x,t), (x_\e, t_\e)) \leq hd$
for all $(x,t) \in hE + (y_\e, \tau_\e)$.
Thus we estimate,
using \eqref{psi-taylor}, \eqref{ye-taue} and \eqref{grid-error},
\begin{align*}
\abs{q} (r(t - \tau_\e) - (x - y_\e)\cdot\nu)
    &= \abs{q} r \pth{t - t_\e + t_\e - \tilde \tau_\e
            + \tilde \tau_\e - \tau_\e}\\
    &\quad  + q \cdot \pth{x - x_\e + x_\e - \tilde y_\e
            + \tilde y_\e - y_\e}\\
    &\geq \psi(x,t) - \psi(x_\e, t_\e)
        - \norm{A} \abs{x - x_\e}^2 - \abs{q_\e - q} \abs{x - x_\e}\\
        &\quad + \abs{q} r (-h) + hr \abs{q} \\
        &\quad + \norm{A} h^2 d^2
               + \abs{q_\e - q} h d + \pth{r + \sqrt{n}}\e \abs{q}\\
        &\quad - (\abs{q} r + \abs{q}\sqrt{n}) \e \\
    &\geq \psi(x,t)- \psi(x_\e,t_\e)
        \qquad \text{for all } (x,t) \in hE + (y_\e, \tau_e).
\end{align*}
And therefore \eqref{max-estimate} yields
\begin{align*}
v_\e(x,t)
    &\leq \bra{\psi(x,t)-\psi(x_\e,t_\e)}_+
    \leq P_{q,r}(x - y_\e, t - \tau_\e) &
&\text{on } \pth{h E + (y_\e, \tau_\e)} \cap \set{t \leq  t_\e}.
\end{align*}
Let us define
\begin{align}
\label{w-epsilon}
w_\e(x,t) =
\begin{cases}
h^{-1} v_\e
            \pth{h x + y_\e, h t + \tau_\e)}
            & (x,t) \in Q_q \cap C^+,\\
0     & (x,t) \in Q_q \setminus C^+.
\end{cases}
\end{align}
We observe that $w_\e \in \subs(g^{\e/h}, Q_q)$
and $w_\e \leq P_{q,r}$ on
$Q_q \cap \set{t \leq 1 + h^{-1}(\tilde\tau_\e - \tau_\e)}$.
Therefore, by definition, $w_\e \leq \lu_{\e/h;q,r} = \lu_{\hat\e;q,r}$
in $Q_q \cap \set{t \leq 1 + h^{-1}(\tilde\tau_\e - \tau_\e)}$.
But by the choice of $\hat\e$ we have
that $\lu_{\hat\e;q,r} = 0$ in $B_{c\hat\e^\be}(r\nu, 1)$,
and by scaling back, we have
\begin{align*}
v_\e = 0 \qquad \text{in $B_{ch\hat\e^\be}(y_\e + hr\nu, \tau_\e + h)
\cap \set{t \leq t_\e}$}.
\end{align*}
When we apply Definition~\ref{def:visc-test-sub}(i)
and \eqref{xe-te-in-ball},
we obtain $(x_\e, t_\e) \notin \cl\Omega(v_\e)$
and that is a contradiction with the choice of $(x_\e,t_\e)$.
\qedhere\end{proof}

In particular, if the solution of the homogenized problem
\eqref{hs-homogenized} is unique,
then the limit is the unique solution.

Now we have enough information to finish the proof of the main homogenization result.

\begin{proof}[Proof of Theorem~\ref{th:homogenization}]
We shall denote
$\ou = \halflimsup_{\e\to0} u^\e$
and $\uu = \halfliminf_{\e\to0} u^\e$.

Let $U \in \subs(m/2, Q)$ and $V \in \supers(M, Q)$
be the boundary barriers constructed
in the proof of Theorem~\ref{th:well-posedness}.
From that proof we know that
$U_* \leq u^\e \leq V^*$ for all $\e > 0$ and
hence $U_* \leq \uu \leq \ou \leq V^*$ on $Q$.
This implies that $\uu$ and $\lu$ have the correct boundary
data in the sense of Theorem~\ref{th:well-posedness}.
Since $\ou \in \subs(\lr(p)/\abs{p}, Q)$
and $\uu \in \supers(\lr(p)/\abs{p}, Q)$
by Proposition~\ref{pr:homogenization-interior},
the proof of uniqueness in Theorem~\ref{th:well-posedness} yields
$u_* \leq \uu \leq \ou \leq u^*$,
which implies $\uu = u_*$, $\ou = u^*$ using Corollary~\ref{co:regularity}.

Let us proceed with the proof of the
convergence of free boundaries.
Denote
$E_\e = \partial \Omega((u^\e)_*;Q)$
and $E = \partial \Omega(u_*;Q)$.
Clearly $E,E_\e$ are compact subsets of $\cl Q$
(see also the proof of Theorem~\ref{th:well-posedness}).
We point out that due to the choice of initial and boundary data,
we have that
\begin{align}
\label{boundary-decomp}
E = \Gamma(u_*;Q)
\cup (\partial \Omega \times [0,T])
\cup ((\Omega \cap \Omega_0) \times \set{0})
\subset \Gamma(u_*;Q) \cup \partial_P Q,
\end{align}
and similarly for $E_\e$
(see the proof of Theorem~\ref{th:well-posedness}).
Let us also recall that
$\cl\Omega(u_*) = \cl\Omega(u)$
by Corollary~\ref{co:regularity}.
Therefore $\partial \Omega(u_*) = \cl\Omega(u) \setminus
\Omega(u_*)$.
The same is true for $u^\e$.

We shall establish
the uniform convergence $E_\e \to E$ in
Hausdorff distance
using the characterization in Proposition~\ref{pr:hausdorff}.
Let us denote
\begin{align*}
H_\rho(\xi,\si) := B_\rho(\xi,\si) \cap \set{t\leq T},
\qquad (\xi,\si) \in \Rn \times \R, \rho > 0,
\end{align*}
in the following.

\textbf{(i)}
Suppose that $\e_k \to 0$ as $k \to\infty$
and $(x_k,t_k) \in E_{\e_k}$ such that $(x_k, t_k) \to (\hat x, \hat t)$.
Clearly $\hat t \leq T$.
Lemma~\ref{le:convergence-of-supports}
yields that $(\hat x,\hat t) \in \cl\Omega(u)$.
Furthermore, due to the behavior of $u^\e$
at the parabolic boundary of $\partial_P Q$,
we have $\hat x \in \Omega$,
and $(\hat x, \hat t) \in E$ if also $\hat t = 0$.
Let us therefore assume that $\hat t > 0$ and
$(\hat x, \hat t) \in \Omega(u_*)$.
There exists $\de > 0$ such that $\cl H_\de(\hat x, \hat t)
\subset \Omega(u_*)$.
Semi-continuity yields
$\eta = \min_{\cl H_\de(\hat x, \hat t)} u_* > 0$.
The property of $\halfliminf$ then implies
that $\min_{\cl H_\de(\hat x, \hat t)} (u^{\e_k})_* > \eta/2 > 0$
for $k$ sufficiently large.
But this is a contradiction with
$(x_k, t_k) \in H_\de(\hat x, \hat t) \cap \partial \Omega((u^\e)_*)$
for large $k$.

\textbf{(ii)}
Now suppose that there exist $(\hat x, \hat t) \in E$
and $\de > 0$ such that
for some $\e_k \to 0$ as $k \to \infty$
we have
$E_{\e_k} \cap \cl H_\de(\hat x,\hat t) = \emptyset$
for all $k$.
Again, since $u^\e$ and $u$ satisfy the same boundary conditions,
we must have $(\hat x, \hat t) \in Q$
by \eqref{boundary-decomp}.
Thus we can assume that $\cl H_\de(\hat x, \hat t) \subset Q$.

Suppose that $\cl H_\de(\hat x,\hat t) \cap
\cl \Omega(u^{\e_k}) = \emptyset$
for infinitely many $k$.
In this case $u = 0$ in $\cl H_\de(\hat x, \hat t)$
and we get a contradiction with the choice of $(\hat x, \hat t)
\in \partial \Omega(u_*)$.

Hence we must have $\cl H_\de(\hat x,\hat t) \subset \Omega((u^{\e_k})_*)$
for all but finitely many $k$.
We can assume for simplicity that this inclusion holds for all $k$.
We recall that the comparison with the barrier $U$
in the proof of Theorem~\ref{th:well-posedness}
yields that the support of $u$ strictly expands at $t = 0$.
Consequently, $\hat x \notin \cl\Omega_0(u; Q)$.
By taking $\de$ smaller if necessary,
we can assume that $\cl B_\de(\hat x) \cap \cl\Omega_0(u; Q) = \emptyset$.
Therefore $\cl B_\de(\hat x) \cap \cl\Omega_0(u^\e; Q) = \emptyset$
and $u^\e$ is nondegenerate in the sense of
Corollary~\ref{co:hs-nondegeneracy}.
Since $(u^\e)_* = ((u^\e)^*)_*$ is nondecreasing in time
and $(\hat x, \hat t- \de/2) \in
\cl\Omega((u^{\e_k})_*) = \cl\Omega(u^{\e_k})$,
we have
\[
\sup_{x \in B_{\de/2}(\hat x)} (u^{\e_k})_*(x,t) \geq 2nM
\frac{\de^2}{\hat t - \de/2}
\qquad
\text{for } t \in I := [\hat t -\de/2, \hat t + \de/2] \cap (0,T].
\]
Furthermore, we observe that
$B_{\sqrt3 \de/2}(\hat x) \in \Omega_t((u^{\e_k})_*)$
for $t \in I$
and, since $(u^\e)_*$ is harmonic in $\Omega_t((u^\e)_*)$
due the same argument as in the proof of Proposition~\ref{pr:harmonic},
the Harnack inequality yields
$\inf_{x \in B_{\de/2}(\hat x)} (u^{\e_k})_*(x,t) \geq \eta > 0$
for $t \in I$,
in particular $\inf_{H_{\de/2}(\hat x, \hat t)} u^{\e_k} \geq \eta$
with $\eta$ independent of $k$.
Therefore, by regularity, Corollary~\ref{co:regularity},
we have
\[
u_* = (u^*)_* = \pth{\halflimsup_{\e\to0}u^\e}_* \geq \eta > 0
\qquad \text{in } H_{\de/2}(\hat x, \hat t).
\]
But that is a contradiction with the choice of $(\hat x, \hat t)
\in \partial \Omega(u_*)$.
The proof of uniform convergence in Hausdorff distance is finished.
\qedhere\end{proof}

\appendix
\section{Hele-Shaw problem}

In this section we derive bounds on solutions of the Hele-Shaw problem
that can be obtained by a comparison with radially symmetric barriers.
Some of these estimates appeared previously
in various forms in \cite{K03,K07,KM09}. But the estimate
in Proposition~\ref{pr:solUpperBoundZeroSet},
for example, appears to be new.

We first construct radially symmetric barriers.

\begin{lemma}
\label{le:self-similar-solution}
Given constants $m>0$, $K > 0$, $A \in (0,1)$
there exists a self-similar (w.r.t. the parabolic scaling),
radially symmetric, exterior
solution $\vp \in \sol(m, \Rn \setminus\set0 \times (0,\infty))$
of the form $\vp(x,t) = \psi(x/\rho(t))$
such that $\vp(x,t) > 0$ for $0 < \abs{x} < \rho(t)$
and $\vp(x,t) = 0$ for $\rho(t) \leq \abs{x}$,
and $\vp(x,t) = K$ for $\abs{x} = A \rho(t)$,
$\lim_{t\to0+} \rho(t) = 0$
and $\rho(t) = \sqrt{\al K m t}$ where $\al = \al(n, A)$.
\end{lemma}

\begin{proof}
The solution $\vp$ has the form $\vp(x,t) = \psi(x/\rho(t))$,
where (formulas written in columns for $n \geq 3$ and $n = 2$)
\begin{align*}
n&\geq 3 & n &= 2\\
\psi(x) &= K\frac{\pth{\abs{x}^{2-n} - 1}_+}{A^{2-n} - 1} &
\psi(x) &= K\frac{\pth{\ln \ov{\abs{x}}}_+}{\ln\ov{A}},\\
\rho(t) &= \sqrt{\frac{2(n-2)}{A^{2-n}-1} K m t} &
\rho(t) &= \sqrt{\frac{2}{-\ln A} K m t }.
\end{align*}
One can easily verify that
$\vp \in \sol(m, \Rn\setminus\set0\times (0,\infty))$,
checking the free boundary condition
$\td{\rho}{t} = V_\nu = m \at{\abs{D\vp^+}}{\abs{x} = \rho}$.
\qedhere\end{proof}

The solutions constructed in Lemma~\ref{le:self-similar-solution}
allow us to obtain an upper bound on supersolutions
that are harmonic in their positive sets.

\begin{proposition}
\label{pr:solUpperBoundZeroSet}
There exists a dimensional constant $C = C(n)$
such that for any space-time cylinder $Q = E \times (t_1,t_2)$,
where $E \subset \Rn$ is an open set and $t_1 < t_2$,
and any supersolution $u \in \supers(m, Q)$ for some $m > 0$ such that
$u \in LSC(\cl Q)$,
$u$ is nondecreasing in time and harmonic in $\Omega_t(u;Q)$
for all $t \in (t_1,t_2)$, we have
\begin{align*}
u(x,t) \leq \inf \set{\frac{C}{m} \frac{\abs{x - \zeta}^2}{\si - t}:
    (\zeta,\si) \in Q,\ u(\zeta, \si) = 0,\ t < \si,\
        \abs{\zeta -x} < \dist(x, \partial E)}
\end{align*}
for all $(x,t) \in Q$.
\end{proposition}

\begin{proof}
The idea of the proof is straightforward: if $u(x,t)$ violates the bound
for some $(x,t), (\zeta, \si) \in Q$,
it is possible to put a barrier under $u$ that will
become positive at $(\zeta,\si)$, yielding a contradiction.

Set $C = \frac{1}{c_H \al}$,
where $\al = \al(n,A)$ is the constant from
Lemma~\ref{le:self-similar-solution} with $A = 1/4$,
and $c_H = c_H(n) > 0$ is the constant from the standard Harnack
inequality:
for any positive harmonic function $\psi$ on $B_1$
\begin{equation}
\label{harnack}
\inf_{B_{1/2}} \psi \geq c_H \sup_{B_{1/2}} \psi.
\end{equation}

Suppose that there exist $(\zeta, \si), (\xi, \tau) \in Q$
and $K > 0$
such that $u(\zeta, \si) = 0$,
$\tau < \si$, $\abs{\zeta -x} < \dist(x, \partial E)$
and
\begin{align}
\label{u-larger-KcH}
u(\xi,\tau) > \frac{K}{c_H}
    >  \frac{C}{m} \frac{\abs{\xi - \zeta}^2}{\si - \tau}.
\end{align}
We can in fact assume, by translation invariance,
that $(\xi,\tau) = (0,0)$.

Let $\vp \in \sol(m, (\Rn\setminus\set0) \times (0,\infty))$ and $\rho : [0,\infty) \to [0,\infty)$ be the self-similar,
radially symmetric solution
and radius of its free boundary, respectively,
from Lemma~\ref{le:self-similar-solution}
with parameters $m$, $K$ and $A = 1/4$.
Since $\rho(t) = \sqrt{\al K m t}$,
it follows from the second inequality in \eqref{u-larger-KcH}
that $\rho(\si) > \abs{\zeta}$,
and, consequently, $\vp(\zeta,\si) > 0$.

We want to show that $u(\zeta, \si) > \vp(\zeta,\si)$ to get a contradiction.
Because $u \in LSC$ and $u$ is nondecreasing in time,
there exists $\la > 0$ such that $u > \frac{K}{c_H}$ in
$\cl B_{2\la}(0) \times [0, t_2]$.
Clearly $\la < \abs{\zeta}$.
Let us find $\si_0 > 0$ such that $\rho(\si_0) < \la$
and let $k \in \N$ be the smallest number such that
$\rho\pth{\pth{\frac94}^k\si_0} \geq \abs{\zeta}$.
Set $\si_j = \pth{\frac94}^j\si_0$ for $1 \leq j < k$
and find $\si_k$ so that $\rho(\si_k) = \abs{\zeta}$.
We chose $\si_j$ in this way because now we have
\begin{align}
\label{rho-in-Qj}
A \rho(t) = \ov4 \rho(t) \leq \frac38 \rho(\si_{j-1})
< \frac12 \rho(\si_{j-1}) \qquad t \in [\si_{j-1},\si_j],\ j = 1, \ldots, k.
\end{align}
Therefore, if $1 \leq j \leq k$ and
$u(x, \si_{j-1}) > 0$ for all $x$ such that $\abs{x} \leq \rho(\si_{j-1})$,
then,
by the monotonicity in time and Harnack inequality \eqref{harnack},
\begin{align}
\label{u-lower-harnack}
u(x,t)
    \geq u(x,\si_{j-1})
    \geq \inf_{\abs{y} < \ov2 \rho(\si_{j-1})} u(y, \si_{j-1})
    \geq c_H \sup_{\abs{y} < \ov2 \rho(\si_{j-1})} u(y, \si_{j-1})
    > K
\end{align}
for $\abs{x} \leq A \rho(t) < \ov2 \rho(\si_{j-1})$
and $t \in [\si_{j-1}, \si]$.

Define
$Q_j = \set{(x,t) \in Q : \abs{x} > A \rho(t),\ t \in (\si_{j-1}, \si_j]}$.
By construction, $\vp \leq K$ on $\cl Q_j$.
Our goal is to show inductively that
\begin{align}
\label{vp-prec-u-on-Qj}
\vp \prec u \qquad \text{on $\cl Q_k$
w.r.t. $Q_k$},
\end{align}
which implies
$u(\zeta, \si) \geq u(\zeta, \si_k)
> \vp(\zeta, \si_k) > 0$, yielding a contradiction.

We will show the inductive step: if $2 \leq j \leq k$ and
$\vp \prec u$ on $\cl Q_i$ w.r.t. $Q_i$ for all
$i = 1, \ldots, j-1$ then $\vp \prec u$ on $\cl Q_j$ w.r.t. $Q_j$.

The goal is to apply the comparison theorem,
and therefore we need to show
$\vp \prec u$ on $\partial_P Q_j$ w.r.t. $Q_j$.
First, observe that (for any $j = 1, \ldots, k$)
we have
$\partial_P Q_j = C^1_j \cup C^2_j \cup C^3_j$,
where $C^i_j$ are the closed sets
\begin{align*}
C^1_j &:= \partial E \times [\si_{j-1}, \si_j],
    &&\text{outer boudary}\\
C^2_j &:= \set{(x,\si_{j-1}) \in \cl Q: \abs{x} \geq A\rho(\si_{j-1})},
    &&\text{initial boudary}\\
C^3_j &:= \set{(x,t) : \abs{x} = A\rho(t),\  t \in [\si_{j-1}, \si_j]}.
    &&\text{inner boudary}
\end{align*}

\begin{enumerate}[(1)]
\item
It is simple to see that $\vp \prec u$ on
$C^1_j$
w.r.t. $Q_j$ for $1 \leq j \leq k$
because $\rho(t) \leq \rho(\si_k) = \abs{\zeta} < \dist(0, \partial E)$
for $t \in [\si_{j-1}, \si_j]$
and therefore $\cl\Omega(\vp; Q_j) \cap C^1_j = \emptyset$.

\item
Since the support of $\vp$ expands continuously,
\begin{align}
\label{supp-of-phi-on-c2i}
\cl\Omega (\vp; Q_{i}) \cap C^2_i
    = \set{(x,\si_{i-1}):
        A \rho(\si_{i-1}) \leq \abs{x} \leq \rho(\si_{i-1})}
        \qquad i = 1,\ldots,k,
\end{align}
and
\begin{equation*}
\cl\Omega (\vp; Q_{i-1}) \cap C^2_i = \cl\Omega (\vp; Q_{i}) \cap C^2_i
    \qquad i = 2,\ldots, k.
\end{equation*}
In particular, $\vp \prec u$ in $\cl Q_{j-1}$ w.r.t. $Q_{j-1}$
implies $\vp \prec u$ in $C^2_j$ w.r.t. $Q_j$
since $C^2_j \subset \cl Q_{j-1}$.

\item
The induction hypothesis and (2)
implies $\vp \prec C^2_i$ w.r.t. $Q_i$ for $i = 1, \ldots, j$,
and thus \eqref{supp-of-phi-on-c2i} and the monotonicity of $u$ yield
\begin{align*}
u(x,t) \geq u(x, \si_{i-1}) > 0
    \quad t \in [\si_{i-1},t_2],\ A\rho(\si_{i-1}) \leq \abs{x}
    \leq \rho(\si_{i-1}).
\end{align*}
Since $A \rho(\si_i) < \rho(\si_{i-1})$ for $i = 1, \ldots, k$
and $\rho(\si_0) < \la$,
we conclude that $u(x, t) > 0$ for $\abs{x} \leq \rho(\si_{j-1})$
and $t \in [\si_{j-1}, t_2]$.
Therefore by \eqref{u-lower-harnack} we have
$u > K \geq \vp$ on $C^3_j$.
\end{enumerate}
We have showed in (1)--(3) that
$\vp \prec u$ on $\partial_P Q_j$ w.r.t. $Q_j$
and therefore the comparison principle, Theorem~\ref{th:comparison},
yields $\vp \prec u$ on $\cl Q_j$ w.r.t. $Q_j$,
finishing the proof of the inductive step.

Finally, to close the induction argument,
we need to show the base case, i.e.,
that $\vp \prec u$ on $\cl Q_1$ w.r.t. $Q_1$.
Part (1) still applies for $j = 1$ as well,
and $\vp \prec u$ on $C^2_1 \cup C^3_1$ w.r.t. $Q_1$
follows from the choice $A\rho(\si_1) < \rho(\si_0) < \la$.
Therefore the comparison theorem applies on $Q_1$.
Induction yields \eqref{vp-prec-u-on-Qj},
finishing the proof of the theorem.
\qedhere\end{proof}

Now we turn our attention to obtaining bounds for subsolutions.
The first step is a construction of barriers.

\begin{lemma}
\label{le:contracting-barrier}
Let $M$ and $\mu$ be positive constants
and let $\chi \in C((-\infty, 0))$, $\chi(t) > 0$.
Define $K(t) = \int_0^t \chi(s) \diff s$
and
\begin{align*}
t_0 = \inf \set{t < 0: K(t) > - \frac{\mu^2}{2nM}} \in [-\infty, 0).
\end{align*}
There exists a classical, radially symmetric
solution $\vp \in \sol(M, Q)$
on a space-time cylinder $Q = B_\mu \times (t_0, 0)$,
with boundary data $\vp(x,t) = \chi(t)$ for $\abs{x} = \mu$,
such that $\vp(x,t) > 0$ for $\rho(t) < \abs x \leq \mu$,
and $\vp(x,t) = 0$ for $\abs x \leq \rho(t)$,
where $\rho(t)$ is the radius of free boundary of $\vp$.
Moreover,
$\rho$ is strictly decreasing and $\rho(t) \in (0, \mu)$ for all
$t \in (t_0, 0)$.
Finally,
$\lim_{t \to 0-} \rho(t) = 0$.
\end{lemma}

\begin{proof}
The solution $\vp$ is given by the formulas (separately for $n \geq 3$
and $n = 2$)
\begin{align}
\nonumber
n&\geq 3 & n &=2\\
\label{eq:radialsolutions}
\vp(x,t) &= \chi(t)
    \frac{\pth{\rho(t)^{2-n} - \abs{x}^{2-n}}_+}
        {\rho(t)^{2-n} - \mu^{2-n}},&
\vp(x,t) &= \chi(t)
    \frac{\pth{\ln \frac{\abs{x}}{\rho(t)}}_+}{\ln \frac{\mu}{\rho(t)}},
\end{align}
and $\rho(t)$ is the unique solution of
\begin{align}
\label{rho-radial-equation}
\frac{\mu^2}{2-n}\pth{\frac{\rho(t)^2}{2\mu^2} - \frac{\rho(t)^n}{n\mu^2}}
    &= M K(t),
&\frac{\rho(t)^2}{2} \pth{\ln \frac{\rho(t)}\mu
    - \ov2} &= M K(t).
\end{align}
The equations for $\rho(t)$ are derived
by integrating $V_\nu = -\td{\rho}{t} = \at{\abs{D\vp^+}}{r = \rho}$,
using the condition $\rho(0) = 0$ and the expressions
\begin{align}
\at{\abs{D\vp^+}}{\abs{x} = \rho} &=
    \chi\frac{(n-2)\rho^{1-n}}{\rho^{2-n} - \mu^{2-n}},&
\at{\abs{D\vp^+}}{\abs{x} = \rho} &=
   \chi\frac{1}{\rho \ln \frac\mu\rho},
\end{align}
The left-hand side of the equation \eqref{rho-radial-equation}
is strictly increasing for $\rho \in (0,\mu)$,
with limit values $-\frac{\mu^2}{2n}$
as $\rho \to \mu$ and $0$ as $\rho \to 0$,
and therefore the equation has a unique solution in the set
$\rho \in (0,\mu)$
as long as $M K(t) \in (-\frac{\mu^2}{2n}, 0)$,
which holds for $t \in (t_0, 0)$ by the definition of $t_0$.
It is straightforward to check that $\vp \in \sol(M, Q)$.
\qedhere\end{proof}

Now we will use the radially symmetric barriers
of Lemma~\ref{le:contracting-barrier}
to obtain various estimates on subsolutions.

\begin{proposition}
\label{pr:HS-closing}
Let $M, \mu$ be positive constants
and let $\chi \in C([t_1, t_2])$ for some $t_1 < t_2$,
with $\chi > 0$.
Suppose that $\int_{t_1}^{t_2} \chi(s) \diff s < \frac{\mu^2}{2nM}$.

If $u \in \subs(M, Q)$,
where $Q$ is the cylinder
$Q = B_\mu(\zeta) \times (t_1, t_2)$ for some $\zeta \in \Rn$,
such that
$\cl\Omega_{t_1}(u;Q) = \emptyset$
and
$u^{*,\cl Q}(x, t) \leq \chi(t)$
for $(x,t) \in \partial B_\mu(\zeta) \times [t_1, t_2]$,
then
\begin{align*}
(\zeta, t) \notin \cl\Omega(u;Q)
    \qquad \text{for all $t \in [t_1, t_2]$}.
\end{align*}
\end{proposition}

\begin{proof}
Let us extend $\chi$ to $\R$ by
defining $\chi(t) = \chi(t_2)$ for $t > t_2$ and $\chi(t) = \chi(t_1)$
for $t < t_1$.
By continuity, we can find $\de > 0$ small enough such that
\[
\int_{t_1}^{t_2 + \de} \chi(s) + \de \diff s < \frac{\mu^2}{2n(M+\de)}.
\]
Let $\vp \in \sol(M+\de, B_\mu(0) \times (t_0, 0))$ be the contracting solution provided by
Lemma~\ref{le:contracting-barrier}
with $\tilde \chi(t) = \chi(t + t_2 + \de) + \de$
and $\tilde M = M + \de$,
where $t_0 < t_1 -t_2 - \de$,
and let $\rho(t)$ be the radius of its free boundary.
Define
\[\psi(x,t) = \vp(x - \zeta, t - t_2 - \de) - \kappa (\abs{x - \zeta}^2 - \rho^2(t -t_2 -\de))_+,\]
where $\kappa \in (0, \de/\mu^2)$ is chosen small enough so that
$\psi$ is a superbarrier on $Q$ and $\psi > M$
on $(\partial B_\mu(\zeta)) \times [t_1, t_2]$.
Observe that $u \prec \psi$ on $\partial_P Q$ w.r.t. $Q$.
and hence Definition~\ref{def:visc-barrier}
implies that $u \prec \psi$ on $\cl Q$ w.r.t. $Q$ and the result
follows since $(\zeta, t)\notin \cl\Omega(\psi;Q)$
for all $t \in [t_1,t_2]$ because $\rho(t - t_2 - \de)\geq \rho(-\de) > 0$.
\qedhere\end{proof}

Proposition~\ref{pr:HS-closing} has the following consequence for the bound of a supersolution on a cylinder.

\begin{corollary}[Nondegeneracy]
\label{co:hs-nondegeneracy}
Let $u \in \subs(M, Q)$,
where $Q$ is the cylinder $Q := B_\mu(\zeta) \times (t_1,t_2)$
for some $\zeta \in \Rn$, $\mu > 0$ and $t_1 < t_2$,
such that $(\zeta,t_2) \in \cl\Omega(u;Q)$
and $\cl\Omega_{t_1}(u;Q) = \emptyset$.
Then
\begin{align*}
\sup_Q u \geq \frac{\mu^2}{2n M (t_2 - t_1)}.
\end{align*}
\end{corollary}

\begin{proof}
If $K := \sup_Q u < \frac{\mu^2}{2n M (t_2 - t_1)}$,
we can apply Proposition~\ref{pr:HS-closing}
on $Q$
with $\chi(t) = K$,
yielding $(\zeta, t_2) \notin \cl\Omega(u;Q)$,
a contradiction.
\qedhere\end{proof}

The support of subsolutions expands with a speed that can be controlled with the help of Proposition~\ref{pr:HS-closing}.

\begin{corollary}
\label{co:HS-expansion-speed}
If $u \in \subs(M, Q)$, where $Q = E \times (t_1, t_2)$ for some
open set $E \subset \Rn$ and $t_1< t_2$, and $u \leq K$ on $Q$ for some $K > 0$,
then
\begin{align*}
\cl\Omega_t(u; Q)
    \subset  \cl\Omega_{t_1}(u; Q) \cup E^c + \cl B_{\rho(t)}
        \qquad \text{for all $t \in [t_1, t_2]$},
\end{align*}
where
\begin{align*}
\rho(t) = \sqrt{2n K M (t - t_1)}.
\end{align*}
\end{corollary}

\begin{proof}
The statement will be proved if we show
that $(\zeta, \si) \notin \cl\Omega(u;Q)$
for any $(\zeta, \si) \in \cl Q$ such that
\begin{align}
\label{rho-ball-cond}
\cl B_{\rho(\si)}(\zeta) &\subset E &&\text{and} &
\cl B_{\rho(\si)}(\zeta) \cap \cl \Omega_{t_1}(u; Q) &= \emptyset.
\end{align}
Fix one such $(\zeta, \si)$.
By compactness of $\cl B_{\rho(\si)}(\zeta)$
there exists $\mu > \rho(\si)$ such that
\eqref{rho-ball-cond} holds with $\mu$ instead of $\rho(\si)$.
Set $\chi(t) = K$, and note that
\begin{align*}
\int_{t_1}^\si K \diff s = (\si - t_1) K = \frac{\rho(\si)^2}{2nM}
    < \frac{\mu^2}{2nM}.
\end{align*}
Therefore Proposition~\ref{pr:HS-closing}
yields $(\zeta, \si) \notin \cl\Omega(u;Q)$.
\qedhere\end{proof}

\begin{corollary}
\label{co:rational-contract-bound}
Suppose that positive constants $M$, $\mu$, $\si$,
$A$ and $\e$
satisfy
\begin{align}
\label{sigma-eps-bound}
\si < \e \pth{e^{\frac{\mu^2}{2nM A}} - 1}.
\end{align}
If $u\in \subs(M, Q)$, where $Q = B_\mu(\zeta) \times (0, \si)$,
$\cl\Omega_0(u;Q) = \emptyset$
and
\begin{align}
\label{eq:rationalBound}
u(x,t) \leq \frac{A}{\e + \si - t}
        \quad \text{ in } Q,
\end{align}
then $(\zeta, t) \notin \cl\Omega(u;Q)$ for $t \in (0, \si]$.
\end{corollary}

\begin{proof}
Set $\chi(t) = \frac{A}{\e + \si - t}$ and $t_1 = 0$, $t_2 = \si$.
An easy computation yields,
using the bound \eqref{sigma-eps-bound} in the last step,
\begin{align*}
\int_{t_1}^{t_2} \chi(s) \diff s &= \int_0^\si \frac{A}{\e + \si - s} \diff s
    = - A \ln \frac\e {\e + \si} \\
    &= A \ln \pth{1 + \frac{\sigma}{\e}} < \frac{\mu^2}{2nM}.
\end{align*}
The conclusion follows from Proposition~\ref{pr:HS-closing}.
\qedhere\end{proof}

\section{Nonuniform perturbation}
\label{sec:nonlinear-perturbation}

When the free boundary velocity law in the Hele-Shaw problem
\eqref{HSt} depends on time,
the solutions cannot be easily scaled in time.
In particular, a linear scaling in time preserves
subsolutions or supersolutions only
for short time intervals of length $\frac{m\e}{L}$,
where $m$ and $L$ are constants from \eqref{g-bound}
and $\eqref{g-Lipschitz}$.
In this section we devise
a nonlinear scaling in time that can provide perturbed solutions
on arbitrary long intervals.

\begin{proposition}
\label{pr:nonuniform-perturbation}
Suppose that
$\rho \in C^2(E)$ and $\ta \in C^1(I)$ on some
open sets $E \subset \Rn$ and $I \subset \R$, and $\rho$, $\ta$ satisfy,
for all $x \in E$, $t \in I$,
\begin{compactenum}[(i)]
\item $\ta' > 0$, and either $\rho > 0$ or $\rho \equiv 0$,
\item $\abs{D\rho} \leq 1/2$,
\item\label{rho-deformation-condition}
$\rho \lap \rho \geq (n-1) \abs{D \rho}^2$,
\item \label{rho-ta-perturb-balance}
for some positive constants $a, L, m$,
\begin{align*}
    a \geq
        \ta'(t) (1- \abs{D\rho(x)})^{-2}
        \bra{1 + \frac{L}{m} (\abs{t - \ta(t)} + \rho(x))}.
\end{align*}
\end{compactenum}

Then if $u\in \subs(g, Q)$, a subsolution of the Hele-Shaw problem with
$V_\nu = g(x,t) \abs{Du}$ on an open set $Q \subset E \times I$,
where
$g$ is $L$-Lipschitz and $g \geq m > 0$, then
\begin{align*}
v(x,t) := \sup_{y \in \cl B_{\rho(x)}(x)} a\, u(y, \ta(t)).
\end{align*}
is a subsolution of the same problem on the set
\begin{align*}
\tilde Q = \set{(x,t): t \in \ta(I),\
    B_{\rho(x)}(x) \times \set{\ta^{-1}(t)} \subset Q},
\end{align*}
i.e., $v \in \subs(g, \tilde Q)$.

Similarly,
if we replace condition \eqref{rho-ta-perturb-balance}
with
\begin{compactenum}[(i')]
\setcounter{enumi}{3}
\item \label{rho-ta-perturb-balance'}
for some positive constants a, L, m,
\begin{align*}
    a \leq
        \ta'(t) (1 + \abs{D\rho(x)})^{-2}
        \bra{1 - \frac{L}{m} (\abs{t - \ta(t)} + \rho(x))},
\end{align*}
\end{compactenum}
we have that if  $u\in \supers(g, Q)$ then
\begin{align*}
w(x,t) := \inf_{y \in \cl B_{\rho(x)}(x)} a\, u(y, \ta(t)).
\end{align*}
is a supersolution, $w \in \supers(g, \tilde Q)$.
\end{proposition}

The proof of a similar statement
for time independent Hele-Shaw problem was previously given in,
for instance,
\cite[Lemma~3.40]{K06} with the help of \cite[Lemma~9]{C87}
(see also \cite[proof of Lemma 12]{K07}).
Here we present a slightly more explicit, compact proof
for reader's convenience,
including the new nonlinear deformation in time,
and without the assumption that $u$ is harmonic
or continuous in its positive phase.

\begin{proof}
First, we prove the statement for subsolutions.
We can assume that $u$ is USC,
observing that the sup-convolution commutes with
taking the upper semi-continuous envelope,
and therefore we assume that $v$ is USC.
Suppose that $\phi \in C^{2,1}$
and that $v - \phi$ has a local maximum at $(\hat y, \hat s)$
in $\clset{v > 0}$.
Then, by definition of $v$,
there exists $(\hat x, \hat t) \in \clset{u > 0}$
such that
$\abs{\hat x - \hat y} \leq \rho(\hat y)$,
and $v(\hat y, \hat s) = a u(\hat x, \hat t)$ for
$\hat s = \ta^{-1}(\hat t)$.
We point out here that $\ta$ is invertible
by assumption (i).

Let us set $h := \hat y - \hat x$ for convenience.
We observe that $a u(x, \ta(t)) \leq v(y,t)$ as long as
\begin{align}
\label{inclusion-cond}
    \rho(y) \geq \abs{y - x}.
\end{align}
With this fact in mind,
we shall find a smooth function $y(x)$ so that
$y(\hat x) = \hat y$
and \eqref{inclusion-cond} holds in a neighborhood of $\hat x$, and
therefore $u(x,t) - a^{-1} \phi(y(x),\ta^{-1}(t))$
has a local max at $(\hat x, \hat t)$
in $\clset{u > 0}$.

If $\abs{h} < \rho(\hat y)$, then $y(x) = x + h$ is sufficient
due to the continuity of $\rho$.
A more delicate situation arises when $\abs{h} = \rho(\hat y)$.
We define $y(x)$ in such a way
that the gradients of both sides in \eqref{inclusion-cond}
are equal, while the Hessians are strictly ordered.
The differentiation of $\rho(y(x))$ yields
\begin{align}
\label{differentiation-formulas}
\pd{}{x_i} \rho(y(x)) &= \sum_k D_k \rho \pd{y_k}{x_i}, &
\frac{\partial^2}{\partial x_i \partial x_j} \rho(y(x))
    &= \sum_{k,l} D_{kl} \rho \pd{y_k}{x_i}\pd{y_l}{x_j}
        + \sum_k D_k\rho \frac{\partial^2 y_k}{\partial x_i \partial x_j}.
\end{align}

Let us choose a coordinate system $(e_1, \ldots, e_n)$
so that $e_n = h/\abs{h}$  and $D\rho(\hat y) = \al e_1 + \be e_n$,
and set $\gamma^2 = (1 - \be)^2 + \al^2$.
Let us recall that $\abs{D \rho} \leq 1/2$ and therefore $\gamma \geq 1/2$.
We define $y(x)$ in the form
\begin{align}
\label{def-of-y-x}
y(x) = \hat y + A(x - \hat x) + e_n \bra{\ov2 B(x - \hat x)\cdot (x - \hat x)},
\end{align}
where $B$ is a symmetric matrix with $\trace B \geq 0$
to be specified later
and $A$ is the rotation and scaling matrix
\begin{align*}
A = \frac{1- \be}{\ga^2} e_1 \otimes e_1
    + \frac{1- \be}{\ga^2} e_n \otimes e_n
    + \frac{\alpha}{\ga^2} e_n \otimes e_1
    - \frac{\alpha}{\ga^2} e_1 \otimes e_n
       + \ov\ga\sum_{k=2}^{n-1} e_k \otimes e_k.
\end{align*}
Using the formulas \eqref{differentiation-formulas} and \eqref{def-of-y-x},
we can express the derivatives at $x = \hat x$ as
\begin{align}
\label{diff-rho-y-x}
\at{D(\rho(y(x)))}{x = \hat x} &= A^t [D\rho(\hat y)],
    &
\at{D^2(\rho(y(x)))}{x = \hat x} &= A^t [D^2 \rho(\hat y)] A + \be B,
\end{align}
and similarly
\begin{align*}
\at{D(\abs{y(x) - x})}{x = \hat x}
    &= (A - I)^t e_n,
   \\
\at{D^2(\abs{y(x) - x})}{x = \hat x} &=
    \frac{1}{\abs{h}}(A - I)^t (I - e_n \otimes e_n) (A - I) + B.
\end{align*}
A quick computation shows that the choice of $A$ guarantees that
$A^t [D\rho(\hat y)] = (A - I)^t e_n$,
with the equality of gradients in \eqref{inclusion-cond} as a consequence.
To get the strict ordering of Hessians
$\at{D^2(\rho(y(x)))}{x = \hat x} > \at{D^2(\abs{y(x) - x})}{x = \hat x}$,
we need to find a symmetric matrix $B$ with $\trace B \geq 0$ such that
\begin{align*}
M := A^t [D^2 \rho(\hat y)] A
    - \frac1{\abs{h}}(A - I)^t (I - e_n \otimes e_n) (A - I)
 > (1 - \be) B.
\end{align*}
Since $\abs{\be} < 1$,
this will be possible if and only if the trace of the left-hand side
$M$
is strictly greater than 0.
We observe that $A A^t = \ga^{-2} I$ and
$(A - I) (A - I)^t = (\ga^{-2} + 1) I - A - A^t$,
which is a diagonal matrix,
and therefore we can evaluate
\begin{align*}
\trace M &= \ga^{-2} \Delta \rho(\hat y) -
    \frac{1}{\abs{h}}\bra{(n-1) (\ga^{-2} + 1)
    - 2 (n-2) \ga^{-1}  - 2\ga^{-2}(1 - \be)}\\
    &= \ga^{-2}
        \bra{\Delta \rho(\hat y)
            - \frac{1}{\abs{h}}\bra{(n-2) (1 - \ga)^2 + \al^2 + \be^2}}.
\end{align*}
Recall that $\abs{h} = \rho(\hat y)$, and
we can also estimate
$(n-2) (1 - \ga)^2 + \al^2 + \be^2 \leq (n-2) \be^2 +\al^2 + \be^2
\leq  (n-1)\abs{D\rho(\hat y)}^2$.
This implies that $\trace M > 0$ due to the assumption
\eqref{rho-deformation-condition} on $\rho$.

The above computation proves that
$u - \tilde \phi$ has a local max at $(\hat x, \hat t)$
in $\clset{u > 0}$,
where $\tilde \phi(x, t) := a^{-1}\phi(y(x), \ta^{-1}(t))$.
As always, two cases must be addressed:
\begin{compactenum}
\item
$u(\hat x,\hat t) > 0$:
We first observe that $D\phi(\hat y, \hat s) \cdot e_n \leq 0$.
Indeed, since $\abs{D \rho} < 1$, we see that
$v(\hat y, \hat s) = au(\hat x, \hat t) \leq v(\hat y - \si e_n, \hat s)$
for all $\si \geq 0$ small.
Since $u$ is a subsolution of the Hele-Shaw problem,
\eqref{differentiation-formulas}
and $\trace B \geq 0$ yield the inequality
\begin{align*}
0 \geq - \Delta \tilde \phi(\hat x, \hat t)
    = - \trace A^t \bra{D^2 \phi(\hat y, \hat s)} A
        - \bra{D \phi(\hat y, \hat s) \cdot e_n} \trace B
        \geq - \ga^{-2} \Delta \phi(\hat y, \hat s).
\end{align*}

\item
$u(\hat x, \hat t) = 0$:
First, we set $\xi := a^{-1} \ga^{-2} \ta'(\hat s)
\leq a^{-1} \ta'(\hat s) \pth{1 - \abs{D \rho(\hat y)}}^{-2}$
and, using the assumption
\eqref{rho-ta-perturb-balance}
and the assumptions on $g$, we estimate
\begin{align}
\label{g-lipschitz-estimate}
\xi g(\hat x, \hat t) &\leq
       \xi g(\hat y, \hat s) \pth{ 1+
       \frac Lm (\abs{\hat s - \ta(\hat s)} +  \abs{h})}
       \leq g(\hat y, \hat s),
\end{align}
recalling that $\abs{h} \leq \rho(\hat y)$.
Additionally, differentiation of $\tilde \phi$ in time yields
\begin{align*}
\tilde \phi_t(x,t) =
\pd{}t \phi(y(x), \ta^{-1}(t)) = \phi_t(y(x), \ta^{-1}(t))
\frac1{\ta'(\ta^{-1}(t))}.
\end{align*}
The formula \eqref{diff-rho-y-x} for the gradient,
recalling that $A$ is a rotation and scaling by a factor $\ga^{-1}$,
as well as the estimate \eqref{g-lipschitz-estimate}
allow us to transform the viscosity inequality and obtain
\begin{align*}
0 &\geq \tilde \phi_t(\hat x, \hat t)
    - g(\hat x, \hat t) \abs{D \tilde \phi(\hat x, \hat t)}^2 \\
    &= a^{-1} [\ta'(\hat s)]^{-1}\phi_t(\hat y, \hat s)
        - a^{-2} \ga^{-2} g(\hat x, \hat t)
            \abs{D \phi(\hat y, \hat s)}^2\\
    &= a^{-1} [\ta'(\hat s)]^{-1} \bra{\phi_t(\hat y, \hat s)
        - \xi g(\hat x, \hat t)\abs{D \phi(\hat y, \hat s)}^2}\\
    &\geq a^{-1} [\ta'(\hat s)]^{-1} \bra{\phi_t(\hat y, \hat s)
        - g(\hat y, \hat s)\abs{D \phi(\hat y, \hat s)}^2}.
\end{align*}
\end{compactenum}
Hence $v$ is a viscosity subsolution of the HS problem
and the proof is finished.

\bigskip
The proof for supersolutions follows same idea,
but, in contrast to the proof for subsolutions,
we consider a strict local minimum of $w - \phi$ at $(\hat y, \hat s)$.
It is also necessary to make the obvious changes
in the directions of the inequality signs.
Finally, the estimate \eqref{g-lipschitz-estimate}
is replaced by
\begin{align*}
\xi g(\hat x, \hat t) &\geq
    \xi g(\hat y, \hat s) \pth{1
       - \frac Lm (\abs{\hat s - \ta(\hat s)} +  \abs{h})}
       \geq g(\hat y, \hat s),
\end{align*}
which applies due to the assumption (\ref{rho-ta-perturb-balance'}'),
since now we estimate
$\ga^{-2} \geq (1 + \abs{D\rho(\hat y)})^{-2}$.
The rest is analogous.
\qedhere\end{proof}

\subsection{Subsolutions}
\label{sec:nonlinear-scaling-sub}

In this section, we will give an expression
for a particular choice of $\ta$
so that Proposition~\ref{pr:nonuniform-perturbation}
applies for subsolutions,
given $a, m, L$ and $\rho$.
The notation throughout this section
is consistent with the statement of
Proposition~\ref{pr:nonuniform-perturbation}.

Let us assume that $\rho$ and $D\rho$ are bounded,
and define
\begin{align*}
\al := a \min_E (1 - \abs{D\rho})^2,\qquad
\ga := \frac m L, \qquad \la := \abs\tau + \max_E \rho.
\end{align*}
We will look for $\ta$ of the form $\ta(t) = f(t;\al,\ga,\la) + \tau$,
for some constant $\tau \neq 0$ and
$f \in C^{1}([0, \infty))$, with $f(0) = 0$ and $0 < f' < 1$.

The upper bound on the derivative yields $f(t) \leq t$ for all $t \geq 0$,
hence the estimate $\abs{t - \ta(t)} \leq t - f(t) + \abs\tau$ holds.
The condition \eqref{rho-ta-perturb-balance}
in Proposition~\ref{pr:nonuniform-perturbation}
can be rewritten, after multiplication by $\ga$, as
\begin{align}
\label{rho-f-balance}
\al \ga \geq f'(t)(\ga + \la + t - f(t))
    \geq \ta'(t) (\ga + \rho + \abs{t - \ta(t)}).
\end{align}
Let us now assume that $\al \gamma < \ga + \la$.
In fact, if this condition is violated,
we can set $f(t) = t$ and the condition \eqref{rho-ta-perturb-balance}
is satisfied.

To find $f$ satisfying \eqref{rho-f-balance},
we substitute $h(t) = t - f(t) + \ga + \la$,
and after a short calculation, using $f' = 1 - h'$,
$h$ must satisfy
\begin{align}
\label{h-inequality}
\frac{h}{h-\al \ga} h' \geq 1.
\end{align}
With the initial data $h(0) = \ga + \la$,
the solution of the equality can be expressed
in terms of the principal branch
of the Lambert W function,
\begin{align*}
h(t) = \al \ga \pth{1+W\pth{\frac{\xi}{\al\ga} e^{\frac{t + \xi}{\al\ga}}}},
\end{align*}
where $\xi = \ga + \la - \al \ga$, positive by assumption.

Thus the function $f$ can be expressed in terms of $h$
as
\begin{align*}
f(t;\al,\ga,\la) =
t + \xi - \al\ga W\pth{\frac{\xi}{\al\ga} e^{\frac{t + \xi}{\al\ga}}},
\qquad \xi = \ga + \la - \al \ga.
\end{align*}

\begin{lemma}
\label{le:properties-of-f}
The function $f(t;\al,\ga,\la)$, for $\al, \ga, \la$ positive
such that $\ga + \la - \al \ga > 0$, has the following properties:
\begin{compactenum}
\item $f(0) = 0$,
\item $f(t) \to t$ locally uniformly
on $[0,\infty)$ as $(\al, \la) \to (1,0)$,
\item
    $f' = \frac{\al \ga}{h}$,
    $f'' < 0$,
    $f'(0) = \frac{\al\ga}{\ga + \la}$.
\end{compactenum}
\end{lemma}
\begin{proof}
For (a), use the formula $W(xe^x) = x$.
To show (b), note that $f(t) \to t$
locally uniformly on $[0, \infty)$
as $(\al,\la) \to (1,0)$
is equivalent to the locally uniform convergence $h(t) \to \ga$.
Since $W(0) = 0$ and $W$ is continuous at $0$,
the result follows from the fact that the argument of $W$
goes to $0$ locally uniformly in $t$ as $(\al, \la) \to (1,0)$.
(c) is a simple consequence of the properties of Lambert W function.
\qedhere\end{proof}

\subsection{Supersolutions}
\label{sec:nonlinear-scaling-super}

We follow the idea from \S\ref{sec:nonlinear-scaling-sub}
to construct a function $\ta(t) = f(t) + \tau$
such that condition (\ref{rho-ta-perturb-balance'}')
in Proposition~\ref{pr:nonuniform-perturbation} holds.

We define $\al = a \max(1 + \abs{D \rho})^2$ now
with $\ga$, $\la$ as before,
and look for $f \in C^1([0, T])$, for $T > 0$, $f(0) = 0$ and
$f' > 1$.
The condition on the derivative yields $f(t) \geq t$ for $t \geq 0$
and thus $\abs{t - \ta(t)} \geq f(t) - t - \abs\tau$.
The condition (\ref{rho-ta-perturb-balance'}') can be rewritten
as in \eqref{rho-f-balance},
\begin{align*}
\al \ga \leq f'(t) (\ga - \la + t - f(t))
    \leq f'(t) (\ga - \rho - \abs{t - \ta(t)}).
\end{align*}
Here it is necessary to assume $\ga > \la$.
Furthermore, we assume that $\al \ga > \ga - \la$,
otherwise $f(t) = t$ satisfies the inequality and we are done.

After a substitution $h(t) = \ga - \la + t - f(t)$,
we discover that $h$ satisfies \eqref{h-inequality}.
This time, however,
$h$ is decreasing, and will become zero in finite time.
Therefore the inequality does not have a global solution.
With initial data $h(0) = \ga - \la$,
the solution can be written as
\begin{align*}
h(t) = \al \ga \pth{1+W\pth{\frac{\eta}{\al\ga}
            e^{\frac{t + \eta}{\al\ga}}}},
\end{align*}
where $\eta = \ga - \la - \al \ga$, negative
by assumption.
The singularity occurs when the argument of $W$
becomes $-\ov e$.

The function $f$ has now the form
\begin{align*}
f(t; \al,\ga,\la) = t + \eta - \al\ga W\pth{\frac{\eta}{\al\ga}
            e^{\frac{t + \eta}{\al\ga}}},
            \qquad \eta = \ga - \la - \al\ga.
\end{align*}

\begin{lemma}
\label{le:properties-of-f'}
For every $T > 0$,
there exists $\al_T > 1$ and $\la    _T \in (0, \ga)$ such that
the function $f$ is well defined on $[0, T]$
for $\al \in (0, \al_T)$
and $\la \in (0, \la_T)$
with $\ga - \la - \al \ga < 0$,
and has the following properties:
\begin{compactenum}
\item $f(0) = 0$,
\item $f(t) \to t$ uniformly on $[0,T]$ as $\al \to 1+$ and $\la \to 0+$,
\item $f' = \frac{\al\ga}{h}$, $f'' > 0$,
    $f'(0) = \frac{\al\ga}{\ga - \la}$.
\end{compactenum}
\end{lemma}

\begin{proof}
The existence of $\al_T$ and $\la_T$ follows after
realizing that, as $(\al, \la) \to (1,0)$,
$\frac{\eta}{\al\ga}e^{\frac{\eta}{\al\ga}} \to 0$
while $e^{\frac{t}{\al\ga}}$ stays uniformly bounded on $t \in [0, T]$.
Consequently, it is possible to choose $\al_T > 1$ and $\la_T \in (0, \ga)$
that guarantee that the argument of $W$ is larger than
$- \ov e$ for $t \in [0, T]$.
As in Lemma~\ref{le:properties-of-f},
the above observation also implies (b).
\qedhere\end{proof}

\section{Technical results}

\subsection{Harmonic function on a thin cylinder}
The idea of the following lemma is motivated by the observation that the values of a harmonic function on a thin cylinder of a larger radius
are not strongly influenced by its values on the side of the cylinder.

\begin{lemma}[Subharmonic function on a thin cylinder]
\label{le:subharmonicOnThinDomain}
There is a dimensional constant $c = c(n) > 0$ such that whenever $u$ is a subharmonic function on $\Omega$, $\Omega \subset C$ open, where
\begin{align*}
C = \set{x = (x', x_n) : \abs{x'} < R,\ 0 < x_n < \de}
\end{align*}
for some $R > 0$ and $\de > 0$, and furthermore for some $K > 0$
\begin{align*}
u&\leq K && \text{on $\cl C$},\\
u &\leq -1 && \text{on $\partial \Omega \cap \set{\abs{x'} < R}$},\\
\end{align*}
then
\begin{align*}
u \leq -\frac{1}{2} \qquad \text{on $\set{x \in \Omega: \abs{x'} \leq R', 0 < x_n < \de}$},
\end{align*}
where
\begin{align*}
R' = R - c (K + 2) \de.
\end{align*}
\end{lemma}

\begin{proof}
We shall compare the function $u$ on $\cl \Omega$ with a translate of the superharmonic function
\begin{align*}
\vp(x) = \sqrt{1 + \frac{\abs{x'}^2}{n}} \cdot \cos x_n - \frac{3}{2}.
\end{align*}
Indeed, writing $\abs{x'} = r$, a straightforward computation yields
\begin{align*}
\Delta \vp = \pth{\partial_{x_n x_n} + \partial_{rr} + \frac{n-2}{r} \partial_r} \vp = -\frac{\left(n+2 r^2+n r^2+r^4\right) \cos x_n}{n^2 \left(\frac{n+r^2}{n}\right)^{3/2}}.
\end{align*}
This quantity is negative if $\abs{x_n} < \pi/2$.

Define $\tilde \vp(x) = \vp(\frac{\pi}{3 \de} x)$. Observe that for $0 \leq x_n \leq \de$
\begin{enumerate}
\item $-\Delta \tilde \vp > 0$,
\item $\tilde \vp \geq -1$,
\item $\tilde \vp \geq K$ for $\abs{x'} \geq \frac{6 \sqrt{n}}{\pi} (K + 2) \de$. This follows from estimating $\sqrt{1 + \frac{\abs{x'}^2}{n}} \geq \frac{\abs{x'}}{\sqrt n}$ and $\cos x_n \geq \cos \frac{\pi }{3\de} \de = \cos \frac{\pi}{3} = \frac{1}{2}$.
\end{enumerate}

So if we set $c := \frac{6 \sqrt n}{\pi}$, we can compare a translate $\tilde \vp(\cdot - x_0)$, $x_0 = (x'_0, 0)$ with $u$ on $\cl C$  whenever $\abs{x'_0} \leq R -c (K + 2) \de$ and the maximum principle lets us conclude that $u \leq \vp(\cdot - x_0)$ on $\cl C$. That shows that
\begin{align*}
u &\leq -\frac{1}{2} \qquad \text{on } \set{x : \abs{x'} \leq R - c (K+2) \de, \ 0 \leq x_n \leq \de}.
\end{align*}
\qedhere\end{proof}

\subsection{Integer approximation}
In this section we consider the approximation of
a set $A \subset \R^d$ by the set $A \cap \Z^d$.
First, we introduce some useful notation.

\begin{definition}
\label{def:grid}
For a set $A \subset \R^d$, $d \geq 1$, we shall denote
\begin{align*}
A^\grid = A \cap \Z^d, \qquad A^\gride = A \cap \e \Z^d.
\end{align*}
\end{definition}

The following two lemmas provide sufficient conditions
for approximation of $A$ by $A^\grid$.

\begin{lemma}
\label{le:grid-cover}
Suppose that $A, E \subset \R^d$
and $\la > \frac12\sqrt{d}$.
If $\e > 0$ and
\begin{align}
\label{ball-extension-inclusion}
E + B_{\la\e}(0) \subset A
\end{align}
then
\begin{align*}
E \subset A^{\grid\e} + B_{\la\e}(0).
\end{align*}
\end{lemma}

\begin{proof}
Let $\xi \in E$.
Then, by \eqref{ball-extension-inclusion},
$\xi + B_{\la\e}(0) = B_{\la\e}(\xi) \subset A$.
The choice of $\la$ guarantees that
there exists
$\zeta \in B_{\la\e}(\xi) \cap \e\Z^d \neq \emptyset$.
Clearly $\xi \in B_{\la\e}(\zeta) \subset A^{\grid\e} + B_{\la\e}(0)$
and the proof is finished.
\qedhere\end{proof}

\begin{lemma}
\label{le:grid-approx}
Let $E \subset \R^d$, $\e > 0$ and
$\la > \frac12 \sqrt d$.
Define $K := \set{x \in E: B_{\la\e}(x) \subset E}$.
If $\mu := \dist(E, K) < \infty$
then
\begin{align*}
(E + x) \subset (E+x)^{\grid\e} + B_{\la\e + \mu}(0)
\qquad \text{for any $x \in \R^d$.}
\end{align*}
\end{lemma}

\begin{proof}
Choose $\de > 0$ so that $\la - \de > \frac12 \sqrt d$.
Since $K + x + B_{(\la - \de)\e}(0) \subset E + x$
by definition of $K$,
we have by definition of $\mu$ and Lemma~\ref{le:grid-cover}
\begin{align*}
E+x \subset K + x + B_{\mu + \de\e}(0)
&\subset (E + x)^{\grid\e} + B_{(\la - \de)\e}(0) + B_{\mu + \de\e}(0)
\\&= (E+x)^\grid + B_{\la\e + \mu}(0),
\end{align*}
and that is exactly what we wanted to prove.
\qedhere\end{proof}

\begin{remark}
\label{rem:grid-approx-cone}
Note that if $E = \cone_{\nu,\ta}(0)$
for some $\nu \in \Rn \setminus\set0$ and $\ta \in (0,\frac\pi2)$
then $\mu = \frac{\la}{\sin \ta} \e$
in Lemma~\ref{le:grid-approx}.
\end{remark}

% CONTENTS END

\parahead{Acknowledgments}
% acknowledgments begin
I want to thank Inwon Kim for many helpful discussions.
I also want to thank the anonymous reviewers for careful reading of the manuscript and
their valuable comments.
Much of the work was done while the author was a project researcher of
the University of Tokyo during 2011--2013, supported by a grant (Kiban
S, No.21224001) from the Japan Society for the Promotion of Science.
% acknowledgments end

\end{document}